\documentclass[11pt]{amsart}
\usepackage{amsmath, amsfonts, amsthm, amssymb}
\usepackage{graphicx}
\usepackage{float}
\usepackage{verbatim}
\usepackage{color}

\allowdisplaybreaks

\usepackage{multicol, mathtools, dsfont, verbatim,hyperref}
\usepackage{hyperref}
\usepackage{scalerel,stackengine, subcaption}
\usepackage[usenames,dvipsnames]{xcolor}
\usepackage{enumitem}
\usepackage{pgfplots}
\usepgfplotslibrary{fillbetween}
\pgfplotsset{width=10cm,compat=1.9}
\definecolor{lightblue}{HTML}{2B77A4}
\definecolor{darkred}{HTML}{9E0D0D}
\hypersetup{
	colorlinks=true,
	linkcolor=darkred,
	urlcolor=darkred,
	citecolor=lightblue
}
\newcommand{\blue}{\color{blue}}

\usepackage[square,sort,comma,numbers]{natbib}
\usepackage{fancyhdr}
 
\numberwithin{equation}{section}

\newtheorem{thm}{Theorem}[section]
\newtheorem{lem}[thm]{Lemma}
\newtheorem{cor}[thm]{Corollary}
\newtheorem{prop}[thm]{Proposition}

\theoremstyle{definition}
\newtheorem{defn}[thm]{Definition}

\newtheorem{example}[thm]{Example}
\newtheorem{rem}[thm]{Remark}
\newtheorem{ques}[thm]{Question}

\newcommand{\eps}{\epsilon}

\newcommand{\R}{\mathbb{R}}
\newcommand{\N}{\mathbb{N}}
\newcommand{\Z}{\mathbb{Z}}
\newcommand{\C}{\mathbb{C}}

\DeclareMathOperator{\diam}{diam}
\DeclareMathOperator{\dist}{dist}

\DeclareMathOperator{\Vol}{Vol}
\DeclareMathOperator{\Id}{Id}

\renewcommand{\ge}{\geqslant}
\renewcommand{\le}{\leqslant}
\renewcommand{\geq}{\geqslant}
\renewcommand{\leq}{\leqslant}

\newcommand{\ubd}{\overline{\dim}_{\textup{B}}}

\newcommand{\cA}{{\mathcal A}}

\newcommand{\cC}{\mathcal{C}}
\newcommand{\cD}{\mathcal{D}}
\newcommand{\cE}{\mathcal{E}}
\newcommand{\cF}{\mathcal{F}}
\newcommand{\cH}{\mathcal{H}}
\newcommand{\cM}{\mathcal{M}}
\newcommand{\cU}{\mathcal{U}}

\newcommand{\bone}{\mathbb{I}}
\newcommand{\be}{\mathbf{e}}

\newcommand{\cI}{{\mathcal{I}}}
\newcommand{\bE}{\mathbb{E}}

\newcommand{\loc}{{\scriptstyle{loc}}}
\newcommand{\reg}{{\scriptstyle{reg}}}
\newcommand{\Sob}{{\scriptstyle{Sob}}}
\newcommand{\RH}{{\scriptstyle{RH}}}
\newcommand{\submax}{{\scriptstyle{max}}}

\title[Distortion of intermediate dimension and conformal dimension]{Sobolev and quasiconformal distortion of intermediate dimension with applications to conformal dimension}
\author{Jonathan M. Fraser}
\address{University of St Andrews \\ Mathematical Institute \\ St Andrews KY16 9SS, Scotland \\ \texttt{jmf32@st-andrews.ac.uk}}
\author{Jeremy T. Tyson}
\address{University of Illinois at Urbana-Champaign \\ Department of Mathematics \\ 1409 West Green Street \\ Urbana, IL 61801 \\ \texttt{tyson@illinois.edu}}
\thanks{JMF was financially supported by a {\em Leverhulme Trust Research Project Grant} (RPG-2023-281),  an {\em EPSRC
Standard Grant}  (EP/Y029550/1),  and an \emph{EPSRC Open Fellowship} (EP/Z533440/1). JTT acknowledges support from the {\em Simons Foundation} (grant \#852888). Research for this paper was conducted whilst JTT was a visitor in the School of Mathematics and Statistics of the University of St Andrews in Fall 2024. He wishes to thank the institute for its hospitality during this visit.}

\date{\today}

\begin{document}

\maketitle
\thispagestyle{empty}

\begin{center}\emph{Dedicated to Bob Kaufman, 1942--2025.}\end{center}

\begin{abstract}
We study the distortion of intermediate dimension under supercritical Sobolev mappings and also under quasiconformal or quasisymmetric homeomorphisms. In particular, we extend to the setting of intermediate dimensions both the Gehring--V\"ais\"al\"a theorem on dilatation-dependent quasiconformal distortion of dimension and Kovalev's theorem on the nonexistence of metric spaces with conformal dimension strictly between zero and one. Applications include new contributions to the quasiconformal classification of Euclidean sets and a new sufficient condition for the vanishing of conformal box-counting dimension. We illustrate our conclusions with specific consequences for Bedford--McMullen carpets, samples of Mandelbrot percolation, and product sets containing a polynomially convergent sequence factor.
\end{abstract}

\tableofcontents


\section{Introduction and  main results}\label{sec:introduction}

Among numerous applications, metrically defined notions of dimension are commonly used in the service of mapping theory problems. Invariance or quasi-invariance of such dimensions enables one to distinguish spaces up to various types of mapping equivalence. For example, many standard definitions of dimension in fractal geometry are bi-Lipschitz invariant, and hence can be used to show that certain pairs of spaces are not bi-Lipschitz equivalent. Similarly, effective bounds for the distortion of Hausdorff dimension under quasiconformal mappings have been used to distinguish spaces up to quasiconformal equivalence, whilst the related concept of conformal dimension has been shown to be relevant for a wide variety of problems in quasisymmetric uniformisation, geometric group theory, and dynamics.

A pair of one-parameter families of dimension functions feature prominently in this paper: the {\it intermediate dimensions}, a one-parameter family of dimensions interpolating between Hausdorff and (upper) box-counting dimension, and the {\it Assouad spectrum}, a one-parameter family interpolating between box-counting dimension and Assouad dimension. Although our primary emphasis is on the intermediate dimensions, and especially on their distortion properties under Sobolev and quasiconformal mappings, some conclusions which we obtain are phrased in terms of the Assouad spectrum. We give sharp bounds on the increase of intermediate dimension under supercritical Sobolev mappings of Euclidean domains, and in particular we give two-sided dilatation-dependent bounds for the change of intermediate dimension under quasiconformal mappings. These results extend prior work of Gehring--V\"ais\"al\"a and Kaufman, who treated the Hausdorff and box-counting cases. We also generalise Kovalev's theorem on the nonexistence of metric spaces with conformal dimension strictly between zero and one. These results have concrete applications to the quasiconformal classification of Euclidean sets and also to the theory of conformal dimension. In particular, we provide a new sufficient condition, phrased in terms of the intermediate dimensions of a metric space $(X,d)$, which guarantees that $X$ has vanishing conformal Assouad spectrum.

\subsection{Notions of metric dimension}

We now introduce some of the major players in our story. Let $(X,d)$ be a totally bounded metric space. For $r>0$, the covering number $N(X,r)$ denotes the minimal number of sets of diameter $r$ needed to cover $X$. The {\it (upper) box-counting dimension} of $(X,d)$ is then defined to be
\begin{equation}\label{eq:upper-box-dimension}
\overline\dim_B X := \limsup_{r\to 0} \frac{\log N(X,r)}{-\log r}.
\end{equation}
Equivalently, $\overline\dim_B X$ is the infimum of those values $s>0$ so that the growth rate of $N(X,r)$ as $r \to 0$ is $O(r^{-s})$. Denoting by $\dim_A X$ the Assouad dimension of $X$ (see section \ref{subsec:metric} for the definition), we recall that
\begin{equation}\label{eq:HBA}
\dim_H X \le \overline\dim_B X \le \dim_A X
\end{equation}
for all totally bounded metric spaces $(X,d)$. 

For $0<\theta \le 1$, the {\em (upper) $\theta$-intermediate dimension $\overline\dim_\theta X$} is defined as the critical value $s>0$ associated to the limsup as  $\max_i \diam A_i \to 0$ of the infimal value of the usual diameter sums $\sum_i (\diam A_i)^s$ for coverings $\{A_i:i \in \N\}$ of $X$, but with the added restriction that only coverings $\{A_i\}$ for which 
\begin{equation}\label{eq:int-dim-control}
\delta^{1/\theta} \le \diam(A_i) \le \delta \quad \forall \, i \in \N,
\end{equation}
for some $\delta>0$ are considered. See section \ref{sec:intermediate} for the precise definition and further comments. When $\theta = 1$, the restriction is to coverings by sets with constant diameter and the diameter sum described above reduces to $N(X,\delta)\delta^s$; in this case, the procedure so described recovers the box-counting dimension $\overline\dim_B X$. 

In this paper, we will consistently employ upper intermediate and box-counting dimensions, so henceforth we omit the use of the adjective `upper' in the terminology. However, we continue to denote these notions of dimension by $\overline\dim_\theta$ and $\overline\dim_B$.

Similarly, for $0\le \theta <1$, the {\em Assouad  spectrum}\footnote{In fact the quantity we define here is a variant of the notion of Assouad spectrum defined in \cite[Section 3.3]{fr:book}. It was termed {\it upper Assouad spectrum} in \cite{fr:book} and {\it regularized Assouad spectrum} in \cite{ct:qc-assouad-spectrum}. We suppress these descriptors  here, simply referring to the Assouad spectrum.} $\dim_A^\theta X$ is defined as the infimum of those values $s>0$ so that the covering numbers $N(B(x,R),r)$ have growth rate bounded above by $O((R/r)^s)$, uniformly for $x \in X$ and $R>0$, under the restriction that $r^\theta \le R$. Again, when $\theta=0$ the ancillary restriction on $r$ disappears and we recover the box-counting dimension $\dim_B X$ upon fixing $R = \diam(X)$. In the limit as $\theta \nearrow 1$ one arrives in principle to the standard notion of Assouad dimension $\dim_A X$. However, in general $\lim_{\theta \to 1} \dim_\theta X$ may not equal $\dim_A X$ and one defines the {\em quasi-Assouad dimension}
\begin{equation}\label{eq:qa}
\dim_{qA} X := \lim_{\theta \nearrow 1} \dim_A^\theta X.
\end{equation}
Similarly, in the limit as $\theta \searrow 0$ the lower bound in \eqref{eq:int-dim-control} disappears and, philosophically, one arrives to the Hausdorff dimension $\dim_H X$. However, in general $\lim_{\theta \to 0} \dim_\theta X$ may not equal $\dim_H X$. By analogy with \eqref{eq:qa} we define the {\em (upper) quasi-Hausdorff dimension}\footnote{We thank Alex Rutar for suggesting this terminology.} of a metric space to be
\begin{equation}\label{eq:qHdim}
\dim_{qH} X := \lim_{\theta \searrow 0} \dim_\theta X.
\end{equation}
 For all choices of $0<\theta<1$ we have
\begin{equation}\label{eq:HthetaBthetaA}
\dim_H X \le \overline\dim_\theta X \le \overline\dim_B X \le \dim_A^\theta X \le \dim_A X,
\end{equation}
moreover,
$$
\dim_H X \le \dim_{qH} X \stackrel{\theta \to 0}{\longleftarrow} \overline\dim_\theta X \stackrel{\theta \to 1}{\longrightarrow} \overline\dim_B X
$$
and
$$
\overline\dim_B X\stackrel{\theta \to 0}{\longleftarrow} \dim_{A}^\theta X \stackrel{\theta \to 1}{\longrightarrow} \dim_{qA} X \le \dim_A X.
$$
Understanding when
\begin{equation}\label{eq:HqH}
\dim_H X = \dim_{qH} X
\end{equation}
for a given metric space $(X,d)$ is a subtle question which has already been shown to be useful for several questions in fractal geometry, dynamics, and stochastics. One may reasonably define $\overline\dim_0 X = \dim_H X$ and then \eqref{eq:HqH} is equivalent to asking if the intermediate dimensions are continuous at $\theta=0$.  Examples of sets for which \eqref{eq:HqH} is satisfied include self-affine carpets of Bedford--McMullen type and polynomially convergent sequences as in \eqref{eq:Ep}, see \cite{ffk:intermediate}. See \cite{bff:intermediate-projections} for consequences of \eqref{eq:HqH} for the box-counting dimensions of orthogonal projections, \cite{bur:fractional-Brownian} for applications to the dimensions of images under fractional Brownian motion, and \cite{am:preprint} for further results on the behaviour of quasi-Hausdorff dimension under projections. In this paper, we show that \eqref{eq:HqH} is also relevant in the study of conformal dimension.

\subsection{Conformal dimension}

The conformal dimension of a metric space $(X,d)$ is the infimum of dimensions of metric spaces quasisymmetrically equivalent to $X$. We recall that metric spaces $(X,d)$ and $(Y,d')$ are {\it quasisymmetrically equivalent} if there exists a homeomorphism $f:X \to Y$ and an increasing homeomorphism $\eta:[0,\infty) \to [0,\infty)$ so that
$$
\frac{d'(f(x),f(a))}{d'(f(x),f(b))} \le \eta \left( \frac{d(x,a)}{d(x,b)} \right) \qquad \forall \, x,a,b \in X, x \ne b.
$$
Such a map $f$ is called a \emph{quasisymmetric homeomorphism} and we write $X\stackrel{qs}{\simeq} Y$ if there exists a quasisymmetric homeomorphism from $X$ to $Y$. Note that the inverse of a quasisymmetric homeomorphism is again quasisymmetric and so $\stackrel{qs}{\simeq}$ is an equivalence relation on the class of metric spaces.

A version of conformal dimension can be defined for any naturally occurring metric dimension. The original concept, due to Pansu \cite{P}, was {\em conformal Hausdorff dimension}
\begin{equation}\label{eq:c-dim-h}
C\dim_H X := \inf \{ \dim_H Y : X \stackrel{qs}{\simeq} Y \},
\end{equation}
but {\em conformal Assouad dimension}
\begin{equation}\label{eq:c-dim-a}
C\dim_A X := \inf \{ \dim_A Y : X \stackrel{qs}{\simeq} Y \}
\end{equation}
has also received significant attention in the literature. For uniformly perfect spaces, conformal Assouad dimension $C\dim_A X$ agrees with the {\em Ahlfors regular conformal dimension}, and under the latter name features extensively in applications to uniformisation, dynamics on non-smooth spaces, and analysis on fractals. Various authors have studied the question of when a given metric space is minimal for conformal dimension, and when conformal dimension is attained (i.e., when the infimum in \eqref{eq:c-dim-h} or \eqref{eq:c-dim-a} is a minimum). For instance, Bonk and Kleiner \cite{bonk-kleiner} showed that attainment of $C\dim_A X$, in the case where $X$ is a topological $2$-sphere arising as the boundary at infinity of a Gromov hyperbolic group $G$, suffices to ensure that $G$ is a Kleinian group. Keith and Laakso \cite{keith-laakso} gave necessary and sufficient conditions for an Ahlfors regular metric space to attain its conformal Assouad dimension. More recently, Rossi and Suomala \cite{rossi-suomala} have shown that samples of fractal percolation are almost surely not minimal for conformal Hausdorff dimension.  Other references include \cite{bonk:icm}, \cite{mackay:random-groups}, \cite{bourdon-pajot:l-p-cohomology}, \cite{bourdon-kleiner:clp}, and \cite{mur:combinatorial}.

Interposing between the preceding notions is {\it conformal box-counting dimension}
\begin{equation}\label{eq:c-dim-b}
C\overline\dim_B X :=  \inf \{ \overline\dim_B Y : X \stackrel{qs}{\simeq} Y \}.
\end{equation}
Similarly, we may define the {\it conformal $\theta$-intermediate dimension}
\begin{equation}\label{eq:c-dim-intermediate-theta}
C\overline\dim_\theta X :=  \inf \{ \overline\dim_\theta Y : X \stackrel{qs}{\simeq} Y \}.
\end{equation}
and the {\it conformal Assouad spectrum}
\begin{equation}\label{eq:c-dim-assouad-spectrum-theta}
C\dim_A^\theta X :=  \inf \{ \dim_A^\theta Y : X \stackrel{qs}{\simeq} Y \}.
\end{equation}
In view of \eqref{eq:HthetaBthetaA} we have
$$
C\dim_H X \le C\overline\dim_\theta X \le C\overline\dim_B X \le C\dim_A^\theta X \le C\dim_A X
$$
for all $0<\theta<1$. In Theorem \ref{th:main-1} we prove a distortion estimate for conformal intermediate dimension $C\overline\dim_\theta X$, which we show in Theorem \ref{th:applic-1} to have applications to the study of conformal Assouad spectrum $C\dim_A^\theta X$. 

Eriksson-Bique \cite{EB} has shown that 
$$
C\dim_H X = C\dim_A X
$$
for all quasi-self-similar metric spaces $(X,d)$.  We note, however, that this is not true for all self-similar sets, provided one allows for overlaps in the construction \cite{fr:book}.   The primary examples which we focus on in this paper are not quasi-self-similar. Relevant examples include self-affine carpets and sequences and sets defined by processes converging at a polynomial (as opposed to an exponential) rate. For an interesting conclusion relating polynomial convergence and box-counting dimension, we refer the reader to \cite[Theorem 1.2]{chr:spheres}.

Kovalev \cite{Kov} showed the nonexistence of spaces with conformal dimension strictly between zero or one. Specifically, if $(X,d)$ is a metric space with $\dim_H X < 1$ (resp.\  $\overline\dim_B X < 1$), then $C\dim_H X = 0$ (resp.\ $C\overline\dim_B X = 0$). Our first result extends Kovalev's theorem to cover all instances of the intermediate dimensions.

\begin{thm}\label{th:main-1}
Let $(X,d)$ be a totally bounded metric space and let $0<\theta<1$. If $\overline\dim_\theta X < 1$, then $C\overline\dim_\theta X = 0$.
\end{thm}

Coupling the conclusion of Theorem \ref{th:main-1} with certain inequalities relating the panoply of dimensional notions described above yields the following corollary.

\begin{thm}\label{th:applic-1}
Let $(X,d)$ be a totally bounded and doubling metric space such that $\dim_{qH} X < 1$. Then
\begin{equation}\label{eq:applic-1}
C\dim_A^\theta X = 0 \qquad \forall \, 0<\theta<1.
\end{equation}
In particular,
\begin{equation}\label{eq:applic-1-a}
C\overline\dim_B X = 0.
\end{equation}
\end{thm}

In particular, if \eqref{eq:HqH} is satisfied for a totally bounded and doubling space, then $\dim_H X < 1$ implies that \eqref{eq:applic-1} and hence in particular \eqref{eq:applic-1-a} hold true.

We say that a metric space $(X,d)$ has {\it vanishing conformal Assouad spectrum} if $C\dim_A^\theta X = 0$ for every $0<\theta<1$. Note that we do not know whether this condition is the same as the condition that $(X,d)$ has vanishing conformal quasi-Assouad dimension: $C\dim_{qA} X = 0$. Although $\dim_{qA} X$ is equal to the limit as $\theta \nearrow 1$ of $\dim_A^\theta X$, it is not clear whether the interchange of that limit with the infimum in the definition of conformal dimension is justified. See also Question \ref{q:continuity-of-conformal-Assouad-spectrum}.

Note that Theorem \ref{th:applic-1} is new only when $\overline\dim_B X \ge 1$; if $\overline\dim_B X < 1$ then the conclusion follows from Kovalev's theorem and known relationships between box-counting dimension and the Assouad spectrum. We also remark that the proof of Theorem \ref{th:applic-1} establishes a related result: if $X \subset \R^n$ and $\overline\dim_B X < n$, then the quasisymmetric reduction of Assouad spectrum guaranteed by the theorem can in fact be accomplished via global quasiconformal maps of $\R^n$.

Theorem \ref{th:applic-1} has an antecedent in work of Mackay on Assouad dimension and conformal Assouad dimension of Bedford--McMullen carpets \cite{Mackay}. Indeed, it follows from results in \cite{Mackay} that any such carpet with (quasi-)Hausdorff dimension strictly less than one necessarily has conformal (quasi-)Assouad dimension equal to zero. Note that Bedford--McMullen carpets $E$ satisfy \eqref{eq:HqH}; see \cite{ffk:intermediate} for the first proof of this fact and \cite{bk:intermediate-BM-carpets} for precise asymptotics for $\overline\dim_\theta X$ as $\theta \to 0$. Theorem \ref{th:applic-1}, which has the same hypothesis (quasi-Hausdorff dimension strictly less than one) but a weaker conclusion (vanishing conformal Assouad spectrum), applies to all metric spaces, not merely those originating from certain planar self-affine constructions.

Furthermore, Theorem \ref{th:applic-1} is sharp in the following sense: {\em there exists a compact, totally disconnected, planar set $E$ with $C\dim_H E = \dim_A E = 1$.} Note that these conclusions imply that all notions of dimension and conformal dimension described above are equal to one. Total disconnectivity is included here to rule out elementary connected examples. See Example \ref{ex:binder-hakobyan-li} for details. For earlier examples along these lines (phrased only for Hausdorff dimension), see \cite[Corollary 2]{bishop-tyson}.

A canonical example of a metric space for which Hausdorff, intermediate, and box-counting dimensions disagree is the polynomially decaying sequence set
\begin{equation}\label{eq:Ep}
E_s := \{m^{-s}: m \in \N \} \cup \{0\}, \qquad s>0.
\end{equation}
Indeed, 
$$
\dim_H E_s = 0 < \overline\dim_\theta E_s = \frac{\theta}{s+\theta} < \overline\dim_B E_s = \frac{1}{s+1}
$$ 
for any $0<\theta<1$. Using these sets, we construct a variety of Euclidean subsets fulfilling the assumptions in Theorem \ref{th:applic-1}. The following sets satisfy the hypotheses in Theorem \ref{th:applic-1}, and hence have vanishing conformal Assouad spectrum. Moreover, all of these examples have $\overline{\dim}_B X \geq 1$ and consequently this conclusion does not follow from Kovalev's theorem.
\begin{itemize}
\item[(a)] $E = Z \times E_s$ for any $s>0$ and $Z \subset \R^{n-1}$ such that $\dim_{qH} Z < 1$ and $\overline\dim_B Z \ge \frac{s}{s+1}$,
\item[(b)] $E = Z \times E_s$ for any $s>0$ and any Ahlfors $t$-regular $Z$ with $\tfrac{s}{s+1} \le t < 1$,
\item[(c)] $E = E_s \times E_s \subset \R^2$ provided $0<s \le 1$,
\item[(d)] $E = E_s \times E_{1/s} \subset \R^2$ for any $s>0$.
\end{itemize}

\subsection{Sobolev and quasiconformal distortion of intermediate dimension}

Next, we turn to results concerning the Sobolev and quasiconformal distortion properties of intermediate dimension, with applications to the quasiconformal classification of Euclidean sets. In what follows, we restrict to subsets of Euclidean space, although we anticipate that many results also extend to doubling metric measure spaces supporting a Poincar\'e inequality. For recent results on the Sobolev distortion of box-counting dimension of subsets of metric spaces, we refer the reader to \cite{chr:Sobolev-metric}.

To begin, we recall that if $f:E \to \R^N$ is an $\alpha$-H\"older map from a bounded set $E \subset \R^n$, then
\begin{equation}\label{eq:Holder-bound}
\overline\dim_\theta f(E) \le \frac{1}{\alpha} \overline\dim_\theta E.
\end{equation}
See \cite[14.2.1 (5)]{fal:intermediate-survey}. Further estimates for the distortion of intermediate dimension under H\"older maps can be found in \cite[Section 4]{bur:fractional-Brownian}.

The estimate in \eqref{eq:Holder-bound} has immediate applications to the quasiconformal classification problem   beyond what can be gleaned from the Hausdorff and box dimensions alone. For instance, among compact sets $E \subset \R^n$ with Hausdorff dimension zero, vanishing of the quasi-Hausdorff dimension is a quasiconformal invariant.

\begin{prop}\label{thm:application1}
Let $E,F \subset \R^n$, $n \ge 2$, be compact sets with $\dim_H E = \dim_H F =0$. If $\dim_H E = \dim_{qH} E$ and $\dim_H F < \dim_{qH} F$, then no quasiconformal map of $\R^n$ sends $E$ to $F$.
\end{prop}

\begin{proof}[Proof of Proposition \ref{thm:application1}]
Assume that a $K$-quasiconformal map $f:\R^n \to \R^n$ exists with $f(E) = F$. Then $f$ is locally $\alpha$-H\"older continuous with $\alpha = \tfrac1K$. Equation \eqref{eq:Holder-bound} gives
\begin{equation}\label{theta-F-K-theta-E}
{\overline\dim}_\theta F \le K {\overline\dim}_\theta E
\end{equation}
for all $0<\theta<1$. If $0 = \dim_H E = \dim_{qH} E$ and $0 = \dim_H F < \dim_{qH} F$, then we obtain a contradiction in the limit as $\theta \searrow 0^+$.
\end{proof}

Recall that if ${\overline\dim}_B E = \dim_A E$, then ${\overline\dim}_\theta E = \overline\dim_B E$ for all $\theta>0$, hence $\dim_{qH} E = {\overline\dim}_B E$, see \cite{ffk:intermediate}. We obtain the following corollary.

\begin{cor}\label{cor:application1}
Let $E,F \subset \R^n$, $n \ge 2$, be compact sets with $\dim_{qH} E = 0$ and $\dim_H F = 0$. If $0 < {\overline\dim}_B F = \dim_A F$, then no quasiconformal map of $\R^n$ sends $E$ to $F$.
\end{cor}

The polynomially decaying sequence set $E_s$, $s>0$, satisfies the hypotheses of Proposition \ref{thm:application1} and Corollary \ref{cor:application1}. Identifying $E_s$ with the subset $\{x \be_1:x \in E_s \} \subset \R^n$ for $n \ge 2$, we conclude that $E_s$ is not quasiconformally equivalent to any subset of $\R^n$ of zero Hausdorff dimension which also has equal and positive box-counting and Assouad dimension.

We obtain similar conclusions when both quasi-Hausdorff dimensions equal zero, if the rates of convergence of intermediate dimension as $\theta \to 0^+$ differ sufficiently.

\begin{prop}\label{thm:application2}
Let $E,F \subset \R^n$, $n \ge 2$, be compact sets with $\dim_{qH} E = \dim_{qH} F = 0$. Assume that $\theta \mapsto {\overline\dim}_\theta E$ and $\theta \mapsto {\overline\dim}_\theta F$ are right differentiable at $\theta=0$, but that 
\begin{equation}\label{eq:application2}
 \left. \frac{d}{d\theta} \, {\overline\dim}_\theta E \right|_{\theta = 0} < \left. \frac{d}{d\theta} \, {\overline\dim}_\theta F \right|_{\theta = 0} = + \infty.
\end{equation}
Then no quasiconformal map of $\R^n$ sends $E$ to $F$.
\end{prop} 

The proof of this result is similar to that of Proposition \ref{thm:application1} discussed above. Assume that $\dim_{qH} E  = \dim_{qH} F  = 0$, that \eqref{eq:application2} holds, and that $f$ is a $K$-quasiconformal (and hence $\alpha$-H\"older continuous) map with $f(E)=F$. Dividing both sides of \eqref{theta-F-K-theta-E} by $\theta$ and taking the limit as $\theta \searrow 0^+$ leads to a contradiction with \eqref{eq:application2}.

The conclusion in Proposition \ref{thm:application2} no longer holds true if $\left. \tfrac{d}{d\theta} \, {\overline\dim}_\theta E \right|_{\theta = 0}$ and $\left. \tfrac{d}{d\theta} \, {\overline\dim}_\theta F \right|_{\theta = 0}$ are both assumed to be finite. However, in certain cases one can still draw conclusions regarding the minimal dilatation for any quasiconformal map sending $E$ to $F$; see for instance Example \ref{ex:qc-equiv-of-e-s-sets}.

The preceding conclusions rely only on the H\"older distortion properties of intermediate dimension, but are restricted to the sets of Hausdorff dimension zero. To move beyond this restricted setting, we next investigate the distortion of intermediate dimension under Sobolev and quasiconformal mappings. 

Assume that $f:\Omega \to \Omega'$ is a $K$-quasiconformal mapping between domains in $\R^n$, $n \ge 2$. Basic definitions and properties of quasiconformal maps are recalled in section \ref{sec:background}. By Gehring's fundamental higher integrability theorem \cite{geh:higher}, $f$ lies in a local Sobolev space $W^{1,p}_\loc$ for some $p>n$ depending only on $n$ and $K$. Denote by
$$
p^\Sob_n(K)>n
$$ 
the supremum of all $p>n$ so that every $K$-quasiconformal map $f$ between $n$-dimensional domains lies in $W^{1,p}_\loc$. The value $p^\Sob_2(K) = \tfrac{2K}{K-1}$ was computed by Astala \cite{ast:2d-higher}, but the well-known conjecture $p^\Sob_n(K) = \tfrac{nK}{K-1}$ remains open for all dimensions $n \ge 3$. 

Recalling that the inverse of a $K$-quasiconformal map between domains in $\R^n$ is $K^{n-1}$-quasiconformal, we now review known results on Sobolev and quasiconformal dimension distortion. The starting point for this line of research is Gehring and V\"ais\"al\"a's 1973 result on the distortion of Hausdorff dimension.

\begin{thm}[Gehring--V\"ais\"al\"a]\label{th:gv}
Let $f:\Omega \to \Omega'$ be a $K$-quasiconformal map between domains in $\R^n$, $n \ge 2$. Let $E \subset \Omega$ satisfy $\dim_H E = s \in (0,n)$. Then
\begin{equation}\label{eq:qc-dim-H-bound}
\left( 1 - \frac{n}{p^\Sob_n(K)} \right) \left( \frac1s - \frac1n \right) \le \frac1{\dim_H f(E)} - \frac1n \le \left( 1 - \frac{n}{p^\Sob_n(K^{n-1})} \right)^{-1} \left( \frac1s - \frac1n \right).
\end{equation}
\end{thm}

For convenience, we introduce the following notation. For $p>n\ge 1$ and $0<s<n$, we set
\begin{equation}\label{eq:tau-s}
\tau_p(s) := \frac{ps}{p-n+s},
\end{equation}
\begin{equation}\label{eq:Phi-s}
\Phi(s) := \frac1s - \frac1n,
\end{equation}
and
\begin{equation}\label{eq:alpha-p}
\alpha(p) = 1-\frac{n}{p}.
\end{equation}
We recall that every $W^{1,p}$ map defined on a domain in $\R^n$ is locally $\alpha(p)$-H\"older continuous. Moreover,
\begin{equation}\label{eq:relationship}
\alpha(p) \Phi(s) = \Phi(\tau_p(s)).
\end{equation}
With this notation in place, \eqref{eq:qc-dim-H-bound} can be restated as follows:
\begin{equation}\label{eq:qc-dim-H-bound2}
\alpha(p^\Sob_n(K)) \Phi(\dim_H E) \le \Phi(\dim_H f(E)) \le \alpha(p^\Sob_n(K^{n-1}))^{-1} \Phi(\dim_H E).
\end{equation}
In view of \eqref{eq:relationship}, the left hand inequality in \eqref{eq:qc-dim-H-bound2} coincides with the estimate 
$$
\dim_H f(E) \le \tau_{p^\Sob_n(K)}(\dim_H E),
$$
whilst the right hand inequality can be rewritten as
$$
\dim_H E \le \tau_{p^\Sob_n(K^{n-1})}(\dim_H f(E)).
$$
Kaufman \cite{kau:sobolev} observed that the estimate
\begin{equation}\label{eq:Sob-dim-H-bound}
\dim_H f(E) \le \tau_p(\dim_H E).
\end{equation}
continues to hold for any $W^{1,p}$ map (not necessarily quasiconformal or even injective) $f:\Omega \to \R^N$, $\Omega \subset \R^n$, and for $E \subset \Omega$. In the same paper, Kaufman proved a corresponding statement for box-counting dimension: if $E \subset \Omega \subset \R^n$ is compact,\footnote{Kaufman does not include the compactness assumption, but phrases his conclusion for packing dimension instead of box-counting dimension. In the case of box-counting dimension itself, which is only well-defined for bounded sets, compactness appears to be necessary.} then
\begin{equation}\label{eq:Sob-dim-B-bound}
{\overline\dim}_B f(E) \le \tau_p({\overline\dim}_B E) \qquad \forall \, f \in W^{1,p}(\Omega:\R^N);
\end{equation}
consequently, if $f:\Omega \to \Omega'$ is $K$-quasiconformal, then
\begin{equation}\label{eq:qc-dim-B-bound}
\alpha(p^\Sob_n(K)) \Phi({\overline\dim}_B E) \le \Phi({\overline\dim}_B f(E)) \le \alpha(p^\Sob_n(K^{n-1}))^{-1} \Phi({\overline\dim}_B E)
\end{equation}
for compact sets $E \subset \Omega$. Kaufman also gave two distinct probabilistic arguments to demonstrate (in a nonconstructive fashion) the sharpness of both \eqref{eq:Sob-dim-H-bound} and \eqref{eq:Sob-dim-B-bound} in a strong sense: for any closed set $E \subset \R^n$ with $\dim E \ge s$ (where $\dim$ denotes either $\dim_H$ or ${\overline\dim}_B$; if $\dim = {\overline\dim}_B$ then assume $E$ compact) and any $p>n$, there exists $f \in W^{1,p}(\R^n:\R^n)$ so that $\dim f(E) \ge \tau_p(s)$. 

Sharpness of the quasiconformal distortion estimate \eqref{eq:qc-dim-H-bound} for Hausdorff dimension was established by Astala in \cite{ast:2d-higher}. Recently, Fuhrer, Ransford and Younsi \cite{fry:inf-harmonic} introduced a novel approach to the distortion of Hausdorff and box-counting dimensions under quasiconformal mappings using holomorphic motions and inf-harmonic functions. Their methodology also gives rise to examples demonstrating the sharpness of such estimates in the planar case.

Analogous results were established in \cite{tys:assouad} and \cite{ct:qc-assouad-spectrum} for the Assouad dimension and Assouad spectrum. We omit the precise statements here in the interest of brevity, but note some distinctive features. First, the distortion estimates for both $\dim_A$ and $\dim_A^\theta$ hold only for quasiconformal homeomorphisms and not for general supercritical Sobolev mappings. 
In addition, the role of the Sobolev higher integrability exponent $p^\Sob_n(K)$ in the relevant distortion estimates, e.g.\ \eqref{eq:qc-dim-H-bound}, is here replaced by the {\it reverse H\"older higher integrability exponent} $p^\RH_n(K)$. See Remark \ref{rem:assouad-case} for more information.

Our next result is a Sobolev distortion theorem for intermediate dimensions which interpolates between \eqref{eq:Sob-dim-H-bound} and \eqref{eq:Sob-dim-B-bound} and sharpens the H\"older estimate \eqref{eq:Holder-bound} for the subclass of supercritical Sobolev maps.

\begin{thm}\label{thm:sob_dim_dist_intermediate}
Let $\Omega \subset \R^n$, $n \ge 1$, be a domain, let $f \in W^{1,p}(\Omega:\R^N)$, $p>n$, be a continuous map, and let $E \subset \Omega$ be compact. Then
\begin{equation}\label{eq:intermediate-Sobolev}
{\overline\dim}_\theta f(E) \le \tau_p({\overline\dim}_\theta E) = \frac{p \, {\overline\dim}_\theta E}{p-n+{\overline\dim}_\theta E} \quad \forall \, 0<\theta<1.
\end{equation}
\end{thm}

An appeal to the higher integrability theorem yields consequences for quasiconformal mappings.

\begin{cor}\label{cor:qc_dim_dist_intermediate}
Let $f:\Omega \to \Omega'$ be a $K$-quasiconformal mapping between domains in $\R^n$, $n \ge 2$, and let $E$ be a compact subset of $\Omega$. Then for each $0<\theta<1$ we have
\begin{equation}\label{eq:intermediate-quasiconformal}
\alpha(p^\Sob_n(K)) \Phi({\overline\dim}_\theta E) \le \Phi({\overline\dim}_\theta f(E)) \le \alpha(p^\Sob_n(K^{n-1}))^{-1} \Phi({\overline\dim}_\theta E).
\end{equation}
\end{cor}

In particular, using Astala's theorem we obtain the following estimates in the planar case.

\begin{cor}\label{cor:qc_dim_dist_intermediate-2d}
Let $f:\Omega \to \Omega'$ be a $K$-quasiconformal mapping between domains in $\R^2$, and let $E$ be a compact subset of $\Omega$. Then for each $0<\theta<1$ we have
\begin{equation}\label{eq:intermediate-quasiconformal-2d}
K^{-1} \Phi({\overline\dim}_\theta E) \le \Phi({\overline\dim}_\theta f(E)) \le K \Phi({\overline\dim}_\theta E).
\end{equation}
\end{cor}

The proof of Theorem \ref{thm:sob_dim_dist_intermediate} involves a synthesis of the arguments used in the proofs of the analogous results for both Hausdorff and box-counting dimensions, following the reasoning in \cite{kau:sobolev}. Starting from a covering of $E$ by a collection of dyadic cubes admissible for the computation of ${\overline\dim}_\theta E$, we separate the collection of images of these cubes into several subclasses, depending on the size of the image in relation to both scales involved in the definition of ${\overline\dim}_\theta$. These various subclasses are treated by distinct methods, some of which lie closer to the argument used for Hausdorff dimension and others of which resemble the box-counting dimension argument. In order to implement the idea effectively, we take advantage of the fact that intermediate dimension can be computed by considering only coverings by dyadic cubes. While that fact is well-known for both Hausdorff and box-counting dimension, it is not explicitly stated in the literature in for intermediate dimensions. We therefore include a proof of this fact, see Proposition \ref{prop:intermediate-dimension-via-dyadic-cubes}. For a discussion of the role of dyadic cube-type constructions in the definitions of various metric dimensions on general doubling metric spaces, we refer to \cite{chr:dyadic}.

In Example \ref{ex:BMcarpets-after-BK}, we apply Corollary \ref{cor:qc_dim_dist_intermediate} to the quasiconformal classification question. Banaji and Kolossv\'ary \cite{bk:intermediate-BM-carpets} computed the full range of intermediate dimensions for Bedford--McMullen carpets. Example 2.14 of \cite{bk:intermediate-BM-carpets} highlights two such carpets $E$ and $F$ with the property that any $\alpha$-H\"older map sending $E$ to $F$ necessarily has $\alpha \le \alpha_0$ for some explicit $\alpha_0 < 1$. It follows that there is an explicit $K_0 > 1$ so that no $K$-quasiconformal map $f$ with $K<K_0$ sends $E$ to $F$. But employing Corollary \ref{cor:qc_dim_dist_intermediate} instead of the H\"older estimate leads to a strictly larger value $\widetilde{K_0} > K_0 >1$ so that the same conclusion continues to hold.

We conclude the paper with examples illustrating sharpness of some of our results. For instance, in Theorem \ref{thm:theta-qc-sharpness} we exhibit a $K$-quasiconformal mapping $f:\C \to \C$ and a compact set $E \subset \C$ so that
\begin{equation}\label{eq:theta-qc-sharpness-introduction}
\frac1K \left( \frac1{{\overline\dim}_\theta E} - \frac12 \right) = \frac1{{\overline\dim}_\theta f(E)} - \frac12 \qquad \forall \, 0<\theta \le 1.
\end{equation}
Our example is surprisingly elementary, e.g., it makes no use of the machinery of holomorphic motions or inf-harmonicity (as is present in the examples of Astala \cite{ast:2d-higher} and Fuhrer--Ransford--Younsi \cite{fry:inf-harmonic} for the Hausdorff dimension case). Moreover, in Theorem \ref{th:banaji-rutar-remark} we establish an intriguing relationship between the Banaji--Rutar lower bounds for intermediate dimension (Proposition \ref{prop:banaji-rutar}) and sharpness of Corollary \ref{cor:qc_dim_dist_intermediate} as seen in \eqref{eq:theta-qc-sharpness-introduction}. To wit, any compact, non-doubling and non-uniformly perfect subset of $\C$ which satisfies \eqref{eq:theta-qc-sharpness-introduction} for $\theta = 1$ (i.e., for box-counting dimension) and which is also sharp for the Banaji--Rutar bound \eqref{eq:BR} necessarily satisfies \eqref{eq:theta-qc-sharpness-introduction} for all $0<\theta \le 1$.

Finally, in Theorem \ref{thm:intermediate-sharpness}, we show that Theorem \ref{thm:sob_dim_dist_intermediate} is sharp in a similar strong sense: if $E \subset \R^n$ is compact and $p>n$, then there exists a continuous mapping $f \in W^{1,p}(\R^n:\R^n)$ so that 
$$
{\overline\dim}_\theta f(E) = \tau_p({\overline\dim}_\theta E) \qquad \forall \, 0<\theta\le 1.
$$
The proof is an adaptation of Kaufman's argument \cite{kau:sobolev} for box-counting dimension, reformulated in capacity-theoretic language. In particular, it is nonconstructive; the existence of such a map is established by constructing a suitable random family of $W^{1,p}$ maps, and showing that with positive probability some representative achieves the optimal estimate for a countable dense set of values $\theta_k \in (0,1)$. An appeal to the continuity of $\theta \mapsto {\overline\dim}_\theta E$ for $0<\theta \le 1$ completes the proof.

\subsection{Outline of the paper}

We close this introduction with an outline of the paper. Section \ref{sec:background} recalls basic terminology and notation. In section \ref{sec:intermediate} we review intermediate dimension along with some basic tools for its study. Section \ref{sec:Sobolev-distortion-intermediate} contains the proof of Theorem \ref{thm:sob_dim_dist_intermediate}, and comments on the relationship between Sobolev distortion results for intermediate dimension and quasiconformal distortion results for the Assouad spectrum. In section \ref{sec:QC-classification} we collect applications of H\"older and Sobolev distortion of intermediate dimension to the quasiconformal classification problem. In particular, we prove Theorems \ref{thm:application1} and \ref{thm:application2}. In Section \ref{sec:qc-lowering}, we turn to the question of lowering intermediate dimension by quasisymmetric mappings. We prove Theorems \ref{th:main-1} and \ref{th:applic-1} and present a diverse collection of examples of sets and spaces with vanishing conformal Assouad spectrum. Finally, in section \ref{sec:sharpness-section} we return to the topic of Sobolev and quasiconformal distortion and we prove the sharpness of \eqref{eq:intermediate-quasiconformal-2d} and \eqref{eq:intermediate-Sobolev} as discussed above. A short concluding section \ref{sec:open-questions} collects several open questions motivated by our results.

\subsection*{Acknowledgements}

We thank Roope Anttila, Amlan Banaji, Efstathios Chrontsios Garitsis, and Alex Rutar for helpful conversations regarding the topics of this paper. We are also grateful to Istv\'an Kolossv\'ary for providing us with access to unpublished computational data related to Example \ref{ex:BMcarpets-after-BK}.  We thank an anonymous referee for carefully reading our paper and for making several helpful suggestions. 

\section{Background and notation}\label{sec:background}

Multiplicative constants which arise in various estimates throughout this paper will be denoted by $C$ or $c$, occasionally decorated with subscripts to denote dependence on parameters. All such constants are always nonnegative. The value of such a constant may change from one line to the next, or even within a single line. When appropriate, we explicitly identify the dependence of such constants on auxiliary data. We use the expression $C$ when we wish to emphasize that a particular constant is {\it finite}, and we use $c$ when we wish to emphasize that the constant is {\it nonzero}. The notation $A \lesssim B$ will also be used; this expression is synonymous with the inequality $A \le C B$ for a constant $C$ as above. Similarly, $A \gtrsim B$ is synonymous with $A \ge c B$.

\subsection{Dyadic decomposition}\label{subsec:dyadic-and-Sobolev}

The standard dyadic decomposition of $\R^n$ is denoted $\cD = \cD(\R^n)$. For each $m \in \Z$, we let $\cD_m = \cD_m(\R^n)$ denote the collection of all cubes $Q \in \cD$ whose side length is equal to $2^{-m}$. Thus $\cD_m$ consists of all cubes of the form $[p_1 2^{-m},(p_1+1)2^{-m}] \times \cdots \times  [p_n 2^{-m},(p_n+1)2^{-m}]$ for integers $p_1,\ldots,p_n$. For a domain $\Omega \subset \R^n$ we denote by $\cD(\Omega)$ and $\cD_m(\Omega)$ the subsets of $\cD$ and $\cD_m$ respectively, consisting of cubes which are contained in $\Omega$. It is well known that Hausdorff dimension can be computed as the critical threshold for {\it dyadic Hausdorff measure}, defined as for the usual Hausdorff measure but using coverings by dyadic cubes. Similarly, box-counting dimension can be computed by restricting attention to coverings by dyadic cubes. In section \ref{subsec:dyadic-intermediate} we will establish a similar fact for intermediate dimension.

For an integer $M \ge 2$ we denote by $\cD(M)$ and $\cD_m(M)$ the corresponding $M$-adic decomposition into $M$-adic cubes of the form  $[p_1 M^{-m},(p_1+1)M^{-m}] \times \cdots \times  [p_n M^{-m},(p_n+1)M^{-m}]$ for integers $p_1,\ldots,p_n$. Thus, for example, $\cD = \cD(2)$. We will make use of the $M$-adic decomposition in the discussion of fractal percolation in section \ref{subsec:percolation}.

\subsection{Sobolev and quasiconformal maps}\label{subsec:quasiconformal}

A map $f=(f_1,\ldots,f_N):\Omega \to \R^N$, where $\Omega$ is a domain in $\R^n$, lies in the Sobolev space $W^{1,p}(\Omega:\R^N)$ if each component $f_j \in W^{1,p}(\Omega)$. Let $f$ be the continuous representative of a map in $W^{1,p}(\Omega:\R^N)$ for $p>n$. Then $f$ satisfies the {\bf Morrey--Sobolev inequality}: there exists $C = C(n,p) > 0$ so that the inequality
\begin{equation}\label{eq:MS}
|f(x)-f(y)| \le C |x-y|^{1-n/p} \left( \int_Q |Df|^p \right)^{1/p}
\end{equation}
holds for all $x,y \in Q$, whenever $Q$ is a closed cube contained in $\Omega$.

A homeomorphism $f:\Omega \to \Omega'$ between domains in $\R^n$, $n \ge 2$, is {\it $K$-quasiconformal ($K$-QC)} for some $K \ge 1$ if $f$ lies in the local Sobolev space $W^{1,n}_\loc(\Omega:\R^n)$ and the inequality $||Df(x)||^n \le K \det Df(x)$ holds for a.e.\ $x \in \Omega$. The inverse of a $K$-QC map is $K^{n-1}$-QC. As explained in the introduction, the {\it Sobolev higher integrability exponent}
$$
p^\Sob_n(K) = \sup \{ p>n: \mbox{every $K$-QC map $f$ between domains in $\R^n$ lies in $W^{1,p}_\loc$} \}
$$
exists for each $K \ge 1$ and $n \ge 2$. The radial stretch map $f(x) = |x|^{1/K-1}x$ shows that $p^\Sob_n(K)$ necessarily satisfies $p^\Sob_n(K) \le \tfrac{nK}{K-1}$ for all $n$ and $K$. Astala \cite{ast:2d-higher} proved that $p^\Sob_2(K) = \tfrac{2K}{K-1}$, but the sharp value of $p^\Sob_n(K)$ is unknown when $n \ge 3$. Iwaniec and Martin \cite{im:higher} established the existence, for each $n \ge 3$, of a constant $\lambda(n) \ge 1$ so that
\begin{equation}\label{eq:im}
p^\Sob_n(K) \ge \frac{n \lambda(n) K}{\lambda(n) K - 1}.
\end{equation}

A homeomorphism $f:\R^n \to \R^n$, $n \ge 2$, is quasiconformal if and only if it is quasisymmetric. If $f:\Omega \to \Omega'$ is a homeomorphism between domains in $\R^n$, then $f$ is quasiconformal if and only if it is locally quasisymmetric.

\subsection{Assouad dimension, lower dimension, and the Assouad spectrum}\label{subsec:metric}

In a metric space $(X,d)$ we denote by $B(x,r)$ the closed ball with center $x$ and radius $r$. A metric space $(X,d)$ is {\em doubling} if there exists a constant $C>0$ so that every ball $B(x,r)$ can be covered by at most $C$ balls of radius $r/2$. Moreover, $(X,d)$ is {\em uniformly perfect} if there exists $0<c<1$ so that for every ball $B(x,r)$ with $r<\diam(X)$, the set $B(x,r) \setminus B(x,cr)$ is nonempty. Observe, for instance, that every connected metric space is uniformly perfect with $c=\tfrac12$.

The {\em Assouad dimension} $\dim_A X$ is the infimum of those values $s>0$ so that there exists $C>0$ so that for every ball $B(x,R)$ in $X$ and every $0<r<R$, $B(x,R)$ can be covered by no more than $C(R/r)^s$ balls of radius $r$. The {\em lower dimension} $\dim_L X$ is the supremum of those values $t>0$ so that there exists $C>0$ so that for every ball $B(x,R)$ in $X$ and every $0<r<R$, it requires at least $C(R/r)^t$ balls of radius $r$ to cover $B(x,R)$.

It is known that $\dim_A X < \infty$ if and only if $X$ is doubling. Similarly, $\dim_L X > 0$ if and only if $X$ is uniformly perfect. See \cite[Section 13.1.1]{fr:book}. Moreover, $\dim_L X \le \dim_H X$ for all compact $X$ and $\overline\dim_B X \le \dim_A X$ for all totally bounded $X$.

According to a result of Heinonen \cite[Theorem 14.16]{H}, if $(X,d)$ is uniformly perfect, then the conformal Assouad dimension $C\dim_A X$ agrees with the Ahlfors regular conformal dimension $ARC\dim X$, defined as the infimum of those exponents $Q$ for which there exists a quasisymmetric homeomorphism from $X$ to an Ahlfors $Q$-regular metric space $(Y,d')$.

The {\it Assouad spectrum} was introduced by the first named author and Yu in \cite{fy:assouad-spectrum} and here we define the (related) upper Assouad spectrum, but we refer to it throughout as simply the Assouad spectrum. For $0<\theta<1$, $\dim_A^\theta X$ denotes the infimum of those values $s>0$ so that there exists a constant $C>0$ so that for every $x \in X$ and all $0<r\le R^{1/\theta} \le R \le 1$, the inequality
$$
N(B(x,R),r) \le C(R/r)^s
$$
holds true. The function $\theta \mapsto \dim_A^\theta X$ is nondecreasing, with $\lim_{\theta\to 0} \dim_A^\theta X = \overline\dim_B X$. The {\em quasi-Assouad dimension} is $\dim_{qA} X = \lim_{\theta \to 1} \dim_A^\theta X$. The quasi-Assouad dimension was introduced, expressed in different language, by L\"u  and Xi \cite{quasiassouad}. For future reference, we record the inequality
\begin{equation}\label{eq:assouad-spectrum-inequality}
\dim_A^\theta X \le \min \left\{ \frac{\overline\dim_B X}{1-\theta}, \dim_{qA} X \right\},
\end{equation}
see e.g.\ \cite[Lemma 3.4.4]{fr:book}.

\subsection{Bedford--McMullen carpets}\label{subsec:bm-carpets}

Fix integers $n>m \ge 2$ and a collection $\cF$ of contractive affine maps $f_{(i,j)}$ of $\R^2$, where $f_{(i,j)}(x,y) = (m^{-1}(x+i-1),n^{-1}(y+j-1))$ and $(i,j)$ ranges over an index set $\cA \subset \{1,\ldots,m\} \times \{1,\ldots,n\}$. Associated to the iterated function system $\cF = \{f_{(i,j)}:(i,j) \in \cA\}$ is a unique nonempty compact set $E = E_\cF$, termed a {\it Bedford--McMullen carpet}. 

Let $N = \# \cA$, $M = \#\{i:\exists\,j \mbox{ s.t.\ }(i,j) \in \cA\}$, $N_i = \#\{j:(i,j) \in \cA\}$, and $\gamma = \log_m(n) > 1$. Bedford \cite{bed:carpets} and McMullen \cite{mcm:carpets} computed the Hausdorff and box-counting dimensions of $E$, obtaining the formulas 
\begin{equation}\label{eq:BM-dim-H}
\dim_H E = \log_m\bigg (\sum_{i=1}^M N_i^{1/\gamma}\bigg)
\end{equation}
and
\begin{equation}\label{eq:BM-dim-B}
{\overline\dim}_B E = \log_m(M) + \log_n(N/M).
\end{equation}
Assouad dimensions of Bedford--McMullen carpets were computed by Mackay \cite{Mackay}, who showed that
\begin{equation}\label{eq:BM-dim-A}
\dim_A E =  \log_m(M) + \max_i \log_n(N_i).
\end{equation}
Finally, lower dimensions of these carpets were computed by the first named author \cite{fr:homogeneity-fractals} as follows:
\begin{equation}\label{eq:BM-dim-L}
\dim_L E = \log_m(M) + \min_i \log_n(N_i).
\end{equation}
Since $\dim_L E > 0$, all Bedford--McMullen carpets are uniformly perfect.

The Assouad spectra of Bedford--McMullen carpets were computed by the first named author and Yu, see also \cite[Theorem 8.3.3]{fr:book}. One has
\begin{equation}\label{eq:BM-dim-A-theta}
\dim_A^\theta E = \frac{\overline\dim_B E - \theta \left( \log_m(N/\max_i N_i) + \log_n(\max_i N_i) \right)}{1-\theta}, \qquad 0<\theta <\gamma^{-1}
\end{equation}
and $\dim_A^\theta E = \dim_A E$ for $\gamma^{-1} \le \theta < 1$.

The intermediate dimensions of Bedford--McMullen carpets have a substantially more intricate behavior. These values have recently been computed by Banaji and Kolossv\'ary \cite{bk:intermediate-BM-carpets},
who show that  the value of ${\overline\dim}_\theta E $, for $\gamma^{-L} < \theta \le \gamma^{1-L}$ and $L \in \N$, is the unique solution $s = s(\theta)$ to the equation
$$
\gamma^L\theta \log N - (\gamma^L \theta - 1) t_L(s) + \gamma(1-\gamma^{L-1}\theta) (\log M - I(t_L(s))) - s \log n = 0.
$$
Here $t_\ell(s) := T_s^{(\ell-1)}((s-\log_m(M)) \log n)$, $T_s(t) := (s-\log_m(M)) \log n + \gamma \, I(t)$, and $I(t)$ is the Legendre transform of the function $\lambda \mapsto \log(M^{-1} \sum_{i=1}^M N_i^\lambda)$. 

The function $\theta \mapsto {\overline\dim}_\theta E$ so obtained showcases novel features not previously observed for intermediate dimension functions in natural cases. In fact, it is piecewise analytic and piecewise strictly concave with the `pieces' precisely given by the intervals of the form $(\gamma^{-L},\gamma^{1-L})$, $L \in \N$, has discontinuous first derivative at each $\theta_L := \gamma^{-L}$, and has infinite right derivative at $\theta = 0$.\footnote{More precisely, there exists $C<\infty$ such that for sufficiently small $\theta$, one has $\dim_H E + C^{-1} (\log \theta)^{-2} \le {\overline\dim}_\theta E \le \dim_H E + C (\log \theta)^{-2}$.}

\subsection{Mandelbrot percolation}\label{subsec:percolation}

Fractal percolation, a.k.a.\ Mandelbrot percolation, is one of the most well-known examples of a stochastic fractal. It has been extensively studied since the early works of Mandelbrot in the 1970s \cite{man:turbulence}. The survey by Rams and Simon \cite{RS:survey} provides a clear overview of the subject. We review known results on the dimensions of samples of Mandelbrot percolation. Let $Q = [0,1]^n$ denote the unit cube, and fix an integer $M \ge 2$ and a probability $p \in (0,1)$. Consider the collection of level one $M$-adic cubes $R$ in $\cD_1(M)$ such that $R \subset Q$. We keep each such cube independently with probability $p$. For each cube $R \in \cD_1(M)$ which survives, we repeat the process, considering the collection of level two $M$-adic cubes $S$ in $\cD_2(M)$ such that $S \subset R$ and keeping each such cube independently with probability $p$. Iterating over all levels $m \in \N$, we define the {\it fractal percolation} as the random set
$$
F := \bigcap_{m \in \N} \bigcup_{R} R,
$$
where the union is taken over all cubes $R \in \cD_m(M)$ which survive the above process. If $p \le M^{-n}$ then the random set $F$ is empty almost surely. If $p > M^{-n}$ then there is a positive probability that $F$ is non-empty and we have almost sure formulas for the dimensions of $F$, conditioned on non-extinction. Specifically,
$$
\dim_H F = \overline\dim_\theta F = \overline\dim_B F = \dim_A^\theta F = n - \frac{\log(1/p)}{\log(M)} \qquad \forall \, 0<\theta<1
$$
almost surely, whilst
$$
\dim_A F = n
$$
almost surely. See \cite{FMT:random} and \cite[section 9.4]{fr:book} for details.

\section{Intermediate dimension}\label{sec:intermediate}

The intermediate dimensions were introduced in \cite{ffk:intermediate} for subsets of Euclidean space. We follow Banaji \cite{ban:generalized-intermediate} for the theory for metric spaces, but discuss the general problem of defining the intermediate dimensions in arbitrary metric spaces in subsection \ref{sec:intmetric} below.

Let $X$ be a totally bounded metric space, and fix an isometric embedding $X \hookrightarrow Z$ into a uniformly perfect metric space $Z$. For instance, since $X$ is separable we may embed $X$ into the Banach space $\ell^\infty$ via the Fr\'echet embedding \cite[Exercise 12.6]{H}.

Let $0<\theta < 1$. The $\theta$-intermediate dimension of $X$ is defined with respect to the aforementioned embedding, see however Proposition \ref{prop:independence} below. Following \cite{ffk:intermediate} and \cite{ban:generalized-intermediate}, we define the {\it (upper) $\theta$-intermediate dimension} of $X$, denoted ${\overline{\dim}_\theta} X$, to be the infimum of those values $s>0$ such that for each $\eps>0$ there exists $\delta_0>0$ so that for each $0<\delta \le \delta_0$ there is a covering $(A_i)$ of $X$ by subsets of $Z$ satisfying
\begin{equation}\label{eq:theta-dim-1}
\delta^{1/\theta} \le \diam(A_i) \le \delta \qquad \forall \, i
\end{equation}
and
\begin{equation}\label{eq:theta-dim-2}
\sum_i (\diam A_i)^s < \eps.
\end{equation}
Note that such coverings exist for all $\delta\le\delta_0$, provided that the constant $\delta_0$ is chosen sufficiently small relative to $\theta$ and the uniform perfectness constant $c$ of $Z$. Indeed, assume that $\delta_0 \le (c/2)^{\theta/(1-\theta)}$. Then for any $z \in Z$ and any $0<r\le \tfrac12\delta_0$ the ball $B(z,r)$ in $Z$ satisfies the diameter constraints in \eqref{eq:theta-dim-1} for $\delta = 2r$. For the lower bound, we note that $\diam(B(z,r)) \ge cr \ge \delta^{1/\theta}$ by the choice of $\delta$ and the constraint on $\delta_0$.

We recall that the function $\theta \mapsto \overline\dim_\theta X$, $0<\theta<1$, is non-decreasing and continuous; see \cite[Theorem 3.12]{ban:generalized-intermediate} for these results in the metric space setting. Moreover, the value of $\overline\dim_\theta X$ for a given $\theta \in (0,1)$ is constrained by the values of the lower, box-counting, and Assouad dimensions of $X$. An early indication of this phenomenon appears in \cite[Proposition 2.4]{ffk:intermediate}; the version which we present here was proved first by Banaji and Rutar \cite[Corollary 2.8]{br:attainability} in the setting of subsets of Euclidean space. See Banaji \cite[Corollary 3.14]{ban:generalized-intermediate} for the metric space version.

\begin{prop}[Banaji--Rutar, Banaji]\label{prop:banaji-rutar}
Let $X$ be a totally bounded metric space with lower dimension $\dim_L X = \lambda$, box-counting dimension $\overline\dim_B X = \beta$, and Assouad dimension $\dim_A X = \alpha$, for some $0\le\lambda \le \beta \le \alpha < \infty$. For $0<\theta<1$ we have
\begin{equation}\label{eq:BR}
{\overline\dim}_\theta X \ge \frac{\alpha (\beta - \lambda) \theta + (\alpha - \beta) \lambda}{(\beta - \lambda) \theta + (\alpha-\beta)}.
\end{equation}
\end{prop}

Recall from section \ref{subsec:metric} that a metric space $X$ has finite Assouad dimension if and only if it is doubling. Since the expression on the right hand side of \eqref{eq:BR} is decreasing as a function of $\alpha$ and increasing as a function of $\lambda$, we obtain a weaker conclusion if we replace $\lambda$ by $0$ and $\alpha$ by $+\infty$.  To wit, if $X$ is a totally bounded and doubling metric space, then
\begin{equation}\label{eq:BR-cor}
{\overline\dim}_\theta X \ge\theta \, \cdot \, \overline\dim_B X
\end{equation}
for each $0<\theta<1$.


\subsection{Dyadic intermediate dimension}\label{subsec:dyadic-intermediate}

In preparation for our proof of the Sobolev distortion estimate \eqref{eq:intermediate-Sobolev}, we now show that restricting to coverings by dyadic cubes does not change the value of the intermediate dimensions. Given $E \subset \R^n$ bounded and $0<\theta\le 1$, let
$$
\overline{\dim}_{\theta,\cD}(E)
$$
be the infimum of those values $s>0$ such that for each $\eps>0$ there exists $\delta_0>0$ so that for each $0<\delta \le \delta_0$ there is a covering $(Q_i)$ of $E$ by dyadic cubes satisfying
\begin{equation}\label{eq:int-dim-1}
\delta^{1/\theta} \le \diam(Q_i) \le \delta \qquad \forall \, i
\end{equation}
and
\begin{equation}\label{eq:int-dim-2}
\sum_i (\diam Q_i)^s < \eps.
\end{equation}
Note that since cubes in $\cD_m$ have diameter $\sqrt{n} \, 2^{-m}$, all of the cubes $Q_i$ involved in the above covering (and satisfying \eqref{eq:int-dim-1}) necessarily lie in $\cD_m$ for some $m$ satisfying
\begin{equation}\label{eq:dyadic-cube-estimate}
\log_2(\tfrac1\delta) \le m -  \tfrac12 \log_2(n) \le \tfrac1\theta \log_2(\tfrac1\delta).
\end{equation}

\begin{prop}\label{prop:intermediate-dimension-via-dyadic-cubes}
For each bounded set $E \subset \R^n$ and each $0<\theta<1$,
$$
\overline{\dim}_\theta(E) = \overline{\dim}_{\theta,\cD}(E).
$$
\end{prop}

\begin{proof}
The inequality
$$
\overline{\dim}_\theta(E) \le \overline{\dim}_{\theta,\cD}(E)
$$
is clear, so we focus on proving the reverse inequality.

Fix $s>\overline{\dim}_\theta(E)$. For each $\eps>0$ there exists $\delta_0=\delta_0(\eps)>0$ so that for each $\delta\le \delta_0$ there is a covering $(A_i)$ of $E$ satisfying \eqref{eq:int-dim-1} and \eqref{eq:int-dim-2}. We will find a value $\theta' = \theta'(\theta,\delta_0) < \theta$ so that
$$
\overline{\dim}_{\theta',\cD}(E) < s.
$$
Moreover, $\theta'(\theta,\delta_0) \to \theta$ as $\delta_0 = \delta_0(\eps) \to 0$. In view of the continuity of $\theta \mapsto \overline{\dim}_\theta(E)$, this suffices to complete the proof.

For each index $i$, let $m_i\ge 1$ be the unique integer satisfying
$$
\sqrt{n} 2^{-m_i} < \diam(A_i) \le \sqrt{n} 2^{-m_i + 1}.
$$

\begin{lem}
For each $i$, the set $A_i$ can be covered by at most $P(n)$ elements of $\cD_{m_i}$, where $P(n)$ is an explicit constant depending only on $n$.
\end{lem}

For the proof of the lemma, fix an integer $p > 2\sqrt{n}$. Next, choose one cube $Q_0 \in \cD_{m_i}$ so that $Q_0 \cap A_i \ne \emptyset$, and consider the collection $\cC$ of all cubes $Q \in \cD_{m_i}$ so that there exists a chain of $p+1$ elements of $\cD_{m_i}$,
$$
{Q_0},{Q_1},\ldots,{Q_p}={Q},
$$
so that $\overline{Q_\ell} \cap \overline{Q_{\ell-1}} \ne \emptyset$ for each $\ell=1,\ldots,p$. The size of $\cC$ is $P(n) := (2p+1)^n$. We claim that the collection $\cC$ of cubes covers $A_i$. If not, then there exist points $x \in Q_0 \cap A_i$ and $y \in A_i \setminus \cup\{Q : Q \in \cC \}$ and we obtain the contradiction
$$
2\sqrt{n} 2^{-m_i} < p 2^{-m_i} < |x-y| \le \diam(A_i) \le \sqrt{n} 2^{-m_i+1}.
$$

Returning to the proof of the proposition, we now index, for each $i$, the elements of the covering collection so obtained as $(Q_{ij})$. For each $i$ and $j$ we have
$$
\diam(Q_{ij}) = \sqrt{n} 2^{-m_i} \le \diam(A_i) \le \delta
$$
and
$$
\diam(Q_{ij}) \ge \tfrac12 \diam(A_i) \ge \tfrac12 \delta^{1/\theta} \ge \delta^{1/\theta'}
$$
where
\begin{equation}\label{eq:theta-prime}
\theta' := \frac{\theta}{1 + \tfrac{\log 2}{-\log \delta_0} \theta}.
\end{equation}
Indeed, \eqref{eq:theta-prime} is equivalent to $\delta_0^{1/\theta'-1/\theta} = \tfrac12$ which implies that $\delta^{1/\theta'-1/\theta} \le \tfrac12$ for any $\delta \le \delta_0$, since $\theta'<\theta$. We also note that $\theta' \to \theta$ as $\delta_0 \to 0$.

The collection $(Q_{ij})$ for all relevant choices of $i$ and $j$ is therefore an admissible covering for the $\theta'$-intermediate dimension of $E$ computed with respect to dyadic cubes. Finally, we note that
$$
\sum_{i,j} (\diam Q_{ij})^s \le P(n) \sum_i (\diam A_i)^s < P(n) \eps =: \eps'
$$
and $\eps'\to0$ as $\eps \to 0$. We conclude that
$$
\overline{\dim}_{\theta',\cD}(E) < s,
$$
which, as explained above, suffices to complete the proof of Proposition \ref{prop:intermediate-dimension-via-dyadic-cubes}.
\end{proof}

\subsection{A capacitary approach to intermediate dimension}\label{subsec:intermediate-capacity}

In \cite{bff:intermediate-projections}, motivated by the study of Marstrand-type projection theorems, a capacitary formulation for intermediate dimensions was introduced. This builds on earlier work of Falconer \cite{fal:capacity-for-box-dimension} in the box-counting dimension case. For box-counting and intermediate dimension, the almost sure value of the dimensions of $m$-dimensional projections of a subset $E \subset \R^n$ can be characterised via the associated {\it $m$-dimensional dimension profile}. In this paper, we only use this concept in the special case $m=n$, in which case the dimension profile agrees with the original notion of dimension.

For a bounded set $E \subset \R^n$ and for $0\le s \le n$ and $0<\theta \le 1$, set
\begin{equation}\label{eq:S-s-r-theta}
S_{r,\theta}^s(E) := \inf \{ \sum_i (\diam A_i)^s \, : \, E \subset \bigcup_i A_i, r \le \diam A_i \le r^\theta \, \forall \, i \}.
\end{equation}
Following \cite[Lemma 2.1]{bff:intermediate-projections} we note that ${\overline\dim}_\theta \, E = s_0$
where $s = s_0$ is the unique nonnegative value such that
$$
\limsup_{r \to 0} \frac{\log S^s_{r,\theta}(E)}{\log(1/r)} = 0.
$$

\begin{defn}
For $s>0$, $m \in \{1,\ldots,m\}$, $r>0$, and $0<\theta \le 1$, introduce the kernel $\phi_{r,\theta}^{s,m}(u) := \min \{ 1, (\tfrac{r}{u})^s, \tfrac{r^{\theta(m-s)+s}}{u^m} \}$. The {\it $\phi_{r,\theta}^{s,m}$-energy} of a Radon measure $\mu$ is
\begin{equation}\label{eq:Ers}
\cE_{r,\theta}^{s,m}(\mu) := \iint \phi_{r,\theta}^{s,m}(|x-y|) \, d\mu x \, d\mu y,
\end{equation}
and the {\it $\phi_{r,\theta}^{s,m}$-capacity} of a set $E \subset \R^n$ is
\begin{equation}\label{eq:Crs}
C_{r,\theta}^{s,m}(E) := \left( \inf \{ \cE_{r,\theta}^{s,m}(\mu) \, : \, \mu \in \cM(E) \} \right)^{-1}.
\end{equation}
Finally, denote by ${\overline\dim}_\theta^m \, E = s_0'$ the unique nonnegative value $s=s_0'$ such that
$$
\limsup_{r\to 0} \frac{\log C_{r,\theta}^{s,m}(E)}{\log(1/r)} = s,
$$
and call ${\overline\dim}_\theta^m \, E$ the {\it $s$-box dimension profile} of $E$.
\end{defn}

We collect some basic observations in the following proposition. Details can be found in \cite[Lemma 3.1, Proposition 4.2, and Theorem 4.1]{bff:intermediate-projections}. Item (iii) highlights the connection between the capacities $C_{r,\theta}^{s,m} (E)$ and the covering sums $S^s_{r,\theta}(E)$.

\begin{prop}\label{prop:box-profile}
Let $E \subset \R^n$ be a nonempty and bounded set. 
\begin{itemize}
\item[(i)] For each $s>0$ and $r>0$, the infimum in \eqref{eq:Crs} is achieved by some measure $\mu_0$ in $\cM(E)$.
\item[(ii)] For $\mu_0$-a.e. $x \in E$, $\int \phi_r^s(|x-y|) \, d\mu_0(y) = C_r^s(E)^{-1}$.
\item[(iii)] There exists $r_0>0$ and a constant $C(n)>0$ so that
$$
r^s C^{s,n}_{r,\theta}(E) \le S^s_{r,\theta}(E) \le C(n) (\log(\tfrac{\diam E}{r})) \, r^s C^{s,n}_{r,\theta}(E) \qquad \forall\, 0<r\le r_0.
$$
\item[(iv)] ${\overline\dim}_\theta(E) = {\overline\dim}_\theta^n(E)$.
\end{itemize}
\end{prop}

When either $\theta = 1$ or $s=m$ the kernel $\phi_{r,\theta}^{s,m}$ reduces to the kernel $\phi_r^m$ used by Falconer in \cite{fal:capacity-for-box-dimension} in the setting of box-counting dimension. This special case will resurface in Example \ref{ex:dim-B-sharpness-sketch}, where we sketch a capacitary proof of Kaufman's result on the sharpness of the Sobolev distortion estimate \eqref{eq:Sob-dim-B-bound}. For use in that example, we recall here that the box-counting profile ${\overline\dim}_B^m E$ is defined in terms of the capacity $C_r^m(E) := ( \inf \{ \cE_r^m(\mu) \, : \, \mu \in \cM(E) \} )^{-1}$ and the corresponding energy $\cE_r^m(\mu):= \iint \phi_r^m(|x-y|) \, d\mu x \, d \mu y$ via the formula
$$
{\overline\dim}_B^m E := \limsup_{r \to 0} \frac{\log C_r^m(E)}{\log(1/r)}.
$$
An analog of Proposition \ref{prop:box-profile} holds, see e.g.\ Lemma 2.1, Corollary 2.4, and Corollary 2.5 in \cite{fal:capacity-for-box-dimension}. In particular, for $E \subset \R^n$ bounded, ${\overline\dim}_B^n E = {\overline\dim}_B E$.

\smallskip

For later purposes we record here a symmetrisation lemma.

\begin{lem}\label{lem:Ers-symmetrization}
Let $B^n$ be the unit ball in $\R^n$. For $y \in B^n$, $a>0$, $m \in \{1,\ldots,n\}$, $0<s\le m$, and $0<r<1$,
\begin{equation}\label{eq:Ers-symmetrization}
\int_{B^n} \phi_{r,\theta}^{s,m} (|ax-y|) \, dx \le \int_{B^n} \phi_{r,\theta}^{s,m} (|ax|) \, dx.
\end{equation}
\end{lem}

An anonymous referee pointed out that this lemma can be proved by appealing to the  Hardy--Littlewood rearrangement inequality applied to the functions $f(x) = \phi_{r,\theta}^{s,m} (|ax-y|)$ and $g(x) = \chi_{B^n}(x)$.  For the reader's benefit, we retain our alternative  direct  proof here.

\begin{proof}
By Cavalieri's principle,
$$
\int_{B^n} \phi_{r,\theta}^{s,m} (|ax-y|) \, dx = \int_0^1 \Vol( \{ x \in B \, : \, \phi_{r,\theta}^{s,m}(|ax-y|) > t \} ) \, dt,
$$
where $\Vol$ denotes Lebesgue measure in $\R^n$. It suffices to show that
\begin{equation}\label{eq:Ers-symmetrization-proof-1}
 \Vol( \{ x \in B^n \, : \, \phi_{r,\theta}^{s,m}(|ax-y|) > t \} ) 
\le
 \Vol( \{ x \in B^n \, : \, \phi_{r,\theta}^{s,m}(|ax|) > t \} ) 
\end{equation}
for all $0<t<1$.

We note that
$$
\phi_{r,\theta}^{s,m}(u) = \begin{cases} 1 & \mbox{if $0\le u < r$,} \\
\left( \frac{r}{u} \right)^s & \mbox{if $r \le u < r^\theta$,} \\
\frac{r^{\theta(m-s)+s}}{u^m} & \mbox{if $r^\theta \le u$.}
\end{cases}
$$
Set $z:=a^{-1} y$ and define radii 
$$
R_1 := a^{-1} r, \,\, R_2 := a^{-1} r^\theta, \,\, R_2' := a^{-1} t^{-1/s}r, \,\, \mbox{and} \,\, R_3 := a^{-1} t^{-1/m} r^{\theta + (1-\theta)(s/m)}.
$$
Introducing the standard notation $a\wedge b := \min\{a,b\}$ we observe that $R_1 < R_2 \wedge R_2'$ always, whilst
$$
\biggl( R_2 < R_3 \quad \Leftrightarrow \quad t < r^{(1-\theta)s} \biggr) \qquad \Longrightarrow \qquad R_2 = R_2 \wedge R_2'
$$
and
$$
\biggl( R_2' < R_3 \quad \Leftrightarrow \quad t > r^{(1-\theta)s} \biggr) \qquad \Longrightarrow \qquad R_2' = R_2 \wedge R_2.
$$
We observe that $\phi_{r,\theta}^{s,m}(|ax-y|) > t$ if and only if one of the following three conditions holds true:
\begin{itemize}
\item $x \in B(z,R_1)$, or
\item $x \in B(z,R_2 \wedge R_2') \setminus B(z,R_1)$, or
\item $t < r^{(1-\theta)s}$ and $x \in B(z,R_3) \setminus B(z,R_2)$.
\end{itemize} 
It follows that
\begin{equation*}\begin{split}
\Vol( \{ x \in B \, : \, \phi_{r,\theta}^{s,m}(|ax-y|) > t \} ) = \begin{cases} \Vol(B^n \cap B(z,R_3)) & \mbox{if $0<t<r^{(1-\theta)s}$,} \\
\Vol(B^n \cap B(z,R_2')) & \mbox{if $r^{(1-\theta)s} < t < 1$.} \end{cases}
\end{split}\end{equation*}
Taking advantage of the elementary fact that 
$$
\Vol(B(X,R) \cap B(Y,R')) \le \Vol(B(X,R \wedge R')) = \Vol(B(X,R) \cap B(X,R'))
$$
for $X,Y \in \R^n$ and $R,R'>0$,  we conclude that \eqref{eq:Ers-symmetrization-proof-1} holds true. The proof is complete.
\end{proof}

\subsection{Defining intermediate dimensions in arbitrary metric spaces} \label{sec:intmetric}

In this subsection we address two obvious questions concerning how to define intermediate dimensions in arbitrary metric spaces:~(i) does our definition depend on the choice of embedding? (ii) can the intermediate dimensions be defined without going via an embedding? 

In this paper we chose to define intermediate dimensions by  first embedding  into a uniformly perfect space.  Fortunately,  this definition does not depend on the choice of embedding.

\begin{prop}\label{prop:independence}
Let $X$ be a totally bounded metric space. The function $\theta \mapsto \overline\dim_\theta X$, $0<\theta<1$, is independent of the choice of isometric embedding $X \hookrightarrow Z$. 
\end{prop}
 
Proposition \ref{prop:independence} is an easy consequence of the following two lemmas. For convenience in this subsection, we denote by $\overline\dim_\theta^{X\hookrightarrow Z} X$ the value of the $\theta$-intermediate dimension of $X$ computed with respect to a fixed isometric embedding of $X$ into a uniformly perfect metric space $Z$.

\begin{lem}\label{lem:independence-1}
If $X \hookrightarrow Z$ and $Z \hookrightarrow W$ are isometric embeddings, with $Z$ and $W$ uniformly perfect, then
$$
\overline\dim_\theta^{X\hookrightarrow Z} X = \overline\dim_\theta^{X \hookrightarrow W} X
$$
for each $0<\theta<1$.
\end{lem}

If $X \stackrel{\iota_1}{\hookrightarrow} Z_1$ and $X \stackrel{\iota_2}{\hookrightarrow} Z_2$ are isometric embeddings, we denote by $Z = Z_1 \coprod_X Z_2$ the {\it metric join} of $Z_1$ and $Z_2$ with respect to $X$. The space $Z$ is the quotient of the disjoint union $Z_1 \coprod Z_2$ with respect to the equivalence relation which identifies $\iota_1(x) \sim \iota_2(x)$ for each $x \in X$, equipped with the metric
$$
d_Z(a,b) = \begin{cases} d_{Z_1}(a,b) & \mbox{if $a,b \in Z_1$,} \\
d_{Z_2}(a,b) & \mbox{if $a,b \in Z_2$,} \\
\inf_{x \in X} d_{Z_1}(\iota_1(x),a) + d_{Z_2}(\iota_2(x),b) & \mbox{if $a \in Z_1 \setminus \iota_1(X)$ and $b \in Z_2 \setminus \iota_2(X)$,} \\
\inf_{x \in X} d_{Z_1}(\iota_1(x),b) + d_{Z_2}(\iota_2(x),a) & \mbox{if $b \in Z_1 \setminus \iota_1(X)$ and $a \in Z_2 \setminus \iota_2(X)$.}
\end{cases}
$$
We leave to the reader the verification that $d_Z$ is a metric.

\begin{lem}\label{lem:independence-2}
If $Z_1$ and $Z_2$ are uniformly perfect, then $Z_1 \coprod_X Z_2$ is also uniformly perfect.
\end{lem}

Lemma \ref{lem:independence-2} follows immediately by combining \cite[Theorem 13.1.2]{fr:book} and \cite[Theorem 2.2]{fr:homogeneity-fractals}.

\begin{proof}[Proof of Lemma \ref{lem:independence-1}]
We assume that uniformly perfect metric spaces $Z$ and $W$ are given, with $X$ isometrically embedded in $Z$ and $Z$ isometrically embedded in $W$. For notational convenience we assume that $X \subset Z \subset W$. The inequality
$$
\overline\dim_\theta^{X\hookrightarrow Z} X \ge \overline\dim_\theta^{X \hookrightarrow W} X
$$
is obvious from the definitions. For the converse inequality, we fix $0<\theta'<\theta$ and will show that
$$
\overline\dim_{\theta'}^{X\hookrightarrow Z} X \le \overline\dim_\theta^{X \hookrightarrow W} X;
$$
the desired conclusion now follows by letting $\theta' \nearrow \theta$ and using the continuity of intermediate dimension.

For fixed $s>\overline\dim_\theta^{X \hookrightarrow W} X$ and $\theta' < \theta$, we let $\eps'>0$ and choose $\delta_0$ as in the definition of intermediate dimension for the choice $\eps := 2^{-s} \eps'$. We let $c_Z$ denote a constant for the uniform perfectness of $Z$ and impose an additional constraint of the form $\delta_0 \le \eta(c_Z,\theta,\theta')$ for some positive value $\eta(c_Z,\theta,\theta')$ which will be determined later in the proof. For $0<\delta<\delta_0$, let $\{A_i\}$ be a collection of subsets of $W$ which covers $X$, so that $\delta^{1/\theta} \le \diam A_i \le \delta$ for all $i$ and $\sum_i (\diam A_i)^s < \eps$.

For each $i$, fix a point $a_i \in A_i \cap Z$, let $r_i = \diam A_i$, and set $B_i := B_Z(a_i,r_i)$. Observe that $X \cap A_i \subset Z \cap A_i \subset B_i$ so $\{B_i\}$ is a collection of subsets of $Z$ which covers $X$. We verify that the sets $B_i$ satisfy the necessarily diameter constraint. First, $\diam B_i \le 2r_i = 2 \diam A_i \le 2\delta$. On the other hand,
$$
\diam B_i \ge c_Z r_i = c_Z \diam A_i \ge (2\delta)^{1/\theta'}
$$
provided
\begin{equation}\label{eq:cZ-delta}
c_Z \delta^{1/\theta} \ge (2\delta)^{1/\theta'}.
\end{equation}
The inequality in \eqref{eq:cZ-delta} determines the aforementioned constraint $\delta_0 \le \eta(c_Z,\theta,\theta')$. Finally,
$$
\sum_i (\diam B_i)^s \le 2^s \sum_i (\diam A_i)^s < \eps.
$$
This completes the proof.
\end{proof}

There is another, equally natural, way to define intermediate dimensions in arbitrary metric spaces.  This  was suggested to us by an anonymous referee but had also occurred to us when we first wrote the paper. When defining intermediate dimensions above we defined the `cost' of a cover $\{A_i\}$ by
\[
\sum_i (\diam A_i)^s
\]
with the requirement that
\[
\delta^{1/\theta} \leq \diam A_i \leq \delta
\]
for all $i$.  The problem in arbitrary metric spaces is there may not exist sets with the required constraints on diameter.  However, we could define an alternative to intermediate dimensions, temporarily denoted by  $\overline\dim_\theta^* X$, by replacing the above `cost' by
\[
\sum_i \max\{\diam A_i, \delta^{1/\theta}\}^s
\]
with the only requirement now being that $\diam A_i \leq \delta$ for all $i$.  We now prove that these two approaches are, indeed, equivalent. 

\begin{prop} \label{prop:newequiv}
Let $X$ be a totally bounded metric space. Then   $  \overline\dim_\theta X = \overline\dim_\theta^* X$ for all $0<\theta<1$. 
\end{prop}

\begin{proof}
It is clear that for  uniformly perfect metric spaces $A$,  $  \overline\dim_\theta A = \overline\dim_\theta^* A$, and that $   \overline\dim_\theta^* $ is preserved under isometry.  Let $f$ be the isometric embedding of $X$ into a uniformly perfect metric space $Z$ defining  $  \overline\dim_\theta  X$.  Then
\[
 \overline\dim_\theta X =  \overline\dim_\theta f(X) =  \overline\dim_\theta^* f(X) =  \overline\dim_\theta^* X,
\]
as required.
\end{proof}

\section{Distortion of intermediate dimension by supercritical Sobolev maps}\label{sec:Sobolev-distortion-intermediate}

In this section, we prove Theorem \ref{thm:sob_dim_dist_intermediate} on the distortion of intermediate dimension by Sobolev mappings. In order to provide context for our proof, we briefly sketch Kaufman's arguments in \cite{kau:sobolev} for the corresponding distortion estimates for Hausdorff and box-counting dimension. The proof which we give for Theorem \ref{thm:sob_dim_dist_intermediate} involves a synthesis of these two arguments, which further illustrates the manner in which intermediate dimension interpolates between these two concepts.

\subsection{Sobolev distortion of Hausdorff and box-counting dimensions}

We recall from the introduction the following statements of the Sobolev dimension distortion estimates. Let $\dim$ denote either Hausdorff dimension $\dim_H$ or box-counting dimension $\ubd$, let $E \subset \Omega$ be such that $\dim(E) = s \in (0,n)$, and let $f \in W^{1,p}(\Omega:\R^N)$, $p>n$. Then $\dim f(E) \le \tau_p(\dim E)$.

\smallskip

In the case when $\dim = \dim_H$ is Hausdorff dimension, the proof is a simple consequence of H\"older's inequality. For $\delta>0$, let $(Q_i)$ be essentially disjoint dyadic cubes in $\Omega$ such that $E \subset \cup_i Q_i$ and $\diam(Q_i) < \delta$ for each $i$. In view of the Morrey--Sobolev inequality \eqref{eq:MS} we have that $\diam f(Q_i) \le \eps$, where $\eps = C ||Df||_p \delta^{1-n/p}$. For $\sigma>s = \dim_H(E)$ we have
$$
\cH^{\tau}_\eps(f(E)) \le \sum_i \diam f(Q_i)^{\tau} \le C \sum_i \diam(Q_i)^{\tau(1-n/p)} \left( \int_{Q_i} |Df|^p \right)^{\tau/p},
$$
where $\tau = \tau_p(\sigma)$. We apply H\"older's inequality and the identity $\tau(1-\tfrac{n}{p})(\tfrac{p}{p-\tau}) = \sigma$, and take the infimum over coverings and the limit as $\delta \to 0$ to obtain $\cH^\tau(f(E)) \le C \cH^\sigma(E)^{1-\tau/p} \, ||Df||_p^\tau$. Since $\sigma>\dim_H(E)$ we conclude that $\dim_H(f(E)) \le \tau_p(s)$. This proof works for all subsets $E$ of $\Omega$.

\smallskip

We now sketch the argument for box-counting dimension, following the methodology employed by Kaufman in \cite{kau:sobolev}. Assume that $E \subset \Omega$ is a compact subset. We consider dyadic cubes $Q \in \cD_m$ as before, and we choose $M$ large enough that any cube in $\cD_m$ which meets $E$ necessarily lies in $\cD_m(\Omega)$. For a given choice of $s \in (0,n)$ and for each such $m$, set 
\begin{equation}\label{eq:r-m}
r_m := 2^{-ms/\tau_p(s)}.
\end{equation}
We declare a dyadic cube $Q$ to be {\it $r_m$-major} if $\diam f(Q) \ge r_m$ and {\it $r_m$-minor} if $\diam f(Q) < r_m$.  If $Q$ is $r_m$-minor but the parent $\widehat{Q}$ of $Q$ is $r_m$-major, we say that $Q$ is {\it $r_m$-critical}.\footnote{In \cite{kau:sobolev}, a different definition for criticality is adopted: a cube is said to be {\it ($r_m$-)critical} if it is major, but each of its descendants is minor. The definition for criticality which we adopt, referencing the parent of $Q$ rather than its descendants, clarifies the ensuing stopping time argument.}

We first show that the number $N$ of $r_m$-major cubes in $\bigcup \{ \cD_{m'}(\Omega) : m' \ge m \}$ is $\lesssim 2^{ms}$, where the implicit constant depends on $n$, $p$, and $||Df||_p$. For each $r_m$-major cube $Q \in \cD_{m'}(\Omega)$, \eqref{eq:MS} gives the estimate 
$$
2^{-ms/\tau_p(s)} = r_m \le C(n,p) 2^{-m'(1-n/p)} \left( \int_Q |Df|^p \right)^{1/p}.
$$
Rearranging this inequality yields $1 \le C(n,p)^p \, 2^{mps/\tau_p(s)} 2^{-m'(p-n)} \, \int_Q |Df|^p$; since distinct elements of $\cD_{m'}(\Omega)$ are essentially disjoint we conclude that
the number $N_{m'}$ of $r_m$-major cubes $Q \in \cD_{m'}(\Omega)$ satisfies
\begin{equation}\label{eq:Nm-prime}
N_{m'} \le C(n,p)^p \, 2^{mps/\tau_p(s)} 2^{-m'(p-n)} ||Df||_p^p.
\end{equation}
Hence 
$$
N \le \sum_{m'=m}^\infty N_{m'} \le C'(n,p) \, 2^{mps/\tau_p(s)} 2^{-m(p-n)} ||Df||_p^p \le C''(n,p) ||Df||_p^p 2^{ms}.
$$
As a simple consequence, we also note that the number of $r_m$-critical cubes in $\bigcup \{ \cD_{m'}(\Omega) : m' \ge m \}$ is also $\lesssim 2^{ms}$. In fact, the number of $r_{m+1}$-critical cubes is at most $2^n$ times the number of $r_m$-major cubes.

Now let us assume that the value $s$ as above is chosen to be the box-counting dimension of $E$. Fix $\sigma > s$. For each $m \ge M$, we can cover $E$ by a collection of cubes $(Q_i)$ in $\cD_m(\Omega)$ with cardinality $\lesssim 2^{m\sigma}$ . Then $f(E)$ is covered by the images $f(Q_i)$ of the $r_m$-minor cubes in this collection, along with the images of all $r_{m'}$-critical cubes for all $m' \ge m$. The number of such sets is $\lesssim 2^{m \sigma} + 2^{m s} \lesssim 2^{m \sigma}$ and for every one of these sets $f(Q)$ we have $\diam f(Q) \le r_m$. Replacing each $f(Q)$ by a superset whose diameter exactly equals $r_m$ yields an allowable covering for the upper box-counting dimension of $f(E)$, and
$$
\ubd f(E) \le \limsup_{m \to \infty} \frac{\log(2^{m \sigma})}{-\log r_m} = \frac{\tau_p(s) \, \sigma}{s}.
$$
Letting $\sigma \searrow s$ yields $\ubd f(E) \le \tau_p(s)$ as desired.

\subsection{Proof of Theorem \ref{thm:sob_dim_dist_intermediate}}

We now turn to the proof of the corresponding distortion estimate for intermediate dimensions, which makes use of ingredients from both of the arguments sketched in the previous subsection.

\begin{proof}[Proof of Theorem \ref{thm:sob_dim_dist_intermediate}]
Fix $s<\sigma<\sigma'<n$ and set $\tau = \tau_p(\sigma)$ and $\tau' = \tau_p(\sigma')$, where the function $\tau_p(\sigma)$ is defined as in \eqref{eq:tau-s}. We will show that
${\overline\dim}_\theta f(E) \le \tau'$; letting $\sigma'\searrow\sigma$ and $\sigma \searrow s$ completes the proof.

We take advantage of Proposition \ref{prop:intermediate-dimension-via-dyadic-cubes} to compute intermediate dimensions via dyadic coverings. Since ${\overline\dim}_\theta(E)<\sigma$ we have the following: for each $\eps>0$ there exists $\delta_0>0$ so that for each $0<\delta\le\delta_0$ there exists a covering $(Q_i)_{i \in \cI}$ of $E$ by pairwise disjoint dyadic cubes so that
\begin{equation}\label{eq:1}
\delta^{1/\theta} \le \diam(Q_i) \le \delta
\end{equation}
for each $i$ and
\begin{equation}\label{eq:2}
\sum_i (\diam Q_i)^\sigma < \eps.
\end{equation}
Moreover, each cube $Q_i$ lies in $\cD_m$ for some $m$ satisfying \eqref{eq:dyadic-cube-estimate}, hence also satisfying 
\begin{equation}\label{eq:cube-estimate}
2^{-m} \le \delta.
\end{equation}
We may assume without loss of generality that
$$
\delta_0 \le \eps \le 1
$$
and we may also assume that $m$ is so large that any cube in $\cD_m$ which meets $E$ lies in $\cD_m(\Omega)$.
We also record the following estimate for the cardinality of the index set $\cI$, which follows from \eqref{eq:1} and \eqref{eq:2}:
\begin{equation}\label{eq:cardinality-of-I}
\# \cI \le \delta^{-\sigma/\theta} \eps.
\end{equation}
Set
\begin{equation}\label{eq:definition-of-eta}
\eta := \delta^{\sigma/\tau'}
\end{equation}
and, as before, say that a dyadic cube $Q \in \cD_m$ for some $m$ satisfying \eqref{eq:cube-estimate} is
\begin{itemize}
\item {\it $\eta$-major} if $\diam f(Q) > \eta$,
\item {\it $\eta$-minor} if $\diam f(Q) \le \eta$, and
\item {\it $\eta$-critical} if $\diam f(Q) \le \eta$ and $\diam f(\widehat{Q}) > \eta$.
\end{itemize}

\begin{lem}\label{lem:size-estimate}
There exists $\tilde\gamma>0$ so that
$$
N := \# \{ Q \in \cD: \mbox{$Q\in \cD_m$ for some $2^{-m} \le \delta$, and $Q$ is $\eta$-major} \} 
$$
satisfies
$$
N \le C(n,p) ||Df||_p^p \delta^{-\sigma + \tilde\gamma}.
$$
\end{lem}

\begin{proof}[Proof of Lemma \ref{lem:size-estimate}]
We first fix $m$ satisfying \eqref{eq:cube-estimate} and estimate $N_m := \# \{Q \in \cD_m : \mbox{$Q$ is $\eta$-major} \}$. Following a similar line of reasoning as for \eqref{eq:Nm-prime}, we get
$$
N_m < C(n,p) \eta^{-p} 2^{-m(p-n)} ||Df||_p^p
$$
and hence
\begin{equation*}\begin{split}
N &\le C(n,p) \eta^{-p} ||Df||_p^p \sum_{m:2^{-m} \le \delta} 2^{-m(p-n)} \\
&\le C(n,p) \delta^{-p\sigma/\tau' + p - n} ||Df||_p^p.
\end{split}\end{equation*}
To complete the proof, we note that
\begin{equation*}\begin{split}
-\frac{p\sigma}{\tau'} + p - n 
= - \sigma + \frac{(\sigma'-\sigma)(p-n)}{\sigma'}
\end{split}\end{equation*}
and so the desired conclusion holds with
\begin{equation}\label{eq:tilde-gamma}
\tilde\gamma :=\frac{(\sigma'-\sigma)(p-n)}{\sigma'}.
\end{equation}
\end{proof}

\noindent For each $i \in \cI$ so that $Q_i$ is $\eta$-minor, we distinguish two cases:
\begin{itemize}
\item $\diam f(Q_i) \ge \eta^{1/\theta}$. Let $\cM_{\ge}$ denote the collection of all sets $f(Q_i)$ of this type.
\item $\diam f(Q_i) < \eta^{1/\theta}$. In this case we choose a set $A_i \supset f(Q_i)$ so that $\diam A_i = \eta^{1/\theta}$. Let $\cM_{<}$ denote the collection of all sets $A_i$ so selected.
\end{itemize}
Next, for each $\eta$-critical dyadic cube $Q$ with $Q \in \cD_m$ for some $m \ge -\log_2(\delta)+c$, we distinguish two cases:
\begin{itemize}
\item $\diam f(Q) \ge \eta^{1/\theta}$. Let $\cC_{\ge}$ denote the collection of all sets $f(Q)$ of this type.
\item $\diam f(Q) < \eta^{1/\theta}$. In this case we choose a set $B_Q \supset f(Q)$ so that $\diam B_Q = \eta^{1/\theta}$. Let $\cC_{<}$ denote the collection of all sets $B_Q$ so selected.
\end{itemize}
The set $f(E)$ is covered by the collection
\begin{equation}\label{eq:MMCC}
\cM_{\ge} \cup \cM_{<} \cup \cC_{\ge} \cup \cC_{<},
\end{equation}
and we proceed to estimate
\begin{equation}\label{eq:image-sum}
\sum_{S \in \cM_{\ge} \cup \cM_{<} \cup \cC_{\ge} \cup \cC_{<}} (\diam S)^{\tau'}.
\end{equation}
Observe that \eqref{eq:MMCC}
is an admissible collection of sets for estimating the value of ${\overline\dim}_\theta f(E)$; by construction, we have $\eta^{1/\theta} \le \diam(S) \le \eta$ for all such sets $S$.

The quantity in \eqref{eq:image-sum} is equal to
$$
\sum_{i:f(Q_i) \in \cM_{\ge}} (\diam f(Q_i))^{\tau'} + \bigl( \#\cM_{<} \bigr) \eta^{\tau'/\theta} + \sum_{f(Q) \in \cC_{\ge}} (\diam f(Q))^{\tau'} + \bigl( \# \cC_{<} \bigr) \eta^{\tau'/\theta}
$$
and we proceed to estimate each term in the preceding expression.

First, using \eqref{eq:MS} and then H\"older, we have
\begin{equation*}\begin{split}
&\sum_{i:f(Q_i) \in \cM_{\ge}} (\diam f(Q_i))^{\tau'} \\
&\quad \le C(n,p) \sum_{i:f(Q_i) \in \cM_{\ge}} (\diam Q_i)^{\tau'(1-n/p)} \left( \int_{Q_i} |Df|^p \right)^{\tau'/p} \\
&\quad \le C(n,p) \left( \sum_{i:f(Q_i) \in \cM_{\ge}} (\diam Q_i)^{\tau'(1-n/p)\tfrac{p}{p-\tau'}} \right)^{1-\tau'/p} \left( \sum_{i:f(Q_i) \in \cM_{\ge}} \int_{Q_i} |Df|^p \right)^{\tau'/p} \\
&\quad = C(n,p) \left( \sum_{i:f(Q_i) \in \cM_{\ge}} (\diam Q_i)^{\sigma'} \right)^{1-\tau'/p} \left( \int_{\cup_{i:f(Q_i) \in \cM_{\ge}}} \int_{Q_i} |Df|^p \right)^{\tau'/p} \\
&\quad \le C(n,p) \delta^{(\sigma'-\sigma)(1-\tau'/p)} \left( \sum_{i:f(Q_i) \in \cM_{\ge}} (\diam Q_i)^{\sigma} \right)^{1-\tau'/p} 
||Df||_p^{\tau'} \\
&\quad \le C(n,p) ||Df||_p^{\tau'} \delta^{(\sigma'-\sigma)(1-\tau'/p)} \eps^{1-\tau'/p}.
\end{split}\end{equation*}
Since $\delta_0 \le 1$ we may disregard the factor involving a power of $\delta$ and conclude that
$$
\sum_{i:f(Q_i) \in \cM_{\ge}} (\diam f(Q_i))^{\tau'} \le C(n,p) ||Df||_p^{\tau'} \eps^{1-\tau'/p}.
$$
Note that $1-\tau'/p>0$ since $\tau'<n<p$, where the fact that $\tau'<n$ follows from the assumption $\sigma'<n$.

Next, we estimate
$$
\bigl( \#\cM_{<} \bigr) \eta^{\tau'/\theta} \le \eps
$$
using \eqref{eq:cardinality-of-I} and \eqref{eq:definition-of-eta}.

The collection $\cC_{\ge} \cup \cC_{<}$ is contained in the set of all $\eta$-critical cubes, whose cardinality is at most $2^n$ times the cardinality of the set of all $\eta$-major cubes. Using Lemma \ref{lem:size-estimate} we estimate
$$
\bigl( \# \cC_{<} \bigr) \eta^{\tau'/\theta} \le C \delta^{-\sigma + \tfrac{\sigma}{\theta} + \tilde\gamma};
$$
note that $-\sigma + \tfrac{\sigma}{\theta} + \tilde\gamma > 0$.

Finally, we estimate
$$
\sum_{f(Q) \in \cC_{\ge}} (\diam f(Q))^{\tau'} \le \big( \# \cC_{\ge} \bigr) \eta^{\tau'} 
$$
where we use the fact that each cube $Q$ with $f(Q) \in \cC_{\ge}$ is $\eta$-critical. Again using Lemma \ref{lem:size-estimate} and \eqref{eq:definition-of-eta} we estimate
$$
\bigl( \# \cC_{\ge} \bigr) \eta^{\tau'} \le \delta^{\tilde\gamma}.
$$
Observe that the rationale for making the double choice of $\sigma>s$ and $\sigma'>\sigma$ comes from the estimate for this term; since $\tilde\gamma>0$ we are also able to make this term arbitrarily small.

Combining all four of the preceding estimates and using the fact that $\delta < \delta_0 \le \eps$, we conclude that
$$
\sum_{S \in \cM_{\ge} \cup \cM_{<} \cup \cC_{\ge} \cup \cC_{<}} (\diam S)^{\tau'} \le 
C(n,p,||Df||_p) \left( \eps^{1-\tau'/p} + \eps + \eps^{\sigma(1/\theta-1) + \tilde\gamma} + \eps^{\tilde\gamma} \right).
$$

\medskip

We now show that ${\overline\dim}_\theta f(E) \le \tau'$. First, fix an arbitrary $\eps'>0$. Next, choose $0<\eps\le 1$ so that
$$
C(n,p,||Df||_p) \left( \eps^{1-\tau'/p} + \eps + \eps^{\sigma(1/\theta-1) + \tilde\gamma} + \eps^{\tilde\gamma} \right) < \eps',
$$
where $\tilde\gamma$ is defined as in \eqref{eq:tilde-gamma}. In view of the assumption ${\overline\dim}_\theta E < \sigma$ we have a value $\delta_0 \le \eps$ so that for any $0<\delta \le \delta_0$ there exists an essentially disjoint covering of $E$ by dyadic cubes $(Q_i)$ so that \eqref{eq:1} and \eqref{eq:2} hold. Set
$$
\eta_0 :=  \delta_0^{\sigma/\tau'}
$$
and consider values $0<\eta \le \eta_0$. Each such $\eta$ has the form $\eta = \delta^{\sigma/\tau'}$ for some $\delta \le \delta_0$. The argument given above yields a covering $\{S:S \in \cM_{\ge} \cup \cM_{<} \cup \cC_{\ge} \cup \cC_{<}\}$ for $f(E)$ by sets with diameter in the range $[\eta^{1/\theta},\eta]$, and for which $\sum_S (\diam S)^{\tau'} < \eps'$. It follows that
$$
{\overline\dim}_\theta f(E) \le \tau'
$$
as desired.
\end{proof}

\begin{rem}\label{rem:assouad-case}
As mentioned in the introduction, analogs for \eqref{eq:qc-dim-H-bound2} and \eqref{eq:qc-dim-B-bound} for the Assouad dimension and Assouad spectrum were obtained in \cite{ct:qc-assouad-spectrum}, building on prior work in \cite{tys:assouad}.
More precisely, it was shown in \cite{ct:qc-assouad-spectrum} that
\begin{equation}\label{eq:qc-dim-A-bound}
\alpha(p^\RH_n(K)) \Phi(\dim_A E) \le \Phi(\dim_A f(E)) \le \alpha(p^\RH_n(K^{n-1}))^{-1} \Phi(\dim_A E)
\end{equation}
for compact sets $E \subset \Omega$ and $f:\Omega \to \Omega'$ $K$-quasiconformal. A distinctive feature of \eqref{eq:qc-dim-A-bound} relative to its predecessors is the appearance of the {\it reverse H\"older higher integrability exponent} $p^\RH_n(K)$.
In the course of his proof of the higher integrability theorem, Gehring proved that quasiconformal mappings enjoy the quantitative and scale-invariant reverse H\"older property: if $f:\Omega \to \Omega'$ is $K$-quasiconformal, then for some choice of $p>n$ depending only on $n$ and $K$ there exists $C>0$ so that the {\it reverse H\"older inequality}
\begin{equation}\label{eq:rhi}
\left( \frac1{|Q|} \int_Q |Df|^p \right)^{1/p} \le C \left( \frac1{|2Q|} \int_{2Q} |Df|^n \right)^{1/n}
\end{equation}
holds for cubes $Q$ with $\diam(Q) < \dist(Q,\partial\Omega)$ and $\diam(f(2Q)) < \dist(f(2Q),\partial\Omega')$. The {\it reverse H\"older higher integrability exponent} $p^\RH_n(K)$, 
introduced in \cite{ct:qc-assouad-spectrum}, is the supremum of all $p>n$ so that every $K$-quasiconformal map $f$ between $n$-dimensional domains satisfies \eqref{eq:rhi}. 
Obviously, $p^\RH_n(K) \le p^\Sob_n(K)$ and it is natural to conjecture a strong form of the higher integrability conjecture, namely, that $p^\RH_n(K)$ should equal $\tfrac{nK}{K-1}$ in all dimensions $n \ge 3$.\footnote{Astala's proof in fact shows that $p^\RH_2(K) = p^\Sob_2(K) = \tfrac{2K}{K-1}$ for all $K$.} The appearance of $p^\RH_n(K)$ in the Assouad-type distortion estimates is natural: Assouad dimension is a quantitative and scale-invariant notion which accounts for the behavior of the set at all locations and scales, so it is not unexpected that the higher integrability behavior of the map also needs to be controlled across all locations and scales. 

In \cite{ct:qc-assouad-spectrum} the authors also provided a corresponding conclusion for the Assouad spectrum: for any compact $E \subset \Omega$, $f:\Omega \to \Omega'$ $K$-quasiconformal, and $t>0$,
\begin{equation}\begin{split}\label{eq:qc-dim-A-theta-bound}
&\alpha(p^\RH_n(K)) \Phi(\dim_{A,\reg}^{K/(K+t)} E) \le \Phi(\dim_{A,\reg}^{1/(1+t)} f(E)) \\
&\qquad \qquad \le \alpha(p^\RH_n(K^{n-1}))^{-1} \Phi(\dim_{A,\reg}^{1/(1+Kt)} E).
\end{split}\end{equation}
An interesting novelty here is that the dilatation parameter $K$ features not only in the comparison constants, but also in the choice of the Assouad spectrum parameter in the source relative to the choice in the target. This feature is also present in \cite{fy:assouad-spectrum} in the context of distortion bounds for the Assouad dimension under bi-H\"older mappings.  We emphasise again that the results in \cite{ct:qc-assouad-spectrum} for Assouad spectrum and the Assouad dimension only hold for quasiconformal maps, and not for general (non-injective) Sobolev maps. 

Table \ref{tab:info-table} summarizes mapping properties of the various notions of metric dimension discussed in this paper.

\smallskip

\begin{center}\label{tab:info-table}
\begin{tabular}{||l|c|c|c|c|c||}
\hline\hline
Dimension & $\dim_H$ & ${\overline\dim}_\theta$ & ${\overline\dim}_B$ & $\dim_A^\theta$ & $\dim_A$ \\ [.1cm]
\hline\hline
Lipschitz contractivity & \checkmark & \checkmark & \checkmark & ${\mathcal X}$ & ${\mathcal X}$ \\  [.1cm]
\hline
Bi-Lipschitz invariance & \checkmark & \checkmark & \checkmark & \checkmark & \checkmark \\  [.1cm]
\hline
H\"older quasi-contractivity & \checkmark & \checkmark & \checkmark & ${\mathcal X}$ & ${\mathcal X}$ \\  [.1cm]
\hline
Sobolev distortion bounds & \checkmark & \checkmark & \checkmark & ${\mathcal X}$ & ${\mathcal X}$ \\  [.1cm]
\hline
QC distortion bounds & \checkmark & \checkmark & \checkmark & \checkmark & \checkmark \\  [.1cm]
\hline
Higher integrability exponent & $p^\Sob$ & $p^\Sob$ & $p^\Sob$ & $p^\RH$ & $p^\RH$ \\ [.1cm]
\hline\hline
\end{tabular}
\end{center}

\end{rem}

\section{Applications to quasiconformal classification}\label{sec:QC-classification}

In this section, we present several examples illustrating the use of intermediate dimensions in the quasiconformal classification problem.

Proposition \ref{thm:application2} indicated that Euclidean sets $E$ with quasi-Hausdorff dimension zero can be quasiconformally distinguished based on the convergence/divergence of $\lim_{\theta \to 0} \tfrac{d}{d\theta} {\overline\dim}_\theta E$. Moreover, even if both sets have the same convergence or divergence behaviour for this quantity, we can still draw conclusions about quasiconformal equivalence, however, those conclusions are now dilatation-dependent.

\begin{rem}\label{rem:K0}
Let $E,F \subset \R^n$, $n \ge 2$, be bounded sets with $\dim_{qH} E = \dim_{qH} F = 0$ and
\begin{equation}\label{eq:K0}
\liminf_{\theta \searrow 0^+} \frac{{\overline\dim}_\theta E}{{\overline\dim}_\theta F}  = \delta \in (0,1).
\end{equation}
Then no $K$-quasiconformal map $f:\R^n \to \R^n$ with $K<\delta^{-1}$ and $f(E) = F$ exists. Indeed, such a map would necessarily satisfy
${\overline\dim}_\theta F \le K {\overline\dim}_\theta E$ for all $0<\theta \le 1$ by the H\"older distortion estimate \eqref{eq:Holder-bound}. From \eqref{eq:K0} we deduce that $K \ge \delta^{-1}$.
\end{rem}

\begin{example}\label{ex:qc-equiv-of-e-s-sets}
Let
\begin{equation}\label{eq:ES}
E_s := \{ n^{-s} : n \in \N \} \cup \{0\}, \qquad s>0.
\end{equation}
Then $\dim_H E_s = 0$, ${\overline\dim}_B E_s = 1/(1+s)$, and $\dim_A(E_s) = 1$, whilst
\begin{equation}\label{eq:Fp-intermediate}
{\overline\dim}_\theta E_s  = \frac{\theta}{\theta+s}, \qquad \mbox{for $0<\theta<1$,}
\end{equation}
and
\begin{equation}\label{eq:Fp-Assouad-spectrum}
\dim_A^\theta E_s  = \min \bigg\{ 1 , \frac{1}{(1+s)(1-\theta)} \bigg\}, \qquad \mbox{for $0<\theta<1$.}
\end{equation}
See e.g.\ \cite[\S 2.1]{fr:interpolating-survey}. For $0<r \le s$, set $K = s/r$. The radial stretch map $f(z) = |z|^{1/K - 1}z$ satisfies $f(E_s) = E_r$ and is $K$-quasiconformal. By Remark \ref{rem:K0}, no map with a smaller quasiconformal dilatation can send $E_s$ to $E_r$. Note that the value of $\delta$ in this case is $r/s$. Observe that we cannot obtain this conclusion by considering instead the Assouad spectra of $E_r$ and $E_s$ and appealing to \cite{ct:qc-assouad-spectrum}. Roughly speaking, this is because the Assouad spectra of $E_s$ lie in the compact subinterval $[\tfrac1{1+s},1]$ of $(0,2)$. In fact, a (rather technical) computation reveals that using \eqref{eq:Fp-Assouad-spectrum} in conjunction with the mapping estimate \eqref{eq:qc-dim-A-theta-bound} fails to yield the sharp lower bound $K \ge \tfrac{s}{r}$, regardless of the choice of $t$.
\end{example}

In considering the statement of Corollary \ref{cor:application1}, the reader may wonder whether the equality of box-counting and Assouad dimension itself could be a quasiconformal invariant. In fact, this is not the case, as the following example indicates.

\begin{example}
For each $0<t<1$, there exists $E \subset \R^2$ compact with $t<{\overline\dim}_B E < \dim_A E$ and a quasiconformal map $f:\R^2 \to \R^2$ so that $f(E)$ is the standard Cantor set. In fact, the existence of such a set is an immediate consequence of Banaji and Rutar's attainability theorem for intermediate dimensions \cite{br:attainability} and MacManus' classification of planar quasiconformal images of the Cantor set \cite{McM}.

Given $0<t<1$, choose $s$ satisfying $t^2<s<t$ and note that $\tfrac{t}{s}<\tfrac{1-t}{t-s}$. Choose $\lambda>0$ and $\alpha<1$ so that $0<\lambda<s<t<\alpha<1$ and
$$
\frac{t-\lambda}{s-\lambda} < \frac{\alpha-t}{t-s}.
$$ 
Then $0<(s-\lambda)(\alpha-s) - 2(s-\lambda)(t-s) - (t-s)^2$ and hence
$$
t - s < \frac{\bigl((s-\lambda)+(t-s)\theta\bigr)\bigl((\alpha-s)-(t-s)\theta\bigr)}{(\alpha-\lambda)\theta} \qquad \forall \, 0<\theta<1.
$$
Consequently, the function $h:[0,1] \to \R$ defined by $h(\theta) = s+(t-s)\theta$ lies in the class $\cH(\lambda,\alpha)$ defined in \cite[Definition 1.3]{br:attainability}. By \cite[Theorem B]{br:attainability}, there exists a compact set $E \subset \R^2$ with lower dimension $\lambda$, Hausdorff dimension $s$, box-counting dimension $t$, and Assouad dimension $\alpha$.\footnote{The set $E$ moreover has ${\overline\dim}_\theta E = h(\theta)$ for all $0<\theta<1$, but we will not need this fact.} Since any set with Assouad dimension strictly less than one is uniformly disconnected \cite[Lemma 15.2]{ds:fractured-fractals}, the set $E$ so constructed is a compact, uniformly perfect, and uniformly disconnected subset of $\R^2$. By \cite[Theorem 3]{McM}, there exists $f:\R^2 \to \R^2$ quasiconformal so that $f(E)$ is the standard Cantor set.
\end{example}

Next, we give a quasiconformal classification result for which the full strength of our quasiconformal distortion estimate \eqref{eq:intermediate-quasiconformal} is needed.

\begin{example}\label{ex:BMcarpets-after-BK}
In \cite[Example 2.14]{bk:intermediate-BM-carpets}, the authors consider a pair of Bedford--McMullen carpets $E=E_\cF$ and $E'=E'_{\cF'}$ with the following data. Both $\cF$ and $\cF'$ are constructed on a grid of $m\times n = 32\times 243$ rectangles, with every column containing at least one entry. The IFS $\cF$ contains two columns with $27$ entries each, $11$ columns with $3$ entries each, and $19$ columns with a single entry each, whilst the IFS $\cF'$ contains one column with $27$ entries, six columns with nine entries each, and $25$ columns with a single entry each. Note that both IFS are comprised of $106$ elements, and it is possible to choose the rectangles in such a way that $E$ and $E'$ are both totally disconnected, and hence homeomorphic. 

From \eqref{eq:BM-dim-H} and \eqref{eq:BM-dim-B} we find that
\begin{equation}\label{eq:BMH}
\dim_H E = \dim_H E' = \log_{32}(57) \approx 1.16658
\end{equation}
and
\begin{equation}\label{eq:BMB}
{\overline\dim}_B E = {\overline\dim}_B E' = 1 - \log_3(2) + \log_{243}(106) \approx 1.21804.
\end{equation}
On the other hand, the intermediate dimension functions for $E$ and $E'$ are not the same, hence these two sets are not bi-Lipschitz equivalent. Moreover, considering how the ratio of dimensions varies with $\theta$, the authors of \cite{bk:intermediate-BM-carpets} find that the extremal choice of $\theta$ is $\theta_0 := \gamma^{-2} = (\log_3(2))^2 \approx 0.398$, and using this choice of $\theta_0$ they show that if $f:\R^2 \to \R^2$ is a homeomorphism with $f(E) = E'$ and $f^{-1}|_{E'}$ $\alpha$-H\"older continuous, then necessarily
\begin{equation}\label{eq:alpha-upper-bound}
\alpha \le \frac{d'}{d} < 1 - \delta = 0.9995.
\end{equation}
where $d = \dim_{\theta_0} E$ and $d' = \dim_{\theta_0} E'$. 

Using only the local $\tfrac1K$-H\"older continuity of $K$-quasiconformal maps, we immediately conclude that any $K$-QC homeomorphism of $\R^2$ sending $E$ to $E'$ must have
\begin{equation}\label{eq:Kold}
K > \frac1{1-\delta} \approx 1.0005.
\end{equation}
We obtain a better result by appealing to Corollary \ref{cor:qc_dim_dist_intermediate}. Using the exact value $p^\Sob_2(K) = 2K/(K-1)$
in \eqref{eq:intermediate-quasiconformal} and solving for $K$, we find
$$
K \ge \frac{d(2-d')}{d'(2-d)}.
$$
From \eqref{eq:alpha-upper-bound} we deduce that $d-d' > \delta \cdot d$, and hence
$$
\frac{d(2-d')}{d'(2-d)} = 1 + 2 \frac{d-d'}{2-d} > 1 + 2 \delta \frac{d}{2-d}.
$$
Since $d \ge \dim_H E$ and the function $t\mapsto t/(2-t)$ is increasing on the interval $(0,2)$, we conclude using \eqref{eq:BMH} that
\begin{equation}\label{eq:Knew}
K > 1 + 2 \delta \frac{\dim_H E}{2-\dim_H E} \approx 1.0014.
\end{equation}
\end{example}

\section{Quasiconformal reduction of intermediate dimension and conformal dimension}\label{sec:qc-lowering}

In this section, we prove Theorem \ref{th:main-1} and its application in Theorem \ref{th:applic-1}. We also present examples of Euclidean sets which verify the assumptions of Theorem \ref{th:applic-1}.

Theorem \ref{th:main-1}, along with its variant for Euclidean sets and global quasiconformal box-counting dimension described in the introduction, follow from the next theorem. Indeed, as observed in \cite{Kov}, every nonseparable metric space has infinite conformal Hausdorff dimension, so it suffices to restrict attention to separable metric spaces. Moreover, every separable metric space embeds isometrically into a separable Banach space. For example, the Banach--Fr\'echet--Mazur theorem \cite[Chapter XI, section 8]{Ban} exhibits an isometric embedding of any such space into $C([0,1])$, the Banach space of continuous real-valued functions on $[0,1]$ equipped with the uniform norm.

\begin{thm}\label{thm:lowering-intermediate}
Let $V$ be any separable Banach space, and let $E \subset V$ be a totally bounded set satisfying ${\overline\dim}_\theta E < 1$ for some $0<\theta<1$. Then for any $\eps>0$ there exists a quasisymmetric homeomorphism $f:V \to V$ such that ${\overline\dim}_\theta f(E) < \eps$.
\end{thm}

In the proof, we use the following result, see \cite[Corollary 4.2]{Kov}. Theorem \ref{thm:lowering-intermediate} is a straightforward consequence of Proposition \ref{prop:Kov}, apart from some technical but essentially elementary manipulations involving the interpolation parameter $\theta$.

\begin{prop}[Kovalev]\label{prop:Kov}
Let $V$ be a separable Banach space, $E \subset V$, and $0<s<1$. Assume that there exist balls $\{B(x_{ij},r_{ij})\}$, $i,j \ge 1$, so that for each $i$,
$$
E \subset \cup_j B(x_{ij},r_{ij}) \quad \mbox{and} \quad \sum_j (1+||x_{ij}||) r_{ij}^s < 2^{-i}.
$$
Then for any $\eps>0$ there exists a quasisymmetric homeomorphism $f:V \to V$ so that
$$
f(E) \subset \bigcap_i \bigcup_j B(f(x_{ij}), R_{ij})
$$
where $R_{ij} = C(\eps,s)r_{ij}^{s/\eps}$.
\end{prop}

\subsection{Proof of Theorem \ref{thm:lowering-intermediate} and Theorem \ref{th:main-1}}
Let $E \subset V$ be a totally bounded set such that ${\overline\dim}_\theta E < s$ for some $s<1$. Then $E \subset B(0,R)$ for some $R>0$, and for each integer $i$ there exists $\delta_0>0$ such that for all $0<\delta\le\delta_0$ there exists a collection of balls $\{B(x_{ij},r_{ij}\}_{j\ge 1}$ which covers $E$ and has $\delta^{1/\theta} \le r_{ij} \le \delta$ and $\sum_j r_{ij}^s < 2^{-i-1}/(1+R)$.

For fixed $0<\eps$, Proposition \ref{prop:Kov} implies that there exists a quasisymmetric map $f:V \to V$ so that $f(E) \subset \cap_i \cup_j B(f(x_{ij}), R_{ij})$ where $R_{ij} = C(\eps,s)r_{ij}^{s/\eps}$. Fixing an auxiliary integer $\nu>\tfrac1\theta$ we choose $\delta_0$ as above and set 
$$
\delta_0' := \min\{\delta_0^{s/\eps} C(\eps,s),C(\eps,s)^{-(1-\theta)(\nu\theta-1)}\}.
$$
For an arbitrary $0<\delta' \le \delta_0'$, define $\delta$ by the equation $\delta' = C(\eps,s)\delta^{s/\eps}$. Associated to this value of $\delta$ we have a collection of balls $\{B(x_{ij},r_{ij})\}_{j\ge 1}$ which covers $E$ and has the two properties listed above. Note that the image covering $\{B(f(x_{ij}), R_{ij})\}_{j \ge 1}$ also satisfies
\begin{equation}\label{eq:intermediate-image}
C(\eps,s) \delta^{\tfrac{s}{\eps\theta}} \le R_{ij} \le C(\eps,s)\delta^{\tfrac{s}{\eps}}
\end{equation}
and
\begin{equation}\label{eq:intermediate-image-sum}
\sum_j R_{ij}^{\eps} = C(\eps,s) \sum_j r_{ij}^s < C(\eps,s) \frac{2^{-i-1}}{1+R} \to 0 \quad \mbox{as $i \to \infty$.}
\end{equation}
In view of the choice of $\delta_0'$ above, we have
\begin{equation}\label{eq:C-delta-1}
C(\eps,s)\delta^{s/\eps} = \delta' \leq \delta_0' \le C(\eps,s)^{-(1-\theta)(\nu\theta-1)}.
\end{equation}
From \eqref{eq:C-delta-1} we obtain
$$
\delta^{s/\eps} \le C(\eps,s)^{\nu \theta (\theta - 1 - \tfrac1\nu)}
$$
and hence
$$
C(\eps,s) \delta^{s/\eps} \le C(\eps,s)^{\theta - \tfrac1\nu} \delta^{(s/\eps) \cdot (\theta-\tfrac1\nu)/\theta},
$$
which in turn implies that
\begin{equation}\label{eq:C-delta-2}
\left( C(\eps,s) \delta^{s/\eps} \right)^{\tfrac1{\theta-\nu^{-1}}} \le C(\eps,s) \delta^{\tfrac{s}{\eps\theta}}.
\end{equation}
Applying \eqref{eq:C-delta-2} with $\theta' = \theta - \tfrac1\nu$ and recalling \eqref{eq:intermediate-image} we have
\begin{equation}\label{eq:intermediate-image-2}
\left( C(\eps,s) \delta^{s/\eps} \right)^{1/\theta'} \le R_{ij} \le C(\eps,s)\delta^{s/\eps}.
\end{equation}
From \eqref{eq:intermediate-image-2} and \eqref{eq:intermediate-image-sum}, after passing to the limit as $i \to \infty$, we conclude that ${\overline\dim}_{\theta-\nu^{-1}} f(E) \le \eps$. Using continuity of intermediate dimension, it then follows that ${\overline\dim}_\theta f(E) \le \eps$. The proof is complete. 

\subsection{Proof of Theorem \ref{th:applic-1}}
Let $(X,d)$ be a doubling metric space with $\dim_{qH} X < 1$ and let $\theta < 1$. Choose $\tilde\theta>0$ so that ${\overline\dim}_{\tilde\theta} X < 1$. Assume that
\begin{equation}\label{eq:cdimAthetaX}
C\dim_A^\theta X \ge \delta > 0.
\end{equation}
Set $\eps = \tilde\theta(1-\theta) \delta$, and, using Theorem \ref{thm:lowering-intermediate},  choose a quasisymmetric homeomorphism $f:X \to Y$ onto a metric space $(Y,d')$ so that ${\overline\dim}_{\tilde\theta} Y < \eps$. Using \eqref{eq:assouad-spectrum-inequality} and \eqref{eq:BR-cor} we obtain
$$
\dim_A^\theta Y \le \frac{\overline\dim_B Y}{1-\theta} \le \frac{\overline\dim_{\tilde\theta} Y}{\tilde\theta(1-\theta)} < \delta
$$
which contradicts \eqref{eq:cdimAthetaX}.
 
\subsection{Further examples and applications}

\begin{rem}
We emphasise that the conclusion in Theorem \ref{th:applic-1} only states that the conformal Assouad spectrum is equal to zero for each $\theta<1$. It does not say anything about the conformal quasi-Assouad dimension. As remarked in the introduction, although we know that $\dim_A^\theta X$ increases to $\dim_{qA} X$ as $\theta \nearrow 1$, it is not clear how the infimum over all quasisymmetric images interacts with the limit in the parameter $\theta$. We comment further on this issue in Question \ref{q:continuity-of-conformal-Assouad-spectrum}.
\end{rem}

\begin{example}\label{ex:binder-hakobyan-li}
Let $E \subset \R^2$ be a Bedford--McMullen carpet as in section \ref{subsec:bm-carpets}. If $M=m$ (i.e., every column contains at least one rectangle), we say that $E$ is {\em full}. Next, if $N_i = N/M$ is constant for $i=1,\ldots,M$ (i.e., every nonempty column contains an equal number of rectangles), we say that $E$ is a {\em Bedford--McMullen carpet with uniform fibres}. It can easily be deduced from the formulas in section \ref{subsec:bm-carpets} that $\dim_H E = \dim_A E$, and hence all of the notions of dimension considered in this paper agree, for carpets with uniform fibres. 

Binder, Hakobyan and Li \cite{BHL} have shown that full Bedford--McMullen carpets with uniform fibres are minimal for conformal Hausdorff dimension. Consider a full Bedford--McMullen carpet with uniform singleton fibres ($N_i = 1$ for all $i=1,\ldots,M$). Then $\dim_A E = 1$. Moreover, we may choose the rectangles so that $E$ is totally disconnected. The result of \cite{BHL} implies that $C\dim_H E = 1$, and hence all notions of dimension and conformal dimension considered in this paper are equal to one. See also Question \ref{q:bhl-consequence}.
\end{example}

\begin{rem}
In \cite[Corollary 2]{bishop-tyson}, Bishop and the second named author provide an example, for each $1 \le s < n$, of a totally disconnected set $E \subset \R^n$ such that $C\dim_H E = \dim_H E = s$.
\end{rem}

\begin{example}\label{ex:percolation-example}
Consider samples $F$ of the Mandelbrot percolation process as in section \ref{subsec:percolation}. As noted in that section, the almost sure values of the dimensions of $F$ are $n - \log_M(1/p)$ (for Hausdorff dimension, intermediate dimension, box-counting dimension, and Assouad spectrum) and $n$ (for Assouad dimension). Rossi and Suomala \cite{rossi-suomala} showed that $F$ has an almost sure constant value for conformal Hausdorff dimension, and that
$$
C \dim_H F < \dim_H F \quad \mbox{a.s.}
$$
regardless of the value of $\dim_H F$. If $p<M^{1-n}$ then $\dim_H F < 1$ and one immediately has $C\dim_H F = 0$ by Kovalev's theorem. Moreover, since $\dim_{qH} F = \dim_H F$, we obtain as a consequence of Theorem \ref{th:applic-1} that $C\dim_A^\theta F = 0$ for each $\theta<1$. (Note that this conclusion was already known as a result of Kovalev's theorem in the box-counting dimension context, in combination with \eqref{eq:assouad-spectrum-inequality}.) On the other hand, $C\dim_A F = n$ a.s. See Questions \ref{q:percolation-1} and \ref{q:percolation-2} for some related open questions.
\end{example}

\begin{example}\label{ex:product-f-p-example}
Let $Z \subset \R^{n-1}$ be any set with $\dim_{qH} Z < 1$ and ${\overline\dim}_B Z \ge \tfrac{s}{s+1}$ for some $s>0$. Then $E := Z \times E_s$ has box-counting dimension at least one and has vanishing conformal Assouad spectrum.

\smallskip

For the proof, we observe that
\begin{equation}\label{eq:dim-B-product}
{\overline\dim}_B E = {\overline\dim}_B Z + \frac{1}{s+1},
\end{equation}
\begin{equation}\label{eq:dim-theta-product}
{\overline\dim}_\theta E \le {\overline\dim}_\theta Z + \frac{\theta}{s+\theta},
\end{equation}
and
\begin{equation}\label{eq:dim-qH-product}
\dim_{qH} E \le \dim_{qH} Z.
\end{equation}
To see why \eqref{eq:dim-B-product} holds true, we note that the upper bound follows from \cite[(7.5)]{fal:fractal-geometry}, whilst the lower bound follows from \cite[Theorem 2.4]{RS}. For \eqref{eq:dim-theta-product}, we appeal to \cite[Proposition 2.5]{ffk:intermediate}. Finally, \eqref{eq:dim-qH-product} follows from \eqref{eq:dim-theta-product} in the limit as $\theta \searrow 0$.

Using \eqref{eq:dim-B-product} and the assumption ${\overline\dim}_B Z \ge \tfrac{s}{s+1}$, we conclude that ${\overline\dim_B} E \ge 1$. On the other hand from \eqref{eq:dim-qH-product} we deduce that $\dim_{qH} E < 1$. Theorem \ref{th:applic-1} implies that $C\dim_A^\theta E = 0$ for every $0<\theta<1$. 
\end{example}

\begin{rem}
Under the stated assumptions in Example \ref{ex:product-f-p-example}, the set $E = Z \times E_s$ admits among its weak tangents any space of the form $Z_\infty \times \R$, where $Z_\infty$ is a weak tangent of $Z$. It follows from \cite[Proposition 4.1.11 and Proposition 6.1.7]{MT} that $E$ is minimal for conformal Assouad dimension.
\end{rem}

\begin{rem}
In Example \ref{ex:product-f-p-example}, we may choose $Z \subset \R^{n-1}$ to be any Ahlfors $t$-regular set with $\tfrac{s}{s+1} \le t < 1$. Then ${\overline\dim}_B Z = {\overline\dim}_\theta Z = t$ for all $0<\theta<1$ and we conclude that $Z \times E_s$ has vanishing conformal Assouad spectrum.
\end{rem}

\begin{rem}
Suppose that $0<s\le 1$. Then $\tfrac{s}{s+1} \le \tfrac{1}{s+1}$ and so the set $Z = E_s$ itself satisfies the conditions in Example \ref{ex:product-f-p-example}. We conclude that $E_s \times E_s$ has vanishing conformal Assouad spectrum, for any $0<s \le 1$.
\end{rem}

\begin{rem}
On the other hand, for any $s>0$ let $Z = E_{1/s}$. Then $\dim_{qH} Z = 0$ and $\dim_B Z = \tfrac{s}{s+1}$. We conclude that $E_{1/s} \times E_s$ has vanishing conformal Assouad spectrum, for any $s>0$.
\end{rem}

\begin{example}\label{ex:CIFS}
Next, we give an example derived from the study of intermediate dimensions of invariant sets of (possibly infinite) conformal iterated function systems (CIFS). Let $n \ge 2$, set $d = n-1$, and let $0<s<\infty$. Define
$$
G_{s,d} := \left\{ \frac{x}{|x|^2} \, : \, x \in \{1^s,2^s,3^s,\ldots\}^d \right\} \subset \R^{n-1}.
$$
According to \cite[Proposition 3.8]{BF}, $\overline\dim_\theta G_{s,d} = \tfrac{d\theta}{s+\theta}$ for all $0<\theta<1$. In particular, $\dim_{qH} G_{s,d} = 0$ and $\overline\dim_B G_{s,d} = \tfrac{d}{s+1} > \tfrac{t}{t+1}$ provided that
\begin{equation}\label{eq:s-p}
s < (n-2) + \frac{n-1}{t}.
\end{equation}
We conclude that $G_{s,n-1} \times E_t \subset \R^n$ has vanishing conformal Assouad spectrum provided that \eqref{eq:s-p} holds.

\medskip

Furthermore, let $0<s\le n-2$ and consider a set of contracting similarities on $\R^{n-1}$ with fixed points in the set $G_{s,d}$ defined above. Assume that the contraction ratios are so small that the system defines a CIFS whose limit set $Z$ satisfies $\dim_H Z < 1$. Then $\overline\dim_\theta Z = \max \{ \dim_H Z, \tfrac{d\theta}{s+\theta} \}$ by \cite[Theorem 3.5]{BF} which implies that $\theta \mapsto \dim_\theta Z$ is continuous at $\theta = 0$ and $\dim_{qH} Z = \dim_H Z < 1$. Moreover, $\overline\dim_B Z = \tfrac{d}{s+1} \ge 1$ since $s \le n-2$. We conclude that $Z \times E_t$ has vanishing conformal Assouad spectrum for any CIFS limit set $Z$ as above.
\end{example}

\section{Sharpness of dimension distortion estimates}\label{sec:sharpness-section}

In this section, we demonstrate the sharpness of the estimates in \eqref{eq:intermediate-quasiconformal-2d} and \eqref{eq:intermediate-Sobolev}. An important aspect of these constructions is that they apply simultaneously for all values of the intermediate dimension parameter $\theta$. 

\subsection{Sharpness of the planar quasiconformal distortion estimates for intermediate dimension}

This subsection is devoted to the proof of the following theorem, which shows that the estimate in \eqref{eq:intermediate-quasiconformal-2d} is sharp in the strong sense described in the previous paragraph.

\begin{thm}\label{thm:theta-qc-sharpness}
For each $K \ge 1$ there exists a compact set $E \subset \C$ with $\dim_H E = 0$ and a $K$-quasiconformal mapping $g: \C \to \C$ so that
$$
\frac1{\overline\dim_\theta g(E)} - \frac12 = \frac1K \left( \frac1{\overline\dim_\theta E} - \frac12 \right), \qquad \forall \, 0<\theta\le 1.
$$
\end{thm}

Recall that the sharp $K$-dependent distortion estimates for Hausdorff dimension under planar quasiconformal mappings were obtained by Astala \cite{ast:2d-higher}. An example demonstrating the sharpness of those estimates can be found in \cite[Theorem 1.4]{ast:2d-higher}. It is interesting to note that the example which we provide here in the intermediate dimensions case is conceptually much simpler than the one given in \cite{ast:2d-higher}.

We first describe the construction of the relevant set $E$. For this construction, the restriction to dimension two is not necessary. The set in question is the image of the $n$-fold product $\{ 1^s,2^s,3^s,\ldots \}^n$ ($s>0$) by an orientation reversing quasiconformal radial stretch mapping. This is a quasiconformal analog of the set considered in \cite[Proposition 3.8]{BF} and Example \ref{ex:CIFS}. For different choices of the radial stretch mapping, these sets are themselves quasiconformally equivalent via another radial stretch map. In the planar case, this combination of set and mapping suffices to verify Theorem \ref{thm:theta-qc-sharpness}.

\smallskip

For $0<\alpha < 1$, let $f_\alpha:\R^n \setminus \{0\} \to \R^n \setminus \{0\}$ be given by
$$
f_\alpha(x) := |x|^{-1-\alpha} x.
$$
We recall that $f_\alpha$ is a $K_\alpha$-quasiconformal mapping for each such $\alpha$, with $K_\alpha = \tfrac1\alpha$. In fact,
$$
Df_\alpha(x) = |x|^{-1-\alpha} \left( \Id_n - (1+\alpha) ( \frac{x}{|x|} \otimes \frac{x}{|x|} ) \right)
$$
and hence $Df_\alpha(x)$ has singular values $|x|^{-1-\alpha}$ (with multiplicity $n-1$) and $\alpha |x|^{-1-\alpha}$ with multiplicity one. Thus
$$
\frac{||Df_\alpha||^n}{\det Df_\alpha} = \frac{|x|^{(-1-\alpha)n}}{\alpha |x|^{(-1-\alpha)n}} = \frac1\alpha
$$
and so $f$ is $K$-quasiconformal with $K = \tfrac1\alpha$.
 
\begin{prop}\label{prop:G-s-n-alpha}
Fix $s>0$ and $0<\alpha<1$ and set
$$
G_{s,n}^\alpha := f_\alpha ( \{ 1^s,2^s,3^s,\ldots \}^n ).
$$
Then
\begin{equation}\label{eq:G-s-n-alpha}
\dim_\theta G_{s,n}^\alpha = \frac{n\theta}{\theta+s\alpha}, \qquad \forall \, 0<\theta \le 1.
\end{equation}
\end{prop}

Recall that we made a standing assumption in this paper to only consider the upper intermediate dimension ${\overline\dim}_\theta$ and to drop use of the modifier `upper' in its description. In Proposition \ref{prop:G-s-n-alpha} we drop this standing assumption. 
By writing $\dim_\theta G_{s,n}^\alpha$ in \eqref{eq:G-s-n-alpha}, we mean that the upper intermediate dimension ${\overline\dim}_\theta G_{s,n}^\alpha$ and the lower intermediate dimension ${\underline\dim}_\theta G_{s,n}^\alpha$ coincide.

\begin{rem}\label{rem:G-s-n-alpha-banaji-rutar}
Note that the sets $G_{s,n}^\alpha$ are extremal for the Banaji--Rutar lower bound \eqref{eq:BR} for intermediate dimension, for all choices of $\theta$. It suffices to check that $G_{s,n}^\alpha$ has lower dimension zero and has Assouad dimension $n$. The former fact holds since $G_{s,n}^\alpha$ has Hausdorff dimension zero. For the latter fact, we note first that $\dim_A(\{1^s,2^s,3^s,\ldots\}) = 1$ by \cite[Theorem 3.4.7]{fr:book} and \cite[Theorem A.10(1)]{luukkainen}. Hence $\dim_A(\{1^s,2^s,3^s,\ldots\}^n) = n$ by \cite[Theorem A.5(5)]{luukkainen}. The desired conclusion $\dim_A G_{s,n}^\alpha = n$ now follows from the fact that quasiconformal mappings of $\R^n$ preserve the property that a subset has Assouad dimension $n$, see e.g.\ \cite[Proposition 5.4]{tys:assouad} or \cite[Theorem 1.2]{ct:qc-assouad-spectrum}.\footnote{Alternatively, one could check directly that $G_{s,n}^\alpha$ is neither doubling nor uniformly perfect.} Consequently, the estimate in \eqref{eq:BR}, for the set $X = G_{s,n}^\alpha$, reads
\begin{equation}\label{eq:G-s-n-alpha-banaji-rutar}
\overline\dim_\theta G_{s,n}^\alpha \ge \frac{n \theta \overline\dim_B G_{s,n}^\alpha}{\theta \, \overline\dim_B G_{s,n}^\alpha + n - \overline\dim_B G_{s,n}^\alpha}
\end{equation}
and using the value in \eqref{eq:G-s-n-alpha} we see that equality holds in \eqref{eq:G-s-n-alpha-banaji-rutar}.
\end{rem}

Let us first show how to deduce Theorem \ref{thm:theta-qc-sharpness} from Proposition \ref{prop:G-s-n-alpha}.

\begin{proof}[Proof of Theorem \ref{thm:theta-qc-sharpness}]
Fix $n=2$ and $K \ge 1$. Choose any $s>0$ and $0<\alpha<1$ and let $E = G_{s,2}^\alpha \subset \C$ be the set described in Proposition \ref{prop:G-s-n-alpha}. Let $\beta = \alpha/K$. The map $g(y) = |y|^{1/K-1} y$ is $K$-quasiconformal and
$$
g \circ f_\alpha = f_\beta,
$$
whence
$$
g(G_{s,2}^\alpha) = G_{s,2}^\beta.
$$
For $0<\theta<1$ we compute
$$
\frac1{\overline\dim_\theta g(G_{s,2}^\alpha)} - \frac12 = \frac1{\overline\dim_\theta G_{s,2}^\beta} - \frac12 = \frac{\theta+s\beta}{2\theta} - \frac12 = \frac{s\beta}{2\theta}
$$
and
$$
\frac1{\overline\dim_\theta G_{s,2}^\alpha} - \frac12 = \frac{s\alpha}{2\theta}.
$$
Since $K = \tfrac\alpha\beta$ we conclude that
$$
\frac1{\overline\dim_\theta g(G_{s,2}^\alpha)} - \frac12 = \frac1K \left( \frac1{\overline\dim_\theta G_{s,2}^\alpha} - \frac12 \right)
$$
for each choice of $0<\theta \le 1$. This completes the proof.
\end{proof}

\begin{proof}[Proof of Proposition \ref{prop:G-s-n-alpha}]
First, we show the upper bound
\begin{equation}\label{eq:eq-1}
\overline\dim_\theta G_{s,n}^\alpha \le \frac{n\theta}{\theta + s \alpha}.
\end{equation}
For some sufficiently small $\delta_0 > 0$ and $0<\delta \le \delta_0$, let
$$
m_0 := \lceil \delta^{-\theta/(\theta+s\alpha)} \rceil.
$$
We define a covering $\cU$ of $G_{s,n}^\alpha$ as follows:
\begin{itemize}
\item[(i)] for each point in the set $f_\alpha( \{1^s,2^s,\ldots,m_0^s\}^n)$, cover the point with a single ball of diameter $\delta$, and
\item[(ii)] cover $[0,m_0^{-s\alpha}]^n$ with $K$ sets of diameter $\delta^\theta$, where $K \approx ( \tfrac{m_0^{-s\alpha}}{\delta^\theta} )^n$.
\end{itemize}
To see why this is a covering of all of $G_{s,n}^\alpha$, assume that $f_\alpha(x)$ is an element of $G_{s,n}^\alpha$ with $x = (x_1,\ldots,x_n) \not \in \{1^s,2^s,\ldots,m_0^s\}^n$. Then $x_\ell \ge (m_0+1)^s$ for some index $\ell$. We claim that $||f_\alpha(x)||_\infty \le m_0^{-s\alpha}$ and hence $f_\alpha(x)$ lies in the set $[0,m_0^{-s\alpha}]^n$ considered in case (ii). In fact,
$$
||f_\alpha(x)||_\infty = \frac{\max\{ |x_j| : j=1,2,\ldots,n \}}{(\sum_k x_k^{2} )^{(1+\alpha)/2}}.
$$
Let $j$ be the index which achieves the preceding maximum. Then
$$
|x_j| = ||f_\alpha(x)||_\infty \cdot (\sum_k x_k^{2} )^{(1+\alpha)/2}
$$
and so
$$
m_0^{s\alpha} ||f_\alpha(x)||_\infty 
= \frac{m_0^{s\alpha} |x_j|}{ ( \sum_k x_k^{2} )^{(1+\alpha)/2}} 
< \frac{|x_\ell|^{\alpha} \, |x_j|}{ ( \sum_k x_k^{2} )^{(1+\alpha)/2}} 
\le \frac{|x_j|^{1+\alpha}}{ ( \sum_k x_k^{2} )^{(1+\alpha)/2}} 
\le 1.
$$
If we now choose $t>t_0 := \tfrac{n\theta}{\theta+s\alpha}$ then we find that
\begin{equation*}\begin{split}
\sum_{U \in \cU} (\diam U)^t  
&\le m_0^n \delta^t + K \delta^{t \theta} \lesssim m_0^n \delta^t + \left( \frac{m_0^{-s\alpha}}{\delta^\theta} \right)^n \delta^{t\theta} \\
&\lesssim \delta^{t\theta - n \theta + n s \alpha (\tfrac{\theta}{\theta+s\alpha})} + \delta^{t - n ( \tfrac{\theta}{\theta+s\alpha} ) } = \delta^{\theta ( t - t_0) } + \delta^{t-t_0} \lesssim 1.
\end{split}\end{equation*}
This completes the proof of \eqref{eq:eq-1}.

Next, we show that
\begin{equation}\label{eq:eq-2}
\underline\dim_\theta G_{s,n}^\alpha \ge \frac{n\theta}{\theta + s \alpha}.
\end{equation}
This time, for sufficiently small $\delta_0 > 0$ and $0<\delta \le \delta_0$, let
$$
m_1 := \lceil \delta^{-1/(1+s\alpha)} \rceil.
$$
If $s > 1$ then an elementary geometric argument yields the estimate
\begin{equation}\label{eq:geometric-argument-1}
\sup_{x \in [\delta,m_1^{-s\alpha}]^n} \inf_{y \in G_{s,n}^\alpha} |x-y| \lesssim \delta.
\end{equation}
For the reader's convenience, we postpone the proof of \eqref{eq:geometric-argument-1} and continue the main argument by observing that
\begin{equation}\begin{split}\label{eq:geometric-argument-2}
N(G_{s,n}^\alpha,\delta) &\ge N(G_{s,n}^\alpha \cap [\delta,m_1^{-s\alpha}]^n, \delta) \\ 
&\gtrsim \left( \frac{(m_1^{-s\alpha} - \delta)}{\delta} \right)^n \\
&\approx m_1^{s\alpha n} \delta^{-n} \\
&\approx \delta^{- ( \tfrac{n}{1+s\alpha}) }.
\end{split}\end{equation}
Alternatively, if $0<s<1$ then another geometric argument shows that the set $Q = f_\alpha( \{ 1^s,2^s,\ldots,m_1^s \}^n )$ is $\delta$-separated. (We leave the proof of this fact to the reader.) In this case we find that
\begin{equation*}\begin{split}
N(G_{s,n}^\alpha,\delta) \ge N(Q, \delta) 
&\gtrsim \delta^{n \eta} \left( \frac{m_1^{-s\alpha}}{\delta} \right)^n \\
&\gtrsim m_1^n \\
&\gtrsim \delta^{-(\tfrac{n}{1+s \alpha})}.
\end{split}\end{equation*}
We conclude that $\underline\dim_B G_{s,n}^\alpha \ge \frac{n}{1+s\alpha}$. Since $\dim_A G_{s,n}^\alpha = n$, the Banaji--Rutar estimate \eqref{eq:BR} yields
$$
\underline\dim_\theta G_{s,n}^\alpha \ge \frac{\theta \, \dim_A G_{s,n}^\alpha \, \underline\dim_B G_{s,n}^\alpha }{\dim_A G_{s,n}^\alpha - (1-\theta) \underline\dim_B G_{s,n}^\alpha } = \frac{n \theta}{\theta + s \alpha}.
$$
Consequently,
$$
\dim_\theta G_{s,n}^\alpha = \frac{n \theta}{\theta + s \alpha}
$$
as desired.
\end{proof}

Next, we prove \eqref{eq:geometric-argument-1}.

\begin{proof}[Proof of \eqref{eq:geometric-argument-1}]
Recall that $G_{s,n}^\alpha = f_\alpha((\N^s)^n)$ where $f_\alpha(u) = |u|^{-1-\alpha} u$ and we have chosen $m_1 \approx \delta^{-1/(1+s\alpha)}$. We are also working under the constraint $s>1$.

For variables $x$ and $y$ as in \eqref{eq:geometric-argument-1}, let us write $x = f_\alpha(u)$ and $y = f_\alpha(v)$. Then $v \in (\N^s)^n$ and $u$ satisfies $|u_j| \ge 1$ for all $j$ and $|u_{\submax}| \ge c_1 m_1^s$, where $c_1 = n^{-1/2-1/2\alpha}$. To see why the claims about $u$ are true, note that $x \in [0,m_1^{-s\alpha}]^n$ implies $|x| \le \sqrt{n} \cdot m_1^{-s\alpha}$ whence $|u| \ge n^{-1/2\alpha} \cdot m_1^s$ and so $|u_\submax| \ge c_1 m_1^s$. Moreover, since $|x_j| \ge \delta$ for all $j$, we have that $|u_j| = |u|^{1+\alpha} |x_j| \ge c_1^{1+\alpha} m_1^{s(1+\alpha)} \delta \ge c_2 m_1^{s(1+\alpha) - 1 - s \alpha} = c_2 m_1^{s-1}$ where $c_2 = c_1^{1+\alpha}$. If $\delta_0$ is chosen sufficiently small, then $m_1$ will be sufficiently large, so that we conclude $|u_j| \ge 1$ for all $j$.

Now  given $u$ as above, choose $v = (k_1^s,\ldots,k_n^s) \in (\N^s)^n$ so that $k_j^s \le u_j < (k_j+1)^s$. Here $k_j \ge 1$ for all $j$ and $k_\submax \ge c_3 m_1$ where $c_3 = \tfrac12 c_1^{1/s}$. Let us write $u_j = (k_j + \eps_j)^s$ for some $0 \le \eps_j < 1$. Then
\begin{equation*}\begin{split}
|f_\alpha(u) - f_\alpha(v)| &= \left| \frac{u}{|u|^{1+\alpha}} - \frac{v}{|v|^{1+\alpha}} \right| \\
&\lesssim \max_j \left| \frac{u_j}{|u|^{1+\alpha}} - \frac{v_j}{|v|^{1+\alpha}} \right| \\
& = \max_j \left| \frac{(k_j+\eps_j)^s}{(\sum_\ell (k_\ell+\eps_\ell)^{2s})^{(1+\alpha)/2}} - \frac{k_j^s}{(\sum_\ell k_\ell^{2s})^{(1+\alpha)/2}} \right|
\end{split}\end{equation*}
Let us define a function $h_j:\R^n \to \R$ by $h_j(X) = X_j^s (\sum_\ell X_\ell^{2s})^{-(1+\alpha)/2}$. Then
$$
|f_\alpha(u) - f_\alpha(v)| = \max_j |h_j(\vec{k} + \vec{\eps}) - h_j(\vec{k})|
$$
where $\vec{k} = (k_1,\ldots,k_n)$ and $\vec{\eps} = (\eps_1,\ldots,\eps_n)$. By the Mean Value Theorem, there exist values $\delta_j \in (0,1)$ so that $h_j(\vec{k} + \vec{\eps}) - h_j(\vec{k}) = \nabla h_j(\vec{k} + \delta_j \vec{\eps}) \cdot \vec{\eps}$. A computation gives
$$
\nabla h_j(X) = s \, \frac{X_j^{s-1} (\sum_\ell X_\ell^{2s}) \, \vec{E}_j - (1+\alpha) X_j^s (\sum_q X_q^{2s-1} \vec{E}_q)}{(\sum_\ell X_\ell^{2s})^{(3+\alpha)/2}}
$$
where $\vec{E}_j$ denotes the $j$th standard basis vector in $\R^n$. Thus
$$
\nabla h_j(\vec{k} + \delta_j \vec{\eps}) \cdot \vec{\eps}
$$
is equal to
$$ 
s \, \frac{(k_j+\delta_j \eps_j)^{s-1} (\sum_\ell (k_\ell + \delta_j \eps_\ell)^{2s}) \, \eps_j - (1+\alpha) (k_j + \delta_j \eps_j)^s (\sum_q (k_q + \delta_j \eps_q)^{2s-1} \eps_q)}{(\sum_\ell (k_\ell + \delta_j \eps_\ell)^{2s})^{(3+\alpha)/2}}.
$$
Since all of the values $\delta_q,\eps_\ell$ lie in the interval $[0,1)$ and all of the values $k_j$ are integers greater than or equal to one, we have that $k_j + \delta_q \eps_\ell \simeq k_j$ for all choices of $j,q,\ell$. Consequently,
\begin{equation*}\begin{split}
|f_\alpha(u) - f_\alpha(v)| 
&\lesssim \max_j \frac{k_j^{s-1} (\sum_\ell k_\ell^{2s}) + k_j^s (\sum_q k_q^{2s-1})}{(\sum_\ell k_\ell^{2s})^{(3+\alpha)/2}} \\
&\lesssim \frac{k_\submax^{3s-1}}{k_\submax^{(3+\alpha)s}} = k_\submax^{-1-\alpha s} \lesssim m_1^{-1-\alpha s} \lesssim \delta,
\end{split}\end{equation*}
which completes the proof.
\end{proof}

 The case $\alpha=1$ in \eqref{eq:geometric-argument-1} is treated in \cite{BF}. In that case, the argument can be somewhat simplified by making use of the identity
\begin{equation}\label{eq:conformal-identity}
\left| \frac{u}{|u|^2} - \frac{v}{|v|^2} \right| = \frac{|u-v|}{|u|\,|v|}
\end{equation}
valid for $u,v \in \R^n \setminus \{0\}$. Choosing $u$ and $v$ so that $x = u/|u|^2$ and $y = v/|v|^2$, rewrite \eqref{eq:geometric-argument-1} in the form
$$
\sup_{u:u/|u|^2 \in [\delta,m_1^{-s}]^n} \inf_{v \in \{1^s,2^s,\ldots\}^n} \frac{|u-v|}{|u|\,|v|} \lesssim \delta \approx m_1^{-(1+s)}.
$$
An anonymous referee pointed out that one might give an alternative proof of \eqref{eq:geometric-argument-1} in general  by using the approximate  identity
\[
\left| \frac{u}{|u|^{1+\alpha}} - \frac{v}{|v|^{1+\alpha}}\right| \approx \frac{|u-v|}{|u||v|} \min\{|u|,|v|\}^{1-\alpha}
\]
 again valid for $u,v \in \R^n \setminus \{0\}$. We leave the details to the reader.

The preceding example is a special case of a much more general phenomenon.

\begin{thm}\label{th:banaji-rutar-remark}
Let $E \subset \C$ be non-uniformly perfect and non-porous, and let $f:\C \to \C$ be a $K$-quasiconformal map such that (i) the pair $E$ and $f$ are extremal for the box-counting dimension distortion bound
\begin{equation}\label{eq:bc-dim-dist}
\frac1{\overline\dim_B f(E)} - \frac12 = \frac1K \left( \frac1{\overline\dim_B E} - \frac12 \right),
\end{equation}
and (ii) the set $E$ is also extremal for the Banaji--Rutar distortion estimate
\begin{equation}\label{eq:br}
\overline\dim_\theta E = \frac{\alpha(\beta-\lambda)\theta + (\alpha-\beta)\lambda}{(\beta-\lambda)\theta + (\alpha-\beta)}, \quad \lambda = \dim_L E, \beta = \overline\dim_B E, \alpha = \dim_A E
\end{equation}
for all $0<\theta\le 1$. Then the pair $E$ and $f$ are necessarily extremal for the intermediate dimension distortion bound, i.e.
\begin{equation}\label{eq:inter-dim-dist}
\frac1{\overline\dim_\theta f(E)} - \frac12 = \frac1K \left( \frac1{\overline\dim_\theta E} - \frac12 \right) \qquad \forall \, 0<\theta<1.
\end{equation}
\end{thm}

\begin{proof}
Recall that a subset $E \subset \R^n$ is uniformly perfect if and only if $\dim_L E > 0$, and $E$ is porous if and only if $\dim_A E < n$. Furthermore, uniform perfectness and porosity are quasiconformal invariants. Let $E' = f(E)$ and let $\lambda' = \dim_L E'$, $\beta'=\overline\dim_B E'$, $\alpha'=\dim_A E'$. In view of the preceding remarks, we necessarily have $\lambda = \lambda' = 0$ and $\alpha = \alpha' = 2$. Hence \eqref{eq:br} reads $\overline\dim_\theta E = 2\beta\theta/(\beta\theta+2-\beta)$, i.e.
\begin{equation}\label{eq:no-1}
\frac1{\overline\dim_\theta E} = \frac12 + \frac1\theta \left( \frac1\beta - \frac12 \right),
\end{equation}
while the estimate in \eqref{eq:BR} for the image set reads $\overline\dim_\theta E' \ge 2\beta'\theta/(\beta'\theta+2-\beta')$, or equivalently,
\begin{equation}\label{eq:no-2}
\frac1{\overline\dim_\theta E'} \le \frac12 + \frac1{\theta} \left( \frac1{\beta'} - \frac12 \right).
\end{equation}
Finally, \eqref{eq:bc-dim-dist} reads
\begin{equation}\label{eq:no-3}
\frac1{\beta'}-\frac12 = \frac1K \left( \frac1\beta-\frac12 \right).
\end{equation}
Combining \eqref{eq:no-1}, \eqref{eq:no-2}, and \eqref{eq:no-3} yields
$$
\frac1{\overline\dim_\theta E'} - \frac12 \le \frac1K \left( \frac1{\overline\dim_\theta E} - \frac12 \right).
$$
Hence \eqref{eq:inter-dim-dist} holds true and the pair $E$ and $f$ is extremal for Corollary \ref{cor:qc_dim_dist_intermediate}.
\end{proof}

\subsection{Sharpness of Sobolev distortion estimates for intermediate dimension}

In this final subsection, we show that Theorem \ref{thm:sob_dim_dist_intermediate} is sharp, simultaneously for all choices of the intermediate dimensions parameter $\theta$. The main result of this subsection is the following theorem.

\begin{thm}\label{thm:intermediate-sharpness}
Let $E \subset \R^n$ be compact, and let $p>n$. Then there exists a continuous mapping $f \in W^{1,p}(\R^n:\R^n)$ so that ${\overline\dim}_\theta f(E) = \tau_p({\overline\dim}_\theta E)$ for all $0<\theta\le 1$.
\end{thm}

To prove Theorem \ref{thm:intermediate-sharpness}, we follow the idea introduced by Kaufman in \cite{kau:sobolev}. He constructs similar examples for Hausdorff and box-counting dimension via a random and non-constructive procedure. 

As the argument in the intermediate dimensions case involves multiple novelties, we first summarise the major steps in Kaufman's argument in general terms, without specific reference to any particular notion of dimension. We then indicate in greater detail how to implement that argument in both the Hausdorff and the box-counting situations. Kaufman's original proof of sharpness for the box-counting dimension directly controlled the covering number $N(E,r_k)$ for a compact set $E$ along a suitable sequence of scales $r_k \to 0$ via combinatorial arguments. We modify those arguments by using instead a capacitary approach to box-counting dimension, using the capacities $\phi^m_r$ described in subsection \ref{subsec:intermediate-capacity}. These preliminaries facilitate our actual proof of Theorem \ref{thm:intermediate-sharpness}, for which we use the capacitary approach to intermediate dimension developed in \cite{bff:intermediate-projections} and the capacities $\phi^{s,m}_{r,\theta}$.

\subsubsection{Overview of Kaufman's strategy}

Our goal is to prove that if the source set $E$ has $\dim E \ge s$ for some $s\ge 0$, then for each $p>n$ there exists a $W^{1,p}$ map $f:\R^n \to \R^N$ so that $\dim f(E) \ge \tau_p(s)$. In Theorem \ref{thm:intermediate-sharpness} we will build such a map $f$ which realises the lower bound for all choices of ${\overline\dim}_\theta E$, $0<\theta \le 1$ simultaneously, but for convenience we start here with a quick overview of the construction for a fixed notion of dimension.

First, the hypothesis of a lower bound on $\dim E$ yields an array of consequences regarding the behaviour of $E$ at various locations and scales. Such consequences may generally be phrased in terms of growth conditions, either for suitable potential-theoretic measures on $E$ or for covering or packing numbers of $E$. 

Choose a suitable countable collection of scales, e.g., associated to the standard dyadic decomposition of $\R^n$. To each dyadic cube $Q$, associate the following data:
\begin{itemize}
\item[(i)] a countable collection of random vectors $\xi_Q^\nu$, $\nu=1,2,\ldots$, each of which is uniformly distributed with respect to the normalised Lebesgue measure in the unit ball $B^N$ of $\R^N$,
\item[(ii)] for each $\nu$, a Lipschitz cutoff function $\phi_Q^\nu$, equal to one on $Q$ and vanishing outside $\tfrac54Q$, and
\item[(iii)] a suitable constant $\overline{c}_Q^\nu>0$. 
\end{itemize}
With the above machinery to hand, define a random map $f_\xi:\R^n \to \R^N$ by 
\begin{equation}\label{eq:f-xi}
f_\xi(x) = \sum_{Q,\nu} \overline{c}_Q^\nu \phi_Q^\nu(x) \xi_Q^\nu, \qquad \xi = (\xi_Q^\nu)_{Q,\nu}.
\end{equation}
Two competing goals guide the choice of the coefficients $\overline{c}_Q^\nu$. First, one seeks to have $f_\xi \in W^{1,p}(\R^n:\R^N)$ for {\bf every} choice of $\xi$. Next, one aims to show that, almost surely w.r.t.\ $\xi$, $f_\xi$ increases the dimension of $E$ by the appropriate optimal amount. Needless to say, the manner in which this scheme is implemented depends on both the initial set $E$ and the type of dimension under consideration.

A key point in the proof is that the vectors $\xi_Q^\nu$ should be chosen independently with respect to the underlying probability measure, for all instances of the dyadic cube $Q$ and all values of $\nu$. This ensures that information derived about the behaviour of the image $f_\xi(E)$ from a single term in \eqref{eq:f-xi} does not interfere with information derived from other terms. 

The need to include multiple random vectors associated to a single dyadic cube arises first in the case of box-counting dimension, where we consider a sequence of scales $r_k \to 0$ on which $E$ has good covering behaviour, and include terms to \eqref{eq:f-xi} corresponding to all dyadic scales $2^{-m} \gtrsim r_k$. This entails adding infinitely many terms for each cube $Q \in \cD_m$, corresponding to all choices of $k$ such that $r_k \lesssim 2^{-m}$. In subsection \ref{subsec:proof-of-intermediate-sharpness}, where we adapt this proof to the intermediate dimensions context, a further complication arises. To wit, we combine a complete set of terms as in \eqref{eq:f-xi} for each choice in a countable dense set $\{\theta_j\}$ of intermediate dimension parameters in $(0,1]$, and for each such $j$, for all terms in a relevant sequence of scales $(r_{jk})$ with $r_{jk} \to 0$. In all of these cases, the independence hypothesis on the random vectors $\xi_Q^\nu$ is essential.

\begin{example}[Hausdorff dimension]\label{ex:dim-H-sharpness-sketch}
In the case when $\dim = \dim_H$, one implements the above scheme by considering measures $\mu$ defined on $E$ satisfying Frostman-type growth bounds with respect to exponents $s<\dim_H E$. Here one uses the following version of \eqref{eq:f-xi}:
\begin{equation}\label{eq:f-xi-H}
f_\xi(x) = \sum_{m=1}^\infty \frac1{m^2} \sum_{Q \in \cD_m} c_Q \phi_Q(x) \xi_Q, \qquad \xi = (\xi_Q)_Q,
\end{equation}
where $c_Q = \mu(E \cap CQ)^{1/\tau_p(s)}$ for a suitable  $C>1$. Note (as stated above) that in this case we do not need to include multiple terms for each dyadic cube $Q$.

Since $\mu$ satisfies a Frostman condition with exponent $s$, $c_Q \le C' (\diam Q)^{s/\tau_p(s)}$ for some constant $C'$; compare with the choice of the target scaling $r_m$ in \eqref{eq:r-m}. Denoting the inner sum in \eqref{eq:f-xi-H} by $f_{\xi,m}(x)$, an easy computation using the Lipschitz estimate for $\phi_Q$ and the bounded overlap property of the expanded cubes $\{\tfrac54 Q:Q \in \cD_m\}$ reveals that $\int |\nabla f_{\xi,m}|^p \le C 2^{m(p-n)} \sum_{Q \in \cD_m} c_Q^p$. Moreover, 
\begin{equation*}\begin{split}
\sum_{Q \in \cD_m} c_Q^p 
&= \sum_{Q \in \cD_m} \mu(E \cap CQ)^{p/\tau_p(s) - 1} \mu(E \cap CQ) \\
&\le 2^{-ms(p/\tau_p(s)-1)} \sum_{Q \in \cD_m} \mu(E \cap CQ) \le 2^{-m(ps/\tau_p(s)-s)},
\end{split}\end{equation*}
using the bounded overlap property of the expanded cubes $\{CQ:Q \in \cD_m\}$. Since $\tfrac{ps}{\tau_p(s)} - s = p-n$, we conclude that $\int |\nabla f_{\xi,m}|^p \lesssim 1$ for each $m$, and hence $f_\xi$ as in \eqref{eq:f-xi-H} lies in $W^{1,p}$ for each $\xi$.

An almost sure lower bound for $\dim_H f_\xi(E)$ is obtained by showing that the $t$-energies of the pushforward measures $(f_\xi)_\#(\mu)$ have finite expectation w.r.t.\ $\xi$:
$$
\bE_\xi(I_t((f_\xi)_\#\mu)) := \bE_\xi \left( \iint |a-b|^{-t} \, d(f_\xi)_\#\mu(a) \, d(f_\xi)_\# \mu (b) \right) < \infty
$$
for values $t<\tau_p(s)$. In view of Tonelli's theorem, it suffices to prove that
\begin{equation}\label{eq:tonelli}
\iint \bE_\xi \bigl( |f_\xi(x)-f_\xi(y)|^{-t} \bigr) d\mu x\, d \mu y < \infty.
\end{equation}
Writing $f_\xi(x) - f_\xi(y) = \sum_{Q} a_Q(x,y) \xi_Q$, where the sum is taken over all dyadic cubes $Q$, and introducing the quantity $||a(x,y)||_\infty := \max_Q |a_Q(x,y)|$, one appeals to the following elementary probabilistic lemma (see \cite[Lemma 4.4]{bmt:sobolev} for a proof).

\begin{lem}\label{lem:probability}
Let $(X_k)$ be a countable sequence of independent random variables, identically distributed according to the uniform measure on $B^N$. Let $a = (a_k) \in \ell^1$ and $0<t<N$. Then
$\bE \bigl( \bigl| \sum_k a_k X_k \bigr|^{-t} \bigr) \le C ||a||_\infty^{-t}$ for some $C = C(N,t)$.
\end{lem}

In view of Lemma \ref{lem:probability} we have
$$
\bE_\xi \bigl( |f_\xi(x) - f_\xi(y)|^{-t} \bigr) \le C ||a(x,y)||_\infty^{-t}
$$
and it suffices to prove that $\sup_{x \in E} \int ||a(x,y)||_\infty^{-t} \, d\mu y$ is finite. This is done by splitting the integral into a countable sum of terms on which a given scale $m \in \Z$ is predominant relative to a fixed point $x \in E$, and using the choice of $c_Q$ and the Frostman energy growth condition on the measure $\mu$.
\end{example}

\begin{example}[Box-counting dimension]\label{ex:dim-B-sharpness-sketch}
Turning to the case when $\dim = {\overline\dim}_B$, we first note that lower bounds on box-counting dimension are weaker than lower bounds for Hausdorff dimension, and only guarantee growth of the covering numbers $N(E,r_k)$ along various scales $r_k \to 0$. Consequently, in Kaufman's original proof in \cite{kau:sobolev}, the preceding proof scheme is implemented by replacing the measure-theoretic arguments of Example \ref{ex:dim-H-sharpness-sketch} by combinatorial arguments involving finite sets of points, with counting measure adopting the role of the Frostman measures. To aid the reader in understanding our proof of Theorem \ref{thm:intermediate-sharpness}, we present here an alternate argument using Falconer's capacities $C_r^m(E)$.

\smallskip

Assume that a given compact set $E \subset [0,1]^n$ has ${\overline\dim}_B E \ge s$, and fix $\sigma < s$. Since ${\overline\dim}_B^n E > \sigma$ (Proposition \ref{prop:box-profile}(iv)), there exists $r_k \to 0$ so that 
\begin{equation}\label{eq:Crn-lower-bound}
C_{r}^n(E) \ge r^{-\sigma}.
\end{equation} 
for each $r = r_k$. For the moment, we fix one such scale $r$. Let $\mu \in \cM(E)$ realise the infimum in the definition of the capacity $C_{r}^n(E)$. Then $\cE_{r}^n(\mu) \le r^\sigma$. Proposition \ref{prop:box-profile}(ii) implies that for $\mu$-a.e.\ $x \in E$,
\begin{equation}\label{eq:upper-bound-old}
\int \phi_{r}^n(|x-y|) \, d\mu y \le r^\sigma.
\end{equation}
Restricting the integral in \eqref{eq:upper-bound-old} to $B(x,r) \cap E$ yields
\begin{equation}\label{eq:upper-bound-old2}
\mu(B(x,r)) \le r^\sigma.
\end{equation}
A priori these conclusions only hold true for a $\mu$-full set of points in $E$. We may restrict attention to a subset $E_0 \subset E$ with $\mu(E_0)=1$ so that \eqref{eq:upper-bound-old} and \eqref{eq:upper-bound-old2} hold for all $x \in E_0$. This restriction does not impact our conclusion, since we seek a lower bound for the box-counting dimension of a suitable Sobolev image of $E$. We emphasise that both \eqref{eq:upper-bound-old} and \eqref{eq:upper-bound-old2} are only guaranteed to hold for the radius $r$ in \eqref{eq:Crn-lower-bound}. For larger radii $\hat{r}\in (r,1)$ the doubling property of $\R^n$ in conjunction with \eqref{eq:upper-bound-old2} gives
\begin{equation}\label{eq:upper-bound-old3}
\mu(B(x,\hat{r})) \le C(n) \hat{r}^n r^{\sigma-n}, \qquad \forall \, x \in E_0, \, \forall r<\hat{r}<1.
\end{equation}
We now define component functions $f_{\xi}$ analogous to those in \eqref{eq:f-xi}. First, for any dyadic cube $Q$, we choose the constant $c_Q$ as follows:
\begin{equation}\label{eq:cQagain}
c_Q = r^\gamma \mu(8Q)^{1/n}.
\end{equation}
Here\footnote{The exponent is chosen to ensure membership of $f_{\xi}$ in the appropriate Sobolev space. Note that
$$
\gamma(\sigma) = \sigma \left( \frac1{\tau(\sigma)} - \frac1n \right).
$$
Thus an alternate formula for the quantity in \eqref{eq:cQagain} is $c_Q = r^{\sigma/\tau} \left( \tfrac{\mu(8Q)}{r^\sigma} \right)^{1/n}$.}
\begin{equation}\label{eq:gamma-p-sigma}
\gamma = \gamma_p(\sigma) := \frac{(p-n)(n-\sigma)}{pn}.
\end{equation}
We then set
$$
f_{\xi}(x) = \sum_{\substack{m \\ 2^{-m} \gtrsim r}} \sum_{Q\in\cD_m} c_Q \phi_Q(x) \xi_Q
$$
where $\phi_Q$ and $\xi_Q$ are as before.\footnote{We emphasise a small distinction here, in contrast with the Hausdorff dimension case. Whereas in that case the function $f_\xi$ was defined by summing immediately over all dyadic cubes $Q$, here we fix a scale $r>0$ and only sum over dyadic scales $2^{-m}$ which exceed $r$. At a later stage of the argument, an additional summation will be performed over the relevant scales $r_k$ chosen as in \eqref{eq:Crn-lower-bound}.} Let $f_{\xi,m}(x)$ be the inner sum, so $f_{\xi} = \sum_m f_{\xi,m}$. We again have $\int |\nabla f_{\xi,m}|^p \lesssim 2^{m(p-n)} \sum_{Q \in \cD_m} c_Q^p$, and this time we find
\begin{equation*}\begin{split}
\sum_{Q \in \cD_m} c_Q^p 
&\lesssim \sum_{Q \in \cD_m} r^{\gamma p} \mu(8Q)^{p/n-1} \mu(8Q) \\
&\lesssim r^{\gamma p}2^{-mn(p/n-1)}r_j^{(\sigma-n)(p/n-1)} \lesssim 2^{-m(p-n)}
\end{split}\end{equation*}
by the choice of $\gamma$. Hence $\int |\nabla f_{\xi,m}|^p$ is bounded uniformly from above (independently of $\xi$ and $m$), and we conclude that $f_{\xi} \in W^{1,p}$ with $||f_{\xi}||_{W^{1,p}} \lesssim \log(1/r)$. Note that in this case the Sobolev norm grows logarithmically in terms of the scale $r$; we will counterbalance this growth at a later stage by choosing appropriate normalisation constants when summing over all scales $r_k$.

As before, we assume that the vectors $(\xi_Q)$ are chosen independently and are identically distributed w.r.t.\ the uniform measure on the unit ball $B$ of $\R^n$. Set $R := r^{\sigma/\tau}$. We claim that there exists a constant $\kappa>0$ depending only on $n$ (in particular, independent of $R$) so that for each $\eps>0$ the inequality
\begin{equation}\label{eq:capacity-lower-bound-for-E-0-at-scale-r-old}
C_{R}^n(f_{\xi}(E_0)) \ge \eps \kappa^{-1} R^{-\tau} \log^{-4}(\tfrac1{R})
\end{equation}
holds with probability at least $1 - c \eps \log^{-2}(1/R)$ w.r.t.\ the random variable $\xi$. This conclusion follows if
\begin{equation}\label{eq:energy-upper-bound-old}
\cE_{R}^n((f_{\xi})_{\#}\mu) \le \eps^{-1} \kappa R^\tau \log^{4}(\tfrac1{R})
\end{equation}
holds with the same probability. By Chebyshev's inequality it suffices to prove the following lemma, where the auxiliary parameter $\eps$ does not appear.

\begin{lem}\label{lem:technical-lemma1old}
$\bE_\xi\bigl( \cE_{R}^n((f_{\xi})_{\#}\mu)\bigr) \le \kappa R^{\tau} \log^{2}(\tfrac1{R})$.
\end{lem}

Towards this end, we compute
\begin{equation}\begin{split}\label{eq:bound0old}
\bE_\xi \bigl( \cE_{R}^n((f_{\xi})_{\#}\mu) \bigr) 
&= \bE_\xi \left( \iint \phi_{R}^n(|f_{\xi}(x)-f_{\xi}(y)|) \, d\mu x \, d\mu y \right) \\
&= \iint \bE_\xi \left( \phi_{R}^n(|f_{\xi}(x)-f_{\xi}(y)|) \right) \, d\mu x \, d\mu y.
\end{split}\end{equation}
Again, we write $f_{\xi}(x) - f_{\xi}(y) = \sum_{Q \in \cD_m, 2^{-m} \gtrsim r} a_Q(x,y) \xi_Q$ with $a_Q(x,y) = c_Q(\phi_Q(x)-\phi_Q(y))$.

\begin{lem}\label{lem:technical-lemma2old}
Let $(X_k)$ be i.i.d.\ w.r.t.\ the uniform measure on $B^n$, let $a = (a_k) \in \ell^1$, and let $R>0$. Then 
\begin{equation}\label{eq:technical-lemma2old}
\bE\bigl( \phi_R^n( \bigl| \sum_k a_k X_k \bigr| ) \bigr) \lesssim \phi_R^n(||a||_\infty) \log_+(||a||_\infty/R),
\end{equation}
where $\log_+(t) = \max\{\log(t),0\}$.
\end{lem}

We will later prove a more general version of this lemma (see Lemma \ref{lem:technical-lemma2}) in the intermediate dimensions context. Since Lemma \ref{lem:technical-lemma2old} is a special case of Lemma \ref{lem:technical-lemma2} and we do not gain any particular insight by proving that special case here, we omit it and refer the interested reader to the proof of Lemma \ref{lem:technical-lemma2}. However, we draw attention to a distinction between the upper bound in \eqref{eq:technical-lemma2old} and the corresponding one in Lemma \ref{lem:probability}; the former includes an extra (harmless) logarithmic factor.

Lemma \ref{lem:technical-lemma2old} implies that
\begin{equation}\begin{split}\label{eq:technical-estimate-old}
\bE_\xi \bigl( \cE_{R}^n((f_{\xi})_{\#}\mu) \bigr) 
&\lesssim \iint \phi_{R}^n(||a(x,y)||_\infty) \log_+(||a(x,y)||_\infty/R) \, d\mu x \, d\mu y
\end{split}\end{equation}
and we are reduced to proving that the right hand side of \eqref{eq:technical-estimate-old} is bounded above by $\kappa R^\tau \log^2(1/R)$ for some $\kappa = \kappa(n)$.

Since $||a(x,y)||_\infty \le r^\gamma$ for all $x,y \in E_0$, we have
\begin{equation}\begin{split}\label{eq:bound1old}
&\iint_{\{|x-y|\le r\}} \phi_{R}^n(||a(x,y)||_\infty) \log_+(||a(x,y)||_\infty/R) \, d\mu x \, d\mu y \\
& \qquad \lesssim \int \log_+(r^{\gamma}/R) \mu(B(x,r)) \, d\mu x \\
& \qquad \lesssim r^\sigma \log(1/R) = R^{\tau} \log(1/R).
\end{split}\end{equation}
It thus suffices to restrict attention to the set of points $x,y \in E_0$ with $|x-y| \ge r$. Fix $x \in E_0$ and define for each $y \in E_0$ with $|x-y|\ge r$ an integer $m(y)$ so that $|x-y| \simeq 2^{-m(y)}$. Since $|x-y|\ge r$ we are assured that $m(y)$ appears as a term in the summand for $f_{\xi}$. Set $E_m = \{y \in E_0:m(y) = m\}$. Given $m$, fix a cube $Q_0^m \ni x$ with $Q_0^m \in \cD_{m+1}$; then $y \in E_m$ $\Rightarrow$ $y \notin \tfrac54 Q_0^m$ and $y \in 8 Q_0^m$. In particular, $E_ m \subset 8Q_0^m$. Moreover, $y \in E_m$ $\Rightarrow$ $||a(x,y)||_\infty \ge |a_{Q_0^m}(x,y)| = c_{Q_0^m} = r^\gamma \mu(8Q_0^m)^{1/n}$. We will prove that for all $x \in E$,
$$
\int_{\{|x-y|\ge r\}} \phi_{R}^n(||a(x,y)||_\infty) \log_+(||a(x,y)||_\infty/R) \, d\mu y \le C(n) R^\tau = C(n) r^\sigma.
$$
We express the left hand side as the sum of integrals over the sets $E_m$ and proceed to estimate the $m$th term in the sum:
\begin{equation*}\begin{split}
&\int_{E_m} \phi_{R}^n(||a(x,y)||_\infty) \log_+(||a(x,y)||_\infty/R) \, d\mu y \\
& \quad \le \phi_{R}^n( r^\gamma \mu(8Q_0^m)^{1/n}) \log_+(r^\gamma \mu(8Q_0^m)^{1/n}/R) \mu(E_m) \\
& \quad \lesssim \left( r^{\gamma - \tfrac\sigma\tau} \mu(8Q_0^m)^{1/n} \right)^{-n} \log_+(r^{\gamma-\tfrac\sigma\tau} \mu(8Q_0^m)^{1/n}) \mu(8Q_0^m) 
\lesssim r^\sigma \log(\tfrac{2^{-m}}{r})
\end{split}\end{equation*}
where we used the fact that $\gamma - \tfrac\sigma\tau = -\tfrac\sigma{n}$. Summing over relevant values of $m$ yields
\begin{equation}\begin{split}\label{eq:bound2old}
&\int_{\{|x-y|\ge r\}} \phi_{R}^n(||a(x,y)||_\infty) \log_+(||a(x,y)||_\infty/R) \, d\mu y \\
&\qquad \lesssim r^\sigma \log^2(1/r) \lesssim R^\tau \log^2(1/R),
\end{split}\end{equation}
where we used $\sum_{m:2^{-m} \gtrsim r} \bigl( m + \log(1/r) \bigr) \lesssim \log^2(1/r)$. Lemma \ref{lem:technical-lemma1old} now follows by combining  \eqref{eq:bound0old}, \eqref{eq:technical-estimate-old}, \eqref{eq:bound1old}, and \eqref{eq:bound2old}. In particular, \eqref{eq:capacity-lower-bound-for-E-0-at-scale-r-old} holds with constant $\kappa$ depending only on $r$. Recall that $E_0$ was chosen to be of full measure with respect to the capacitary measure $\mu$, and that this measure depended on the scale $r$. However, by monotonicity of capacity we now conclude that \eqref{eq:capacity-lower-bound-for-E-0-at-scale-r-old} holds with the same constant $\kappa$ but with $E_0$ replaced by the fixed original set $E$.

\smallskip

We complete the proof by returning to the original sequence of scales $r_k \to 0$ for which $C^n_{r_k}(E) \ge r_k^{-\sigma}$. For each $k$ we carry out the above argument, generating a random family of functions $f_{\xi_k,k}$ so that the inequality
$$
C_{R_k}^n(f_{\xi_k,k}(E)) \ge \eps \kappa^{-1} R_k^{-\tau} \log^{-2}(\tfrac1{R_k})
$$
holds for any $\eps>0$ with probability at least $1 - c(n) \eps \log^{-2}(1/R_k)$ in $\xi$, with $R_k = r_k^{\sigma/\tau}$ and $c(n)<1$. We sum over $k$ with appropriate normalisation coefficients:
\begin{equation}\label{eq:f-xi-old}
f_\xi := \sum_{k=1}^\infty \frac1{k^2} \frac1{\log(1/r_k)} f_{\xi_k,k}.
\end{equation}
Here we may assume without loss of generality that the scales $r_k$ satisfy $r_k < e^{-k^2}$, which implies that the sequence $(\log^{-1}(1/r_k))$ is summable.\footnote{The factor $\log^{-1}(1/r_k)$ is inserted in \eqref{eq:f-xi-old} to counterbalance growth of the $W^{1,p}$ norm of $f_{\xi_k,k}$.} Using independence of the full suite of random variables $\xi = (\xi_k:k \in \N)$, we argue that the function $f_\xi$ so defined satisfies 
$$
C_{R_k}^n(f_{\xi}(E)) \gtrsim \eps R_k^{-\tau} \log^{-2}(\tfrac1{R_k})
$$
for all $k$, with probability $\ge 1 - c(n) \eps$. It follows that ${\overline\dim}_B(f_\xi(E)) = {\overline\dim}_B^n(f_\xi(E)) \ge \tau(\sigma)$ with probability $\ge 1 - c(n) \eps$. Finally, we choose $\sigma_\ell \nearrow s$ and appeal to the previous conclusion with $\eps_\ell = 2^{-\ell-1}$. Since $c(n)<1$ we conclude that with positive probability there exists $\xi$ so that ${\overline\dim}_B(f_\xi(E)) \ge \tau(s)$. This completes the proof.
\end{example}

\subsubsection{The intermediate dimensions case}\label{subsec:proof-of-intermediate-sharpness}

In this subsection we modify the preceding argument to handle the intermediate dimensions ${\overline\dim}_\theta$, $0<\theta\le 1$. Theorem \ref{thm:intermediate-sharpness} is an immediate consequence of the following proposition, whose proof takes up the remainder of this subsection.

\begin{prop}\label{prop:intermediate-sharpness}
Let $E \subset \R^n$ be compact, and assume that ${\overline\dim}_\theta E \ge s(\theta)$ for all $0<\theta \le 1$. Let $p>n$. Then there exists a continuous function $f \in W^{1,p}(\R^n:\R^n)$ so that
$$
{\overline\dim}_\theta f(E) \ge \tau_p(s(\theta)) \qquad \forall \, 0<\theta \le 1.
$$
\end{prop}

\begin{proof}
Let $E \subset [0,1]^n$ be compact and let $p>n$. In the first part of the proof, we fix a parameter $\theta \in (0,1)$ and a scale $r>0$. We will build a random function which achieves optimal behaviour for the $\theta$-intermediate dimension at scale $r$. At the end of the proof, we will explain how to leverage this information to construct a function with the property stated in Proposition \ref{prop:intermediate-sharpness}.

As in the statement of the proposition, assume that $s \le {\overline\dim}_\theta E$. Fix $\sigma<s$. By Proposition \ref{prop:box-profile}(iv),
\begin{equation}\label{eq:dim-theta-profile-lower-bound}
{\overline\dim}_\theta^n E > \sigma.
\end{equation}
Let $r>0$ be a scale such that
\begin{equation}\label{eq:Crtsn-lower-bound}
C_{r,\theta}^{\sigma,n}(E) \ge r^{-\sigma};
\end{equation} 
note that \eqref{eq:dim-theta-profile-lower-bound} implies the existence of a sequence of scales $r_j \searrow 0$ for which this conclusion holds true. Let $\mu \in \cM(E)$ realise the infimum in the definition of $C_{r,\theta}^{\sigma,n}(E)$. Then $\cE_{r,\theta}^{\sigma,n}(\mu) \le r^\sigma$. Proposition \ref{prop:box-profile}(ii) implies that for $\mu$-a.e.\ $x \in E$,
\begin{equation}\label{eq:upper-bound}
\int \phi_{r,\theta}^{\sigma,n}(|x-y|) \, d\mu y \le r^\sigma.
\end{equation}
Restricting the integral in \eqref{eq:upper-bound} to $B(x,r) \cap E$ yields
\begin{equation}\label{eq:upper-bound-1-5}
\mu(B(x,r)) \le r^{\sigma};
\end{equation}
restricting the integral to $B(x,r^\theta) \cap E$ and using the fact that $\phi_{r,\theta}^{\sigma,n}(u) \ge r^{(1-\theta)\sigma}$ for $r<u<r^\theta$ leads us to conclude that
$$
\mu(B(x,r)) + r^{(1-\theta)\sigma} \bigl( \mu(B(x,r^\theta)) - \mu(B(x,r)) \bigr) \le r^\sigma
$$
and hence
\begin{equation}\label{eq:upper-bound2}
\mu(B(x,r^\theta)) \le r^{\sigma\theta}.
\end{equation}
As explained previously (see the discussion following \eqref{eq:upper-bound-old2}) we may restrict attention to a subset $E_0 \subset E$ with $\mu(E_0)=1$ so that \eqref{eq:upper-bound}, \eqref{eq:upper-bound-1-5}, and \eqref{eq:upper-bound2} hold for all $x \in E_0$. For larger radii $\hat{r} \in (r^\theta,1)$ the doubling property of $\R^n$ in conjunction with \eqref{eq:upper-bound2} gives\footnote{More precisely, first cover $B(x,\hat{r})$ with $C(n)(\hat{r} / r^\theta)^n$ balls $B_k$ of radius $r^\theta/2$ with no assumption on whether or not the centers lie in $E_0$. If one of the balls $B_k$ is disjoint from $E_0$ it can be ignored, otherwise, choose a point $y_k \in B_k \cap E_0$ and note that $B_k\cap E_0 \subset B(y_k,r^\theta) \cap E_0$.}
\begin{equation}\label{eq:upper-bound3}
\mu(B(x,\hat{r})) \le C(n) \hat{r}^n r^{(\sigma-n)\theta}, \qquad \forall \, x \in E_0, \, \forall r^\theta<\hat{r}<1.
\end{equation}
We now define a random function $f_{\xi}$ at scale $r>0$. To each dyadic cube $Q \in \cD_m$ with $2^{-m} \gtrsim r$, we associate a Lipschitz cut-off function $\phi_Q(x)$, equal to one in $Q$ and supported in $\tfrac54Q$, and a random vector $\xi_Q$ uniformly distributed in the unit ball $B^n$ of $\R^n$. We also choose constants $c_Q$ as follows:
\begin{equation}\label{eq:cQagain2}
c_Q = r^\gamma \mu(8Q)^{1/n}.
\end{equation}
Here
\begin{equation}\label{eq:gamma-p-sigma-theta}
\gamma = \gamma_p(\sigma,\theta) := \theta \, \frac{(p-n)(n-\sigma)}{pn}.
\end{equation}
We observe that $\gamma_p(\sigma,\theta) = \theta \, \gamma_p(\sigma)$, where $\gamma_p(\sigma)$ denotes the corresponding parameter from the capacitary proof of the analogous theorem for box-counting dimension; see \eqref{eq:gamma-p-sigma}. Set
$$
f_{\xi}(x) := \sum_{\substack{m \\ 2^{-m} \gtrsim r}} \sum_{Q\in\cD_m} c_Q \phi_Q(x) \xi_Q.
$$
Let $f_{\xi,m}(x)$ be the inner sum, so $f_{\xi} = \sum_m f_{\xi,m}$. We have $\int |\nabla f_{\xi,m}|^p \lesssim 2^{m(p-n)} \sum_{Q \in \cD_m} c_Q^p$, and we compute
\begin{equation*}\begin{split}
\sum_{Q \in \cD_m} c_Q^p 
&\lesssim \sum_{Q \in \cD_m} r^{\gamma p} \mu(8Q)^{p/n-1} \mu(8Q) \\
&\lesssim r^{\gamma p}2^{-mn(p/n-1)}r^{\theta(\sigma-n)(p/n-1)} \\
&\lesssim 2^{-m(p-n)},
\end{split}\end{equation*}
where we used \eqref{eq:gamma-p-sigma-theta}. Hence $\int |\nabla f_{\xi,m}|^p$ is bounded uniformly from above (independently of $\xi$ and $m$), so $f_{\xi} \in W^{1,p}$ with $||f_{\xi}||_{W^{1,p}} \lesssim \log(1/r)$.

As before, we assume that the random vectors $(\xi_Q)$ are chosen independently and are identically distributed w.r.t.\ the uniform measure on $B^n$. Set $R := r^{\sigma/\tau_p(\sigma)}$. We claim that there exists a constant $\kappa>0$ depending only on $n$ (in particular, independent of $R$) so that for each $\eps>0$ the inequality
\begin{equation}\label{eq:capacity-lower-bound-for-E-0-at-scale-r}
C_{R,\theta}^{\tau,n}(f_{\xi}(E_0)) \ge \eps \kappa^{-1} R^{-\tau} \log^{-4}(\tfrac1{R})
\end{equation}
holds with probability at least $1 - c \eps \log^{-2}(1/R)$ w.r.t.\ the random variable $\xi$. This conclusion follows if
\begin{equation}\label{eq:energy-upper-bound}
\cE_{R,\theta}^{\tau,n}((f_{\xi})_{\#}\mu) \le \eps^{-1} \kappa R^\tau \log^{4}(\tfrac1{R})
\end{equation}
holds with the same probability. By Chebyshev's inequality it suffices to prove the following lemma, where the auxiliary parameter $\eps$ does not appear.

\begin{lem}\label{lem:technical-lemma1}
$\bE_\xi\bigl( \cE_{R,\theta}^{\tau,n}((f_{\xi})_{\#}\mu)\bigr) \le \kappa R^{\tau} \log^{2}(\tfrac1{R})$.
\end{lem}

Towards this end, we compute
\begin{equation}\begin{split}\label{eq:bound0}
\bE_\xi \bigl( \cE_{R,\theta}^{\tau,n}((f_{\xi})_{\#}\mu) \bigr) 
&= \bE_\xi \left( \iint \phi_{R,\theta}^{\tau,n}(|f_{\xi}(x)-f_{\xi}(y)|) \, d\mu x \, d\mu y \right) \\
&= \iint \bE_\xi \left( \phi_{R,\theta}^{\tau,n}(|f_{\xi}(x)-f_{\xi}(y)|) \right) \, d\mu x \, d\mu y.
\end{split}\end{equation}
Following the outline described earlier, we write $f_{\xi}(x) - f_{\xi}(y) = \sum_{Q \in \cD_m, 2^{-m} \gtrsim r} a_Q(x,y) \xi_Q$ with $a_Q(x,y) = c_Q(\phi_Q(x)-\phi_Q(y))$.

\begin{lem}\label{lem:technical-lemma2}
Let $(X_k)$ be a sequence of random variables, i.i.d.\ w.r.t.\ the uniform measure on $B^n$. Let $a = (a_k) \in \ell^1$ and let $R>0$. Then 
$$
\bE\bigl( \phi_{R,\theta}^{\tau,n}( \bigl| \sum_k a_k X_k \bigr| ) \bigr) \lesssim \phi_{R,\theta}^{\tau,n}(||a||_\infty) \log_+(||a||_\infty/R^\theta),
$$
where $\log_+(t) = \max\{\log(t),0\}$.
\end{lem}

\begin{proof}
Assume without loss of generality that $||a||_\infty = a_1 > 0$ and define $Y := - \sum_{k \ge 2} a_k X_k$ and $\hat{X} := (X_2,X_3,\ldots)$. Then
$$
\bE_X(\phi_{R,\theta}^{\tau,n}(|\sum_k a_k X_k)) = \bE_{\hat{X}} \bE_{X_1} \left( \phi_{R,\theta}^{\tau,n}(|a_1 X_1 - y|) \right)
$$
and we estimate
\begin{equation*}\begin{split}
 \bE_{X_1} \left( \phi_{R,\theta}^{\tau,n}(|a_1 X_1 - y|) \right)
&= \left. \frac1{\Vol(B^n)} \int_{B^n} \phi_{R,\theta}^{\tau,n}(|a_1 x - y|) \, dx \right|_{y=Y} \\
&\le c(n) \int_{B^n} \phi_{R,\theta}^{\tau,n}(|a_1 x|) \, dx
\end{split}\end{equation*}
by the symmetrisation lemma \ref{lem:Ers-symmetrization}. Let us denote by $\bone_S$ the indicator function for a Boolean variable $S$, i.e., $\bone_S = 1$ if $S$ is true and $\bone_S=0$ if $S$ is false. Then
\begin{equation*}\begin{split}
\bE &\bigl( \phi_{R,\theta}^{\tau,n}( \bigl| \sum_k a_k X_k \bigr| ) \bigr) \\
&\lesssim \Vol(B^n \cap B^n(\tfrac{R}{a_1})) + \int_{\{ x \in B^n : \tfrac{R}{a_1} < |x| < \tfrac{R^\theta}{a_1} \}} \left( \frac{R}{a_1} \right)^\tau \, \frac{dx}{|x|^\tau} \\
&\, \hspace{5cm}
+ \int_{\{ x \in B^n : \tfrac{R^\theta}{a_1} < |x| < 1 \}} \frac{R^{\theta(n-\tau)+\tau}}{a_1^n} \, \frac{dx}{|x|^{\blue n}} \\
&\lesssim \min\{1,(\tfrac{R}{a_1})^n\} + (\tfrac{R}{a_1})^\tau \int_{R/a_1}^{1\wedge (R^\theta/a_1)} \rho^{n-\tau-1} \, d\rho \, \cdot \, \bone_{R<a_1} \\
&\, \hspace{5cm}
+ \frac{R^{\theta(n-\tau)+\tau}}{a_1^n} \int_{R^\theta/a_1}^1 \frac{d\rho}{\rho} \, \cdot \, \bone_{R^\theta<a_1} \\
&\lesssim \min\{1,(\tfrac{R}{a_1})^n\} + (\frac{R}{a_1})^\tau \, (1 \wedge \frac{R^\theta}{a_1})^{n-\tau} \, \cdot \, \bone_{R<a_1} + \frac{R^{\theta(n-\tau)+\tau}}{a_1^n} \log(a_1/R^\theta) \, \cdot \, \bone_{R^\theta<a_1} \\
&\lesssim \phi_{R,\theta}^{\tau,n}(a_1) \log_+(a_1/R^\theta).
\end{split}\end{equation*}
To obtain the final line, consider the three possibilities $a_1<R$, $R<a_1<R^\theta$, and $R^\theta<a_1$ in the definition of the kernel $\phi_{R,\theta}^{\tau,n}(a_1)$ separately.
\end{proof}

Lemma \ref{lem:technical-lemma2} implies that
\begin{equation}\begin{split}\label{eq:technical-estimate}
\bE_\xi \bigl( \cE_{R,\theta}^{\tau,n}((f_{\xi})_{\#}\mu) \bigr) 
&\lesssim \iint \phi_{R,\theta}^{\tau,n}(||a(x,y)||_\infty) \log_+(||a(x,y)||_\infty/R^\theta) \, d\mu x \, d\mu y
\end{split}\end{equation}
and we are reduced to proving that the right hand side of \eqref{eq:technical-estimate} is bounded above by $\kappa R^\tau \log^2(1/R)$ for some $\kappa = \kappa(n)$.

Since $||a(x,y)||_\infty \le r^\gamma$ for all $x,y \in E_0$, where $\gamma$ is as in \eqref{eq:gamma-p-sigma-theta}, we have
\begin{equation}\begin{split}\label{eq:bound1}
&\iint_{\{|x-y|\le r\}} \phi_{R,\theta}^{\tau,n}(||a(x,y)||_\infty) \log_+(||a(x,y)||_\infty/R^\theta) \, d\mu x \, d\mu y \\
& \qquad \lesssim \int \log_+(r^{\gamma}/R^\theta) \mu(B(x,r)) \, d\mu x \\
& \qquad \lesssim r^\sigma \log(1/R) = R^{\tau} \log(1/R),
\end{split}\end{equation}
where we used \eqref{eq:upper-bound-1-5} and the definiiton of $R$. It thus suffices to restrict attention to the set of points $x,y \in E_0$ with $|x-y| \ge r$.

\smallskip

Fix $x \in E_0$ and define for each $y \in E_0$ with $|x-y|\ge r$ an integer $m(y)$ so that $|x-y| \simeq 2^{-m(y)}$. Since $|x-y|\ge r$ we are assured that $m(y)$ appears as a term in the summand for $f_{\xi}$. Set $E_m = \{y \in E_0:m(y) = m\}$. Given $m$, fix a cube $Q_0^m \ni x$ with $Q_0^m \in \cD_{m+1}$; then $y \in E_m$ $\Rightarrow$ $y \notin \tfrac54 Q_0^m$ and $y \in 8 Q_0^m$. In particular, $E_ m \subset 8Q_0^m$. Moreover,
\begin{equation}\label{eq:a-lower-bound}
||a(x,y)||_\infty \ge |a_{Q_0^m}(x,y)| = c_{Q_0^m} = r^\gamma \mu(8Q_0^m)^{1/n} \qquad \forall \, y \in E_m.
\end{equation}
We will show that
\begin{equation*}\begin{split}
&\int_{\{|x-y|\ge r\}} \phi_{R,\theta}^{\tau,n}(||a(x,y)||_\infty) \log_+(||a(x,y)||_\infty/R^\theta) \, d\mu y \\
&\qquad \le C(n) R^\tau \log^2(1/R) = C(n) r^\sigma \log^2(1/r)
\end{split}\end{equation*}
for all $x \in E$. We express the left hand side as the sum of integrals over the sets $E_m$ and proceed to estimate the $m$th term in the sum:
$$
\int_{E_m} \phi_{R,\theta}^{\tau,n}(||a(x,y)||_\infty) \log_+(||a(x,y)||_\infty/R^\theta) \, d\mu y.
$$
To do this, we bound the kernel from above as follows:
$$
\phi_{R,\theta}^{\tau,n}(u) \le R^{(1-\theta) \tau} \left( \frac{R^\theta}{u} \right)^n.
$$
This gives
\begin{equation*}\begin{split}
&\int_{E_m} \phi_{R,\theta}^{\tau,n}(||a(x,y)||_\infty) \log_+(||a(x,y)||_\infty/R^\theta) \, d\mu y \\
& \quad \le \int_{E_m} R^{(1-\theta)\tau} \left( \frac{R^{\theta}}{||a(x,y)||_\infty} \right)^n  \log_+\left( \frac{||a(x,y)||_\infty}{R^\theta} \right) \, d\mu y.
\end{split}\end{equation*}
Using monotonicity of $v \mapsto v^{-n} \log_+(v)$ and the estimate in \eqref{eq:a-lower-bound} we proceed:
\begin{equation*}\begin{split}
&\int_{E_m} \phi_{R,\theta}^{\tau,n}(||a(x,y)||_\infty) \log_+(||a(x,y)||_\infty/R^\theta) \, d\mu y \\
& \quad \le R^{(1-\theta)\tau} \left( \frac{R^{\theta}}{r^\gamma \mu(8Q_0^m)^{1/n}} \right)^n  \log_+\left( \frac{r^\gamma \mu(8Q_0^m)^{1/n}}{R^\theta} \right) \, \mu(8Q_0^m) \\
& \quad \lesssim r^{(1-\theta)\sigma + n\theta\sigma/\tau - \gamma n} \log_+(r^{\gamma-\theta \sigma/\tau} \mu(8Q_0^m)^{1/n}) \\
& \quad \lesssim r^{(1-\theta)\sigma + n\theta\sigma/\tau - \gamma n} \log_+(2^{-m} r^{\gamma-\theta \sigma/\tau - \theta + \theta \sigma/n}) \lesssim r^\sigma \log(\tfrac{2^{-m}}{r^\theta})
\end{split}\end{equation*}
where we used \eqref{eq:upper-bound3} and the definition of $\gamma = \gamma_p(\sigma,\theta)$, see \eqref{eq:gamma-p-sigma-theta}. 

Summing over relevant values of $m$ yields
\begin{equation}\begin{split}\label{eq:bound2}
&\int_{\{|x-y|\ge r\}} \phi_{R,\theta}^{\tau,n}(||a(x,y)||_\infty) \log_+(||a(x,y)||_\infty/R) \, d\mu y \\
&\qquad \lesssim r^\sigma \log^2(1/r) \lesssim R^\tau \log^2(1/R),
\end{split}\end{equation}
where we used the fact that $\sum_{m:2^{-m} \gtrsim r} \bigl( m + \log(1/r) \bigr) \le \log^2(1/r)$. 

Lemma \ref{lem:technical-lemma1} now follows by combining  \eqref{eq:bound0}, \eqref{eq:technical-estimate}, \eqref{eq:bound1}, and \eqref{eq:bound2}. In particular, \eqref{eq:capacity-lower-bound-for-E-0-at-scale-r} holds with constant $\kappa$ depending only on $r$. Recall that $E_0$ was chosen to be of full measure with respect to the capacitary measure $\mu$, and that this measure depended on the scale $r$. However, by monotonicity of capacity we conclude that \eqref{eq:capacity-lower-bound-for-E-0-at-scale-r} holds with the same constant $\kappa$ but with $E_0$ replaced by the original set $E$.

\smallskip

To complete the proof we leverage the preceding argument by aggregating over a countable dense set of intermediate dimension parameters $\theta_j$ and the relevant scales $r_{jk}>0$ associated to a given lower bound for ${\overline\dim}_{\theta_j} E$. Let $\{\theta_j:j \in \N\}$ be a countable dense set of values in $(0,1]$. For each $j$, let $\{\sigma_{j,\ell}\}$ be a sequence with $\sigma_{j,\ell} < s_j:= {\overline\dim}_{\theta_j} E$ and $\sigma_{j,\ell} \nearrow s_j$ as $ \ell \to \infty$. Furthermore, let $\{r_{jk}:k \in \N\}$ be a sequence of scales so that
$$
C^{\sigma_{j,\ell},n}_{r_{jk},\theta_j}(E) \ge r_{jk}^{-\sigma_{j,\ell}}.
$$
For each choice of $j$, $k$ and $\ell$ we carry out the above argument, generating a random family of functions $f_{\xi_{jk\ell},j,k,\ell}$ so that the inequality
$$
C_{R_{jk\ell},\theta_j}^{\tau_{j,\ell},n}(f_{\xi_{jk\ell},j,k,\ell}(E)) \ge \eps \kappa^{-1} R_{jk\ell}^{-\tau_{j,\ell}} \log^{-2}(\tfrac1{R_{jk\ell}})
$$
holds for any $\eps>0$ with probability at least $1 - c(n) \eps \log^{-2}(1/R_{jk\ell})$ in $\xi$, for some $c(n)<1$. Here we set $\tau_{j,\ell} = \tau_p(\sigma_{j,\ell})$ and $R_{jk\ell} = r_{jk}^{\sigma_{j,\ell}/\tau_{j,\ell}}$.

We sum $f_{\xi_{jk\ell},j,k,\ell}$ with appropriate normalization coefficients as in \eqref{eq:f-xi-old}. More precisely, let
\begin{equation}\label{eq:f-xi-new}
f_\xi := \sum_{j=1}^\infty \frac1{j^2} \sum_{k=1}^\infty \frac1{k^2} \sum_{\ell=1}^\infty \frac1{\ell^2} \frac1{\log(1/r_{jk})} f_{\xi_{jk\ell},j,k,\ell}.
\end{equation}
We assume as before that $r_{jk} < e^{-k^2}$. Using independence of the full suite of random variables $\xi := (\xi_{jk\ell})$, we argue that the function $f_\xi$ so defined satisfies 
$$
C_{R_{jk}\ell,\theta_j}^{\tau_{j,\ell},n}(f_{\xi}(E)) \gtrsim \eps R_{jk\ell}^{-\tau_{j,\ell}} \log^{-2}(\tfrac1{R_{jk\ell}})
$$
for all $j$, $k$ and $\ell$, with probability $\ge 1 - c(n) \eps$. It follows that 
$$
{\overline\dim}_{\theta_j} f_\xi(E)  = {\overline\dim}_{\theta_j}^n(f_\xi(E)) \ge \tau_p(\sigma_{j,\ell})
$$ 
for all $j$ and $\ell$, with probability $\ge 1 - c(n) \eps$. Finally, we let $\sigma_{j,\ell} \nearrow s_j = {\overline\dim}_{\theta_j} E$ and appeal to the previous conclusion with $\eps_\ell = 2^{-\ell-1}$. Since $c(n)<1$ we conclude that with positive probability there exists $\xi$ so that 
$$
{\overline\dim}_{\theta_j} f_\xi(E) \ge \tau_p(s_j) \qquad \forall \, j \in \N.
$$
Continuity of the function $\theta \mapsto {\overline\dim}_\theta E$ implies that this conclusion upgrades to
$$
{\overline\dim}_\theta f_\xi(E) \ge \tau_p({\overline\dim}_\theta E) \qquad \forall \, 0<\theta \le 1,
$$
and the proof is complete.
\end{proof}

\section{Open questions and further comments}\label{sec:open-questions}

In this final section, we collect several open questions motivated by the results of this paper.

\smallskip

\begin{ques}\label{q:continuity-of-conformal-Assouad-spectrum}
For which metric spaces $(X,d)$ does it hold true that
$$
\lim_{\theta \to 1} C\dim_A^\theta X = C\dim_{qA} X?
$$
\end{ques}

Recall that $C\dim_A^\theta X$ denotes the infimum of the values of $\dim_A^\theta Y$ over all metric spaces $Y$ which are quasisymmetrically equivalent to $X$, while $C\dim_{qA} X$ denotes the infimum of the values of $\dim_{qA} Y = \lim_{\theta \to 1} \dim_A^\theta Y$ over the same class of spaces $Y$. Justifying the interchange of infimum and limit in this setting is a subtle question. For instance, even if we know that $C\dim_A^\theta X = 0$ for all $0<\theta<1$ (see e.g.\ the conclusion in Theorem \ref{th:applic-1}) it is unclear whether we can deduce that $C\dim_{qA} X = 0$. While quasisymmetric mappings have strong convergence properties, such convergence theorems necessarily require a uniform bound on the quasisymmetric distortion function and in practice, there is no a priori reason why a sequence of quasisymmetric mappings $f_\theta:X \to Y_\theta$, $0<\theta<1$, for which $\dim_A^\theta(f_\theta(X)) < \eps$ for some fixed $\eps$ (independent of $\theta$) should satisfy such a uniform bound.

\smallskip

As mentioned in Example \ref{ex:binder-hakobyan-li}, full Bedford--McMullen carpets with uniform fibres are minimal for conformal Hausdorff dimension, see \cite{BHL}.

\begin{ques}\label{q:bhl-consequence}
Is every full Bedford--McMullen carpet minimal for conformal Hausdorff dimension?
\end{ques}

The next two questions relate to Example \ref{ex:percolation-example}. Recall that the Hausdorff and box-counting dimensions of samples of Mandelbrot percolation agree a.s., and such samples are not minimal for conformal Hausdorff dimension a.s.

\begin{ques}\label{q:percolation-1}
Does the strict inequality $C\overline\dim_B F < \overline\dim_B F$ hold a.s.\ for samples $F$ of the Mandelbrot percolation process? In particular, is there a constant a.s.\ value for $C\overline\dim_B F$?
\end{ques}

\begin{ques}\label{q:percolation-2}
Does the equality $C\overline\dim_B F = C\dim_H F$ hold a.s.\ for samples $F$ of Mandelbrot percolation?
\end{ques}

Similar questions can be posed about other notions of dimension (e.g.\ intermediate dimension, Assouad spectrum). As observed in \cite{brt:phi-assouad}, the dimensional size of Mandelbrot percolation samples is naturally quantified using the more general notion of {\it $\phi$-Assouad dimension} $\dim_A^\phi$, and it may be of interest to explore the mapping-theoretic properties of $\phi$-Assouad dimension.

\end{document}